\documentclass[10pt]{amsart}
\usepackage{amsfonts}
\usepackage{amssymb}
\usepackage{amsmath,amsthm}
\usepackage{amscd}
\usepackage[toc,page]{appendix}
\usepackage[english]{babel}
\usepackage{hyperref}
\usepackage{amsmath}
\usepackage{amssymb,amsthm,amscd}
\usepackage[textwidth=17.5cm,textheight=20cm]{geometry}
\usepackage{mathrsfs}
\usepackage[thinlines]{easybmat}

\newtheorem{definition}{Definition}

\newtheorem{theorem}{Theorem}

\newtheorem{remark}{\textbf{Remark}}

\newcommand{\R}{\mathbb{R}}

\newcommand{\N}{\mathbb{N}}

\newcommand{\A}{\mathscr A}

\renewcommand{\d}{\mathrm{d}}

\newcommand{\W}{\mathscr W}

\newtheorem{pro}{Proposition}
\newtheorem{cor}{Corollary}
\newcommand{\noi}{\noindent}
\begin{document}
\author{Gerardo Ariznabarreta}\address{Departamento de F\'{i}sica Te\'{o}rica II, M\'{e}todos y modelos matem\'{a}ticos, Facultad de F\'{\i}sicas, Universidad Complutense, 28040 -- Madrid, Spain}
\author{Manuel Ma\~{n}as}\address{Departamento de F\'{i}sica Te\'{o}rica II, M\'{e}todos y modelos matem\'{a}ticos, Facultad de F\'{\i}sicas, Universidad Complutense, 28040 -- Madrid, Spain}
\author{Piergiulio Tempesta}\address{Departamento de F\'{i}sica Te\'{o}rica II, M\'{e}todos y modelos matem\'{a}ticos, Facultad de F\'{\i}sicas, Universidad Complutense, 28040 -- Madrid, Spain and Instituto de Ciencias Matem\'aticas, C/ Nicol\'as Cabrera, No 13--15, 28049 Madrid, Spain}
%\date{manuel.manas@fis.ucm.es}

\title{Generalized Sobolev orthogonal polynomials, matrix moment problems and integrable systems}
\date{December 15, 2016}

\maketitle
\begin{abstract}
We introduce a large class of Sobolev bi-orthogonal polynomial sequences arising from a $LU$-factorizable moment matrix and
associated with a suitable measure matrix that characterizes the Sobolev bilinear form.
A theory of deformations of Sobolev bilinear forms is also proposed. We consider both polynomial deformations and a class of
transformations related to the action of linear operators on the
entries of a given bilinear form. Transformation formulae among new and old polynomial sequences are determined.

Finally, integrable hierarchies of evolution equations arising from the factorization of a time deformation of the moment matrix are presented.
\end{abstract}
\tableofcontents

\section{Introduction}

\subsection{Historical background and motivation}
In the last decades, the study of Sobolev orthogonal polynomials has become a field of increasing interest both in Applied Mathematics and Mathematical Physics. The purpose of this article is to extend the notion of Sobolev orthogonality by introducing a theoretical framework allowing to define a new, large class of Sobolev bi-orthogonal polynomial sequences (SBPS).

In order to situate our contribution in the context of the existing literature, we start by mentioning some of the most relevant results of the theory established till now. We focus here only on some aspects of special interest for our research. For a nice review of modern results, historical background and an updated bibliography, the reader is referred to \cite{MXreview}, \cite{Meijer}.

Sobolev orthogonal polynomials were introduced in 1962 by Althammer \cite{Alt}. He proposed the idea of defining a class of polynomials orthogonal with respect to a \textit{deformation} of the Legendre inner product, of the form
\begin{equation}
\langle f,g \rangle_{A}= \int_{-1}^{1} f(x)g(x)d x + \lambda \int_{-1}^{1} f'(x) g'(x) dx \ .
\end{equation}

The polynomials arising from this inner product are called nowadays the Sobolev-Legendre polynomials.

Perhaps the most relevant of the early contributions to the theory came in the 70's with the works \cite{Schafke}, \cite{SW}. Indeed, Sch\"afke and Wolf proposed the following family of inner products

\begin{equation}
\langle f,g \rangle_{SW}= \sum_{j,k=0}^{\infty}\int_{a}^{b} f^{(j)}(x)g^{(k)}(x) v_{j,k}(x)w(x)dx \ ,
\end{equation}
where the weight $w$ and the associated integration interval is intended to be one of the three classical cases of Hermite, Laguerre and Jacobi; also, $v_{j,k}(x)$ are suitable polynomials, symmetric in $j,k$.

Starting from this polynomial deformation of classical measures, and specializing conveniently the functions $v_{j,k}$, Sch\"afke and Wolf were able to define eight families of new Sobolev orthogonal polynomials, and extended all previously known results on Sobolev orthogonal polynomials.

Since the last decade of the previous century there was a resurgence of interest in the field of Sobolev orthogonality,
starting with the seminal paper \cite{Coherent}. In this work, the notion of \textit{coherent pairs}, a fundamental idea which has
triggered many new developments,  was introduced. Let $\{d\mu_{1}, d \mu_{2}\}$ be a pair of Borel measures on the real line
with finite moments. To this pair we associate the inner product
$\langle f,g \rangle_{(\mu_{1}, \mu_{2})}= \int_{a}^{b} f(x)g(x)d \mu_{1} + \lambda \int_{a}^{b} f'(x) g'(x) d\mu_{2}$, with $a,b\in\mathbb{R}$.
Essentially, the pair of measures $\{d\mu_{1}, d \mu_{2}\}$ is said to be a coherent pair whenever the sequence of polynomials associated
with $d\mu_{2}$ can be related in a specific way with the first derivatives of the polynomials of the sequence associated with $d\mu_{1}$.
In \cite{Petronilho} a classification of coherent pairs was given when one of the two involved measures is a  classical one (Hermite, Laguerre, Jacobi or Bessel). In \cite{Meijer2} it was proven that in order for $\{d\mu_{1}, d \mu_{2}\}$ to form a coherent pair, at least one of the two measures has to be classical. This result shows that the classification given in \cite{Petronilho} is actually a complete one.

Besides, a huge amount of results concerning many analytic and algebraic aspects of the theory has been obtained in the last twenty years,
including the relation with differential operators \cite{Koekoek}, \cite{DuranIglesia}, the asymptotic behaviour and the study of zeros of
Sobolev polynomials \cite{MarcMor}, etc.

\subsection{Main results}
In this paper, we generalize significantly the construction of Sch\"afke and Wolf by introducing a large class of not necessarily symmetric
Sobolev bilinear forms $(*,*)_{\W}$. These bilinear forms are defined by means of a matrix of measures $\W$, representing one of the crucial mathematical structures of the present paper. To each measure matrix $\W$, or equivalently to the corresponding bilinear form, we can naturally associate a moment matrix $G_{\W}$.
In our analysis, we shall focus on the class of moment matrices that admit an \textit{$LU$-factorization}. Indeed, for this class one
can construct Sobolev bi-orthogonal polynomial sequences (SBPS).
We shall prove that many algebraic techniques related to the $LU$-factorization, that proved to be very useful in order to obtain algebraic
properties of the standard orthogonal polynomial sequences (OPS) can be extended naturally to our Sobolev setting.

%The motivation of our work is to generalize the notion of Sobolev orthogonality by allowing a very general Sobolev bilinear form, defined in terms of a moment matrix. If this matrix admits an $LU$-factorization, we shall prove that one can extend to the Sobolev setting many algebraic techniques that proved to be very useful in dealing with standard orthogonal polynomial sequences (OPS).

A crucial notion proposed in this paper is that of \textit{additive perturbations} of a measure matrix $\W$ in the Sobolev context. Precisely, we shall study under which conditions, by performing an additive matrix perturbation of $\W$, one can still produce families of SBPS. This approach turns out to be particularly fruitful. Indeed, one can describe on the same footing, and generalize widely, important constructions as the coherent pairs and the standard approach of discrete Sobolev bilinear forms.
Concerning the first aspect, we wish to point out that not only a standard coherent pair can be studied from the perspective of perturbation theory, but it also can be generalized, in terms of the new notion of \textit{$m\times m$ block coherent pair}. The SBPS arising from both standard and block coherent pairs are studied.

When the entries of the measure matrix $\W$ are allowed to depend on $\delta$ distributions, we can encompass in our approach the well-known case
of discrete Sobolev orthogonality. Once we split a Sobolev bilinear function into a continuous part, involving those entries of $\W$
having a continuous support, and a discrete one, involving those having a discrete support ($\delta$ distributions)
\footnote{Some authors call type I Sobolev products those involving continuous supports only and
type II and III those involving  a continuous support while the rest are finite subsets}, we can interpret the
discrete part as an \textit{additive discrete perturbation} of its continuous part. This leads to an interesting
characterization of the SBPS associated to the original measure matrix in terms of quasi-determinantal formulae, involving only the continuous
part of the bilinear function.\\

A related aspect is the possibility of classifying measure matrices in terms of \textit{equivalence classes}: To each class it
belongs a set of measure matrices giving rise to the same moment matrix, and therefore to the same SBPS.
Indeed, the correspondence between measure matrices and moment matrices is not one to one. Therefore, different
Sobolev bilinear forms may lead to the same SBPS. An interesting case arises when inside the same equivalence class
possibly Sobolev and non Sobolev-type measure matrices are present. All this is not surprising, taking into account that the integration by parts procedure (at least in a distributional sense)
comes into play, allowing to define elementary operations leaving a measure matrix into the same class.\\

Due to the relevance of measure matrices in our approach, a natural problem is to develop a deformation theory for these matrices
which allows us to relate the corresponding deformed and non deformed SBPS.

Special attention will be devoted to certain classes of transformations well known in the literature on orthogonal polynomials:
Christoffel's and Geronimo's transformations. The first ones were introduced in 1858 by Christoffel \cite{Christoffel}, and amount to
a polynomial deformation of a given classical measure. Precisely, the standard Christoffel formulae establish connections among families of orthogonal
polynomials, allowing to express a polynomial of a family just in terms of a constant number of polynomials of the other family.
We generalize this approach by introducing Christoffel-Sobolev transformations. These involve a matrix polynomial deformation of the Sobolev
measure matrix $\W$, which can be implemented by means of a right or left action of the deformation on the matrix $\W$.
Once suitable resolvents and their adjoints are defined, then it is possible to connect deformed and non-deformed Sobolev polynomial sequences
(and related Christoffel-Darboux kernels). In addition, quasi-determinantal expressions for the deformed polynomial sequences in terms of
the original ones are obtained.

The second class of deformations we shall generalize is that of Geronimus, which was introduced in \cite{Geronimus} (see also \cite{Golinskii}).
We propose, in our context, the notion of Geronimus-Sobolev transformation of a measure matrix. This very general transformation amounts
to a right or left multiplication of the initial measure matrix $\W$  by the inverse of a matrix polynomial, extended by the addition of a discrete deformation.
Once again, one can obtain explicit formulae connecting deformed and non-deformed polynomials (and Christoffel-Darboux kernels) that are
expressed in compact quasi-determinantal expressions.
\\

The previous cases of polynomial and inverse polynomial-type deformations of the measure matrix are of special interest, but do not exhaust the
range of possible transformations we can perform over $\W$. Another novel aspect of the present work is that, indeed, we broaden the family
of possible deformations by admitting much more general deformations. They are expressed in terms of \textit{linear differential operators}
with polynomial coefficients, this is, operators of the form
%\begin{equation}
$\boldsymbol{L}=\sum_k p_{k}(x)\frac{\d^k}{\d x^k}$, \label{LDE}
%\end{equation}
acting on the entries of the original bilinear form.  Due to its generality, the theory
of these operator deformations appears to be extremely rich (see also \cite{OPDO}). In this paper, we focused on several aspects
which look of particular interest. Given a couple of linear differential operators of the form given above, it is possible to define a
new class of Sobolev bilinear forms, which under certain technical conditions still possesses an associated moment
matrix $G_{\W}$ which is $LU$-factorizable and consequently, give a SBPS. \\

%\begin{remark}
We mention that an article which in some sense can be related to section 6 of the present one is Ref. \cite{Car3Tr}. In that work,
the authors consider polynomial perturbations of a generic sesquilinear form. The methods used there are specially suited to
polynomial perturbations of a matrix bivariate functional, and therefore include matrix Sobolev bilinear forms.
The present paper focuses explicitly on the Sobolev scenario, from a different point of view.  The fraction of the results of \cite{Car3Tr}
concerning polynomial deformations of sesquilinear forms, in our opinion cannot be translated into our context in a simple or useful way.
For that reason, we have introduced Sections 6.3, 6.4 and 6.5, where polynomial perturbations are treated expressly for the Sobolev (scalar) setting.
It must be underlined that the deformations of the bilinear forms that the present paper considers (Section 7) are certainly more general since
linear differential operator transformations are allowed instead of just polynomial ones.\\
%\end{remark}

The paper is organized as follows.
In Section 2, we introduce the main notions of our analysis: Measure matrices, Sobolev generalized bilinear forms,
moment matrices and the $LU$-factorization is studied.
In Section 3, we construct the family of Sobolev bi-orthogonal polynomial sequences
arising from $LU$-factorizable moment matrices together with the introduction of their associated second kind functions. Chistoffel-Darboux
and Cauchy kernels associated with these sequences are also defined.
In Section 4, we propose a theory of additive perturbations of measure matrices, which allows us to treat on the same footing
coherent pairs (and a generalization of these) and discrete bilinear forms of Sobolev type.
The crucial idea of equivalence classes of measure matrices is introduced
and developed in Section 5. This idea proves to be of special interest when classical measures are involved in the bilinear form; some attention is devoted to these measures
in order to generalize some known results. A polynomial deformation theory of the measure matrices is proposed in Section 6, which includes the
important case of linear spectral or Darboux-Sobolev transformations. Section 7 is devoted to an extension of our theory of deformations of measure matrices to
the case of linear differential operators. The study of the relation of the present approach with integrable hierarchies of Toda type is presented in the final Appendix.

\section{Algebraic preliminaries}
\subsection{A generalized Sobolev bilinear form}
We shall first introduce the main definitions necessary for our approach.

%We are about to make a couple of definitions in order to specify the main objets we will be working with together with the mathematical structure behind the results we want to express.

\begin{definition}
A measure matrix of order $\mathcal{N}$, with $\mathcal{N}\in\mathbb{N}$ is a matrix $\W$ whose entries $\{\d \mu_{i,j}(x)\}_{i,j}$ are Borel measures
and  $\d \mu_{i,j}=0$ $\forall i,j > \mathcal{N}$:
\begin{align*}
 \W(x)&:=\begin{pmatrix}
   \d \mu_{0,0}          & \d \mu_{0,1}           & \dots &\d \mu_{0,\mathcal{N}}           & 0       & \dots \\
   \d \mu_{1,0}          & \d \mu_{1,1}           &\dots  &\d \mu_{1,\mathcal{N}}           & 0       & \dots \\
   \vdots                &        \vdots          & \ddots&    \vdots                       &\vdots   &       \\
  \d \mu_{\mathcal{N},0} & \d \mu_{\mathcal{N},1} & \dots & \d \mu_{\mathcal{N},\mathcal{N}}& 0       & \dots \\
  0                      &      0                 &       &          0                      & 0       & \dots \\
  \vdots                 &     \vdots             &       &          \vdots                 &         & \ddots
\end{pmatrix} &
\d \mu_{i,j}:\,\,\,&\Omega_{i,j}\subseteq \R \longrightarrow \R
\end{align*}
\end{definition}
\begin{definition}
The bilinear form $(*,*;\W): \R[x]\times \R[x]\longrightarrow \R$ associated with $\W$ is defined to be
\begin{align}\label{bilSob}
(x^i,x^j;\W)&:=\sum_{n,r=0}^{\mathcal{N}}\left\langle \frac{\d^{n} x^i}{\d x^n},\frac{\d^{r} x^j}{\d x^r}\right\rangle_{n,r}  &\mbox{where} & &
\left\langle \frac{\d^{n} x^i}{\d x^n},\frac{\d^{r} x^j}{\d x^r}\right\rangle_{n,r}&:=\int_{\Omega_{n,r}}\frac{\d^{n} x^i}{\d x^n}\frac{\d^{r} x^j}{\d x^r} \d \mu_{n,r}(x)
\end{align}
where we assume the condition $|(x^i,x^j;\W)|< \infty $ $\forall i,j\in \N$.
\end{definition}
It is important to notice that the case $\mathcal{N}\longrightarrow \infty$ is also allowed since for given $i,j\in \N$ the
bilinear form $(x^i,x^j;\W)$ will always involve a finite number of terms only.

We wish to extend the domain of the bilinear form \eqref{bilSob} to a more general function space containing $\R[x]$ as a subspace.
%To that aim, we firstly define
\begin{definition}
Let $\Omega:=\bigcup_{i,j=0}^{\mathcal{N}} \Omega_{i,j}$. The function space $\A^{\mathcal{N}}_{\W}(\Omega)$ is defined as
 \begin{align*}
  \A^{\mathcal{N}}_{\W}(\Omega)&:=\left\{f(x)\in C^{\mathcal{N}}(\Omega) \mbox{ such that }
  |(f,f;\W)|:=\left\lvert\sum_{n,r=0}^{\mathcal{N}}\left\langle \frac{\d^n f }{\d x}, \frac{\d^r f }{\d x}\right\rangle_{n,r}\right\lvert<\infty \right\} \ ,
 \end{align*}
 where $C^{k}(\Omega)$ denotes the space of functions possessing $k$ continuous derivatives in $\Omega$.
\end{definition}
\par
We wish to endow the space $\A^{\mathcal{N}}_{\W}(\Omega)$ with a structure of normed vector space, with norm given by $||f||^{2}:=(f,f;W)$. Therefore, jointly with the existence of finite moments, we need also to require positive definiteness:
%\begin{itemize}
 %\item Finite moments $|(x^i,x^j;\W)|< \infty $ $\forall i,j\in \N$.
$\forall f\neq 0, (f,f;W)>0$.
%\end{itemize}
Hereafter we shall tacitly assume that this condition is satisfied.

Observe that, since every continuous bilinear function is bounded, we have that whenever $f(x),g(x) \in \A^{\mathcal{N}}_{\W}(\Omega)$ the pairing $(f,g;\W)$ satisfies $|(f,g;\W)|\leq C ||f|| ||g||$, and therefore
is finite. Consequently, we can introduce the notion of Sobolev bilinear function.
\begin{definition} \label{def3}
For every $f(x),g(x)\in \A^{\mathcal{N}}_{\W}(\Omega)$ we shall call the
non degenerate positive definite bilinear function $(*,*;\W):\A^{\mathcal{N}}_{\W}(\Omega) \times \A^{\mathcal{N}}_{\W}(\Omega) \longrightarrow \R$ defined by
\begin{align} \label{SBF}
(f,h;\W)&:=\sum_{n,r=0}^{\mathcal{N}}\langle f^{(n)},h^{(r)}\rangle_{n,r}  &\mbox{with} & &
f^{(n)}&:=\frac{\d^{n}f(x)}{\d x^n}
\end{align}
the Sobolev bilinear function associated with the measure matrix $\W$.
\end{definition}
%\begin{remark}
%On the one hand if we impose $(f,f;\W)> 0$ $\forall f\neq0 \in \A^{\mathcal{N}}(\Omega)$ we obtain a positive-definite bilinear function
%in which case $\A^{\mathcal{N}}(\Omega)$ will actually be a normed function space. On the other,
%the case $(f,f;\W)\neq 0$ $\forall f\neq0 \in \A^{\mathcal{N}}(\Omega)$ leads to a quasi-definite bilinear function.
%\end{remark}
%NO SE SI QUEDARME UNICAMENTE CON EL POSITIVO DEFINIDO Y LAS NORMAS...YA QUE PARA ELLOS ES EL BILINEAR CONTINUOUS IMPLICA BOUNDED...
%One may feel bothered by the apparent infinite sum in the definition, but the truth is
%that as we previosuly pointed out, in most of the practical cases only some (a finite number) of the
%$\omega_{n,r}$ are different from zero so that the definition always makes sense.

%\textbf{Nota: Actually, the definition makes sense in the infinite-dimensional case only if the series converges with respect to some metric}.

%\begin{remark}
Several comments are in order.
\begin{itemize}
\item
Definition \ref{def3} includes as a particular case the standard inner product, with no derivatives involved, which corresponds
to the choice $\mathcal{N}=0$, namely $\d \mu_{i,j}=0 \,\,\,\forall i,j>0$.

\item
Choosing a non symmetric $\W$ leads us to extend naturally the concept of orthogonality to that of bi-orthogonality. Indeed, one could have $(f,h;\W)=0$ while $(h,f;\W)\neq0$.
This situation also occurs in the study of standard matrix orthogonality with respect to a non symmetric matrix measure
(see for example \cite{MOPUC}) or when dealing with scalar bivariate linear functionals (see for example \cite{Car3Tr} ) .

\item If $\W=\W^{\top}$ we obtain a positive definite symmetric bilinear form $(f,h;\W)=(h,f;\W)$ which allows us to define
a standard inner product. Observe that the literature on the subject specially focuses on diagonal $\W$, for which obviously $\W=\W^{\top}$.

\end{itemize}
%\end{remark}
\begin{remark}
Unlike the  point of view adopted in \cite{Car3Tr}, based on the bivariate linear functional setting, in this paper we have preferred to work with an integral representation of our bilinear form. This representation exists as a direct consequence of the Riesz-Markov-Kakutani theorem \cite{Kakutani}. The reason for this choice is the fact that we wish to develop a theory explicitly related with measure matrices.

%as a reward, a more tangible theory is obtained, where the measure matrix plays a starring role.
%CONECTAR AMBAS CON ALGUN TEOREMA DE REPRESENTACION
\end{remark}

\subsection{The moment matrix}
%Given a Borel measure, one approach in order to the study its associated OPS is throughout the aid of the corresponding moment matrix (Hankel matrix).
%This is, take the Borel measure $\d \mu(x)$, compute its moments $m_i=\int x^i \d \mu(x)$, build the moment matrix $g_{i,j}:=m_{i+j}$ and proceed with the standard techniques. A related cassical problem is to start off with a given sequence $\{m_i\}_i$ and study  whether such a sequence is a moment sequence associated to a $\d \mu$ or not. In the affirmative case the uniqueness of the solution is studied too. This problem is known as the (Hausdorff, Hamburger, Stieltjes) moment problem (depending on $\Omega$) and it is called  definite if it has a solution and determinate if it is unique.\\

%We would like to keep with the first approach for the Sobolev setting.
Our approach to SBPS requires the definition of a suitable moment matrix. Notice that the Hankel-type form of the moment matrix, usual in the non Sobolev context, is expected to be lost or generalized;
according to \cite{HtransF}, the generalized form can be called Hankel--Sobolev matrices.
The associated moment problem will involve more than just one sequence of
integers (for a study of a diagonal $\W$ see \cite{definitedeterminate},\cite{MonSob});
of course, a propaedeutic problem will be to establish under which conditions a matrix can play the role of a suitable Sobolev moment matrix.
Instead, we prefer to proceed in a somewhat different way: we construct a moment matrix suitable for the Sobolev bilinear
function \eqref{SBF}. We start by settling some notation.  \\

%\begin{definition}
Given two non negative integers $m,n$ we will denote by $(m)^{n}$ and $(m)_{n}$ the rising and lower factorial polynomials respectively, i.e.
 \begin{align*}
  (m)^{n}&:=m(m+1)(m+2)\dots(m+(n-1)) \\
  (m)_{n}&:=\begin{cases}
             m(m-1)(m-2)\dots(m-(n-1)) & \qquad n<m\\
             0 & \qquad n\geq m
            \end{cases}
 \end{align*}
 \begin{align*}
  (m)^{0}=(m)_{0}&:=1 &
  (m)^{1}=(m)_{1}&:=m
 \end{align*}
%\end{definition}

%The following vector of monomials and its derivatives will be of great importance in the sequel.
\begin{definition}
We introduce the vectors
\begin{align*}
\chi(x)&:=\begin{pmatrix}
          1 \\
          x \\
          x^2 \\
          x^3 \\
          \vdots \\
          x^k \\
          \vdots \\
         \end{pmatrix} &
         \chi'(x)&:=\begin{pmatrix}
          0 \\
          1 \\
          2x \\
          3x^2 \\
          \vdots \\
          kx^{k-1} \\
          \vdots \\
         \end{pmatrix} &
         \chi''(x)&:=\begin{pmatrix}
          0 \\
          0 \\
          2 \\
          (3)(2)x \\
          \vdots \\
          k(k-1)x^{k-2} \\
          \vdots \\
         \end{pmatrix} &
         & \dots &
         \chi^{(n)}(x)&:=\begin{pmatrix}
          0 \\
          0 \\
          \vdots \\
          (n)_n \\
          \vdots \\
          (k)_nx^{k-n} \\
          \vdots \\
         \end{pmatrix} &
         & \dots & \\
\end{align*}
and the lower semi infinite matrix
%\begin{definition}
\begin{align*}
 \boldsymbol{\chi}(x)& :=\begin{pmatrix}
                        \chi(x) & \chi'(x) & \chi''(x) & \dots & \chi^{(k)}(x) & \dots
                       \end{pmatrix}
\end{align*}

We also define the auxiliary vector
\[\chi^*(x):=\frac{1}{x}\chi \left(\frac{1}{x}\right). \]  \\
\end{definition}
The previous definition allows to deal with polynomials in a simple way. Let $p(x)\in\R[x]$ be a polynomial of degree $k$, i.e. $p(x)=\sum_{l} p_{l} x^{l}$ with $p_l=0$ $\forall l>k$. Let us denote by
$\boldsymbol{p}:=(p_0,p_1,p_2,\dots)$. Consequently, we have
\[
p(x)=\boldsymbol{p}\chi(x)  \hspace{2mm} \text{and} \hspace{2mm} p^{(k)}(x)=\boldsymbol{p}\chi^{(k)}(x).
\]
%The previously defined $\chi^{(k)}(x)$ can be used to build a semi infinite matrix, out of which the Sobolev moment matrix can be constructed. To this aim,

For each $m\in \mathbb{N}$ (the directed set of natural numbers), we consider the
ring of matrices $\mathbb{M}_m:=\R^{m\times m}$, and its direct
limit $\mathbb{M}_{\infty}:=\lim_{m\to\infty} \mathbb M_m$, i.e. the ring of semi-infinite matrices.
We will denote by $G_\infty$ the group of invertible semi-infinite matrices of $\mathbb{M}_\infty$.
A subgroup of $G_\infty$ is $\mathscr L$, that of lower triangular matrices with the identity matrix along its main diagonal.
Diagonal matrices will be denoted by
$ \mathscr D=\{M \in \mathbb{M}_{\infty} : d_{i,j}= d_i \cdot \delta_{i,j}\} $.
%************************************************************\\
%************************************************************\\
We will also use the notation $E_{i,j}$ for indicating the matrix canonical basis, this is $(E_{i,j})_{l,m}=\delta_{i,l}\delta_{j,m}$. % \\
%************************************************************\\
%************************************************************\\
%\end{definition}

\begin{definition}
 The Sobolev moment matrix associated to the measure matrix $\W$ is
 \begin{align} \label{GW}
  G_{\W}& :=\left(\chi,\chi^{\top};\W \right)=\int_{\Omega} \boldsymbol{\chi}\,\,\, \W \,\,\, \boldsymbol{\chi}^{\top}  &
 (G_{\W})_{n,p}&:=(x^n,x^p;\W)
\end{align}
 and its truncations will be denoted as
 \begin{align*}
  G_{\W}^{[k]}&:=\begin{pmatrix}
                  (G_{\W})_{0,0} & (G_{\W})_{0,1} & \dots & (G_{\W})_{0,k-1} \\
                  (G_{\W})_{1,0} & (G_{\W})_{1,1} & \dots & (G_{\W})_{1,k-1} \\
                      \vdots     &    \vdots                    \vdots      \\
                  (G_{\W})_{k-1,0} & (G_{\W})_{k-1,1} & \dots & (G_{\W})_{k-1,k-1}
                 \end{pmatrix}=\int_{\Omega} \boldsymbol{\chi}^{[k]}\,\,\, \W^{[k]} \,\,\, \left(\boldsymbol{\chi}^{[k]}\right)^{\top}
 \end{align*}

\end{definition}
\noi By means of the previous notation, the Sobolev bilinear form of two polynomials $p(x),q(x)\in \R[x]$ can be rewritten as
\begin{align*}
(p,q;\W)=\boldsymbol{p}G_{\W}\boldsymbol{q}^{\top}.
\end{align*}
The positive definiteness condition on the bilinear function is equivalent to that of $G_{\W}$, i.e., every principal minor of $G_{\W}$ must be
greater than zero $\det[G_{\W}^{[k]}]>0$ $\forall k=1,2,\dots$
This condition will be discussed in detail later on.\\

Now we rewrite the moment matrix in a slightly different way, that will be more suitable for our purposes. To this aim, we introduce the derivation matrix
$D\in \mathbb{M}_\infty$  defined by
\begin{align*}
 D&:=\begin{pmatrix}
   0 & 0 & 0 & 0 &\dots \\
   1 & 0 & 0 & 0 &\dots \\
   0 & 2 & 0 & 0 &\dots\\
   0 & 0 & 3 & 0 &\dots\\
   0 & 0 & 0 & 4 &\ddots\\
   \vdots & \vdots & \vdots &\vdots
\end{pmatrix}
\end{align*}
 and its powers $D^k$, whose action is $D \chi(x)=\chi'(x)$, $D^k \chi(x)=\chi^{(k)}(x)$.
We also introduce the shift operator
\begin{align*}
 \Lambda&:=\begin{pmatrix}
   0 & 1 & 0 & 0 &\dots \\
   0 & 0 & 1 & 0 &\dots \\
   0 & 0 & 0 & 1 &\dots\\
   \vdots & \vdots & \vdots &\ddots
\end{pmatrix}
\end{align*} whose action on $\chi$ is $\Lambda \chi(x)=x\chi(x)$, and on a polynomial $p(x)$ is
$xp(x)=\boldsymbol{p}\Lambda \chi(x)$.\\
The shift and derivation matrices satisfy for any natural number $n$
\begin{align*}
 \Lambda D^n-D^n \Lambda:=[\Lambda,D^{n}]=nD^{n-1}
\end{align*}
\begin{definition} We introduce the operator
 \begin{align*}
 \boldsymbol{D}&:= \begin{pmatrix}
                        \mathbb{I} & D & D^2 & \dots & D^k & \dots
                       \end{pmatrix} .
 \end{align*}
\end{definition}
\noi It is immediate to verify that
\[ \boldsymbol{D} \chi(x)= \boldsymbol{\chi}(x) .\]
The following result is a direct consequence of the previous discussion.
%And with them one can rewrite the expression that defined the moment matrix
\par
Let us denote by $g_{k,r}$ the standard moment matrix associated to the measure $\d \mu_{k,r}$
(notice that $g_{i,j}=0_{\infty \times \infty}$ is a null matrix when $\d \mu_{i,j}=0$ and this is the case $\forall i,j > \mathcal{N}$).
\begin{pro}
The moment matrix admits the following representation
%\begin{align}\label{G_W}
% G_{\W}& =
% \boldsymbol{D} \begin{pmatrix}
%   g_{0,0} & g_{0,1} & g_{0,2} & g_{0,3} & \dots & g_{0,\mathcal{N}} \\
%  g_{1,0} & g_{1,1} & g_{1,2} & g_{1,3} &\dots  & g_{1,\mathcal{N}} \\
%   g_{2,0} & g_{2,1} & g_{2,2} & g_{2,3} &\dots  & g_{2,\mathcal{N}}\\
%   \vdots       &        \vdots&        \vdots&      \vdots  & \ddots&  \vdots \\
%   g_{\mathcal{N},0} & g_{\mathcal{N},1} & g_{\mathcal{N},2} & g_{\mathcal{N},3}& \dots & g_{\mathcal{N},\mathcal{N}}\\
%\end{pmatrix}
% \boldsymbol{D}^\top=\sum_{k,r=0}^\mathcal{N} D^k g_{\omega_{k,r}} (D^r)^\top
%\end{align}
\begin{align}\label{G_W}
 G_{\W}& =
 \boldsymbol{D} \begin{pmatrix}
   g_{0,0} & g_{0,1} & g_{0,2} & g_{0,3} & \dots \\
   g_{1,0} & g_{1,1} & g_{1,2} & g_{1,3} &\dots  \\
   g_{2,0} & g_{2,1} & g_{2,2} & g_{2,3} &\dots  \\
   \vdots       &        \vdots&        \vdots&      \vdots  & \ddots \\
\end{pmatrix}
 \boldsymbol{D}^\top=\sum_{l,r=0}^{\mathcal{N}} D^l g_{l,r} (D^r)^\top \ ,
\end{align}
with its truncations
\begin{align}\label{G_Wk}
 G_{\W}^{[k]}& =
 \boldsymbol{D}^{[k]} \begin{pmatrix}
   g_{0,0} & g_{0,1} & \dots & g_{0,k-1} \\
   g_{1,0} & g_{1,1} & \dots & g_{1,k-1}  \\
   \vdots  & \vdots  &\ddots & \vdots   \\
  g_{k-1,0}&g_{k-1,1}&\dots  & g_{k-1,k-1}
\end{pmatrix}
 \left(\boldsymbol{D}^{[k]}\right)^\top=\sum^{k-1}_{l,r=0} D^l g_{l,r} (D^r)^\top \ .
\end{align}
\end{pro}
\begin{proof}
Using the previous definitions, for the expression \ref{G_W} we can write
\begin{align*}
G_{\W}=(\chi,\chi^{\top};\W)=\int_{\Omega} \sum_{k}\sum_{r} \chi^{(k)}(x) \d \mu_{k,r} \left(\chi^{(r)}(x)\right)^{\top}=
\int_{\Omega} \sum_{k} \sum_{r}  D^k \chi (x) \d \mu_{k,r} \left(D^r\chi(x)\right)^{\top}=
\sum_{k} \sum_{r} D^k g_{k,r} \left(D^{r}\right)^{\top}
\end{align*}
while relation \ref{G_Wk} follows from the shape of the $D^l$. Since they are lower
\begin{align*}
 \left(\sum_{l,r=0} D^l g_{l,r} (D^r)^\top \right)^{[k]}=
 \sum_{l,r=0} \left(D^l\right)^{[k]} \left(g_{l,r}\right)^{[k]} \left((D^r)^\top\right)^{[k]} \ ,
\end{align*}
also observe that $\left(D^{l}\right)^{[k]}=0$ $\forall l\geq k$.
\end{proof}

\noi This expression is a generalization of the case of a diagonal $\W$, already studied in \cite{definitedeterminate} \cite{MonSob}.

\section{Sobolev bi-orthogonal polynomial sequences}
\subsection{Main definitions and LU factorization}
To introduce the Sobolev bilinear function we have required a positive definiteness condition, which amounts to having
every principal minor of $G_{\W}$ greater than zero. This requirement  (quasi-definiteness would also be a valid choice) is necessary
in order to use the LU factorization techniques of the moment matrix.

In the subsequent considerations, we shall assume that this condition for the minors of the moment matrix holds.
Although in this paper we give some requirements on the set $\{\d \mu_{ij}\}_{i,j}$
that would assure definiteness of the associated moment matrix, a thorough analysis of this problem remains open.

In \cite{definitedeterminate}, a diagonal measure matrix $\W$ (i.e. $\d \mu_{i,j}=0$ $\forall i\neq j$) was considered.  Choosing every $\d \mu_{i,i}:=\d \mu_i$ as a positive definite measure
makes the resulting Sobolev bilinear form a positive symmetric definite one and therefore a proper inner product.

%Let us see that it is easy to agree with this result.
This result can be easily interpreted in our framework. Observe that according to (\ref{G_W}) the moment matrix for the diagonal case is
\[
G_{\W}=g_{0}+Dg_{1}D^{\top}+D^2g_{2}(D^2)^{\top}+D^3g_{3}(D^3)^{\top}+\dots.
\]
If we introduce the matrix $N:=diag\{1,2,3,\dots\} $, with the aid of (\ref{G_Wk}), the truncation $G_{\W}^{[k]}$ reads
\begin{align*}
 (G_{\W})^{[k]}=(g_{0})^{[k]}+\left(\begin{array}{c|c}
                                       0_{1\times 1}  &  0 \\
                                       \hline
                                           0          & (N g_{1} N)^{[k-1]}
                                      \end{array}\right)
                                     +\left(\begin{array}{c|c}
                                       0_{2\times 2}  &  0 \\
                                       \hline
                                           0          & (N^2 g_{2} N^2)^{[k-2]}
                                      \end{array}\right)+\dots
                                     +\left(\begin{array}{c|c}
                                       0_{k-1\times k-1}  &  0 \\
                                       \hline
                                           0          & (N^{k-1} g_{k-1} N^{k-1})^{[1]}
                                      \end{array}\right)
\end{align*}
The condition that  $\d \mu_i$ be positive definite amounts to say that, given any vector $\boldsymbol{v}=(\boldsymbol{v}_0,\boldsymbol{v}_1,\dots,\boldsymbol{v}_{l-1})$, the associated quadratic form
$\boldsymbol{v}(g_{i})^{[l]}\boldsymbol{v}^{\top}$ satisfies $\boldsymbol{v}(g_{i})^{[l]}\boldsymbol{v}^{\top}>0$ $\forall \boldsymbol{v},l $.
Therefore, in the computation of
$\boldsymbol{v}(G_{\W})^{[k]}\boldsymbol{v}^{\top}$ only the sum of positive terms is involved; as a result
$\boldsymbol{v}(G_{\W})^{[k]}\boldsymbol{v}^{\top}>0$ $\forall \boldsymbol{v},k$, ensuring that $G_{\W}$ is positive definite and in turn LU factorizable.\\

%All this introduction tries to justify the existence of LU factorizable Sobolev moment matrices, since these are the kind of moment matrices
%we will be working with. Summarizing, as long as every principal minor of the moment matrix is bigger than (different from) zero we can LU factorize the
%moment matrix and use the techniques that this factorization provides us with.

We shall discuss now in the Sobolev context the main algebraic techniques of the present theory: the LU factorization approach
for the moment matrix and the existence of bi-orthogonal sequences of polynomials.

\begin{definition}
We shall say that the moment matrix $G_{\W}$ admits a LU factorization iff $\det \left(G_{\W}^{[k]}\right)\neq 0$ $\forall k=1,2,\dots$; in such a case
there exist two matrices $S_1,S_2\in \mathscr L$ such that
\begin{align}
 G_{\W}&:=S_1^{-1}H \left(S_2^{-1}\right)^{\top} \label{GWLU} \ ,
\end{align}
where $H:=\delta_{r,k}h_k\in \mathscr D$.
\end{definition}
\begin{definition} \label{biorth}
The monic SBPS associated with the LU-factorized moment matrix $G_{\W}$ \eqref{GWLU} are defined to be
\begin{align}
 P_1(x)&:=S_1\chi(x):=\begin{pmatrix}
                       P_{1,0}(x) \\
                       P_{1,1}(x) \\
                       \vdots \\
                       P_{1,k}(x) \\
                       \vdots
                      \end{pmatrix}\ , &
 P_2(x)&:=S_2\chi(x):=\begin{pmatrix}
                       P_{2,0}(x) \\
                       P_{2,1}(x) \\
                       \vdots \\
                       P_{2,k}(x) \\
                       \vdots
                      \end{pmatrix}\ .
\end{align}
\end{definition}
As a well known consequence of the previous definitions expressing our polynomials in terms of the LU factorization matrices, we can write
the following compact relations.
\begin{pro}
The SBPS can be expressed by means of the following quasi-determinantal formulae
\begin{align} \label{P1kP2k}
P_{1,k}(x)&=\Theta_*\left[\begin{array}{c |c}
	               G_{\W}^{[k]}                 & \begin{matrix}  1 \\ x \\ \vdots\\ x^{k-1} \end{matrix}\\
	\hline
	\begin{matrix}(G_{\W})_{k,0} & (G_{\W})_{k,1} & \dots & (G_{\W})_{k,k-1} \end{matrix} & x^k \end{array}\right] \ ,  \\
P_{2,k}(x)&=\Theta_*\left[\begin{array}{c |c}
	         \left(G_{\W}^{\top}\right)^{[k]}   & \begin{matrix}  1 \\ x \\ \vdots\\ x^{k-1} \end{matrix}\\
	\hline
	\begin{matrix}(G_{\W}^{\top})_{k,0} & (G_{\W}^{\top})_{k,1} & \dots & (G_{\W}^{\top})_{k,k-1}\end{matrix} & x^k \end{array}\right]	\ .
\end{align}
\end{pro}

Notice that the definition ensures that
$deg[P_{\alpha,k}]=k\,\,\,\,\alpha=1,2\,\,\,\,\forall k=0,1,\dots$ while the condition on the minors of $G_{\W}$ guarantees that the definition
always makes sense.

%Notice that the condition on the minors of $G_{\W}$ ensures that $deg[P_{i,k}]=k\,\,\,\,i=1,2\,\,\,\,\forall k=0,1,\dots$.\\
Here we have used the notation $\Theta_*[M]$ to denote the \textit{last quasi-determinant} or \textit{Schur complement} of the matrix in brackets. More precisely, we recall that given
$M=\left(\begin{array}{c|c}
   A & B\\
   \hline
   C & D
   \end{array}\right) \in \mathbb{M}_{(n+m)}$ with $A \in \mathbb{M}_{n}$, $\det \left(A\right) \neq 0$ and $D\in \mathbb{M}_{m}$ its last quasi-determinant
or Schur complement with respect to $A$ is given by
\begin{align*}
  \Theta_*\left[\begin{array}{c|c}
   A & B\\
   \hline
   C & D
   \end{array}\right]&:=SC(M):=M/A:=D-CA^{-1}B
\end{align*}
It is worth observing that the block Gauss factorization of $M$ involves the last quasi-determinant
\begin{align*}
 M=
   \begin{pmatrix}
   \mathbb{I}_n & 0\\
   CA^{-1} & \mathbb{I}_m
   \end{pmatrix}
   \begin{pmatrix}
   A & 0\\
   0 & \Theta_*[M]
   \end{pmatrix}
   \begin{pmatrix}
   \mathbb{I}_n & AB^{-1}\\
   0 & \mathbb{I}_m
   \end{pmatrix}.
\end{align*}
From the previous relation one can immediately deduce that
\begin{align*}
 \det \left(\Theta_*\left[\begin{array}{c|c}
   A & B\\
   \hline
   C & D
   \end{array}\right]\right)= \frac{\det\left( M\right)}{\det \left(A\right)}.
\end{align*}
Therefore, whenever $m=1$, $D$ reduces to a scalar $d$ %(as in Definition \ref{PSKF})
, and the quasi-determinants are a ratio of standard determinants
\begin{align*}
 \Theta_*\left[\begin{array}{c|c}
   A & B\\
   \hline
   C & d
   \end{array}\right]= \frac{\det\left( M\right)}{\det \left(A\right)} \ .
\end{align*}
This indeed is the situation we will deal with. %, so whenever a quasi-determinant appears, the reader should feel free to replace it by the appropiated ratio of determinants.
However, we prefer to use quasi-determinants since the relations we obtain will be ready for further generalizations of the theory
(matrix Sobolev, multivariate Sobolev), where the expressions in terms of determinants would no longer hold.
For further details on the theory of quasi-determinants, see \cite{Olver}.
\par
The following proposition clarifies the notion of bi-orthogonality for SBPS.
\begin{pro}
The monic SBPS $P_{1}$  and $P_2$ are Sobolev-bi-orthogonal, that is, they satisfy the relation
\begin{align*}
 (P_{1,r},P_{2,k};\W)&:=h_r\delta_{r,k}
\end{align*}
with the further properties
\begin{align*}
 (P_{1,l},x^r;\W)&:=\delta_{l,r}h_{r} \,\,\,\,\,\,\forall r\leq l & &\Longrightarrow &
 \sum_{k=0}^{l}\sum_{j=0}^{r} & \left\langle P_{1,l}^{(k)},\frac{\d^j x^r}{\d x^j} \right\rangle_{k,j}= \begin{cases}
                                                                                               0  & \forall r<l \\
                                                                                               h_l & r=l
                                                                                              \end{cases} \\
 (x^r,P_{2,l},;\W)&:=h_r\delta_{r,l} \,\,\,\,\,\,\forall r\leq l & &\Longrightarrow &
 \sum_{k=0}^{r}\sum_{j=0}^{l} & \left\langle \frac{\d^j x^r}{\d x^j},P_{2,l}^{(k)} \right\rangle_{j,k}= \begin{cases}
                                                                                               0  & \forall r<l \\
                                                                                               1 & r=l
                                                                                              \end{cases}
\end{align*}
\end{pro}
\begin{proof}
 The previous relations are a direct consequence of the $LU$ factorization of the moment matrix $G_{\W}$.
\end{proof}
\begin{definition}
Let $f(x)=\frac{1}{y-x}$ belong to the subspace $\A_{\W}^{\mathcal{N}}(\Omega)$. Then we introduce the second kind functions
 \begin{align*}
  C_{1,l}(y)&:=\int_{\Omega}
  \sum_{k=0}^{l}\sum_{j=0}^{\mathcal{N}} P_{1,l}^{(k)}(x) \d \mu_{k,j} \left[\frac{\partial^j}{\partial x^j} \left(\frac{1}{y-x}\right)\right]
  =\left(P_{1,l}(x),\frac{1}{y-x} ;\W(x) \right)\ , &   y &\notin \Omega\ ,\\
   C_{2,l}(y)&:=\int_{\Omega}
  \sum_{k=0}^{\mathcal{N}}\sum_{j=0}^{l} \left[\frac{\partial^j}{\partial x^j} \left(\frac{1}{y-x}\right)\right] \d \mu_{j,k} P_{2,l}^{(k)}(x)
  =\left(\frac{1}{y-x},P_{2,l}(x) ;\W(x) \right)\ , &    y &\notin \Omega  \ .
 \end{align*}
\end{definition}

\begin{pro}
The associated Sobolev second kind functions $C_{\alpha}(y)$ admit the following representation in terms of the LU factorization matrices
for all $y$ such that $|y|>max\{|x|,x\in \Omega \}$
\begin{align*}
 C_1(y)&=H(S_2^{-1})^{\top}\chi^*(y):=\begin{pmatrix} \label{C1}
                       C_{1,0}(y) \\
                       C_{1,1}(y) \\
                       \vdots \\
                       C_{1,k}(y) \\
                       \vdots
                      \end{pmatrix}, &
  C_2(y)&=H(S_1^{-1})^{\top}\chi^*(y):=\begin{pmatrix}
                       C_{2,0}(y) \\
                       C_{2,1}(y) \\
                       \vdots \\
                       C_{2,k}(y) \\
                       \vdots
                      \end{pmatrix} \ .
\end{align*}
\end{pro}
\begin{proof}
In order to prove any of the two expressions it is enough to observe that whenever $\forall |x|<|y|$
 \begin{align*}
  \chi(x)^{\top} \cdot \chi(y)^*=\frac{1}{y}\sum_{n=0}^{\infty} \left(\frac{x}{y}\right)^{n}=\frac{1}{y-x} \ .
 \end{align*}
Also, since the given expressions in the proposition can be rewritten as
\begin{align*}
C_1(y)&=S_1 G_{\W} \chi^*(y), &
 \left(C_2(y)\right)^{\top}&=\left(\chi^*(y)\right)^{\top} G_{\W} S_2^{\top} \ ,
\end{align*}
we deduce that, for example for $C_1$
\begin{align*}
 C_1(y)&=&S_1 G_{\W} \chi^*(y)=S_1 \int_{\Omega} \boldsymbol{\chi}\,\,\, \W \,\,\, \boldsymbol{\chi}^{\top} \cdot \chi^*(y) =
 \int_{\Omega} \begin{pmatrix}
                P_1(x) & P_1'(x) & \dots & P_1^{(k)}(x) & \dots
               \end{pmatrix} \W
               \begin{pmatrix}
               \frac{1}{y-x} \\
               \frac{\partial}{\partial x} \left(\frac{1}{y-x}\right) \\
               \vdots \\
               \frac{\partial^j}{\partial x^j} \left(\frac{1}{y-x}\right) \\
               \vdots
               \end{pmatrix}
\end{align*}
and similarly for $C_2(y)$.
\end{proof}

A natural question is to establish the relation between the SBPS (and associated second kind functions)
that arise from a given measure matrix $\W$ and the ones associated with its transposed $\W^{\top}$. A simple answer is provided by the following
\begin{pro}
Let $P_{\W,\alpha}$ and $C_{\W,\alpha}$ with $\alpha=1,2$ denote the SBPS and second kind functions that arise from the measure matrix $\W$ and
$P_{\W^{\top},\alpha}$ and $C_{\W^{\top},\alpha}$ the ones corresponding to $\W^{\top}$. Then we have
\begin{align*}
 P_{\W,1}&=P_{\W^{\top},2}  &  P_{\W,2}&=P_{\W^{\top},1} \\
 C_{\W,1}&=C_{\W^{\top},2}  &  C_{\W,2}&=C_{\W^{\top},1}
\end{align*}

\end{pro}
\begin{proof}
It is straightforward to see that $G_{\W^{\top}}=G_{\W}^{\top}$. The assumption of the $LU$ factorization property for the moment matrix implies
the proposition.
\end{proof}
The previous proposition implies that if $\W=\W^{\top}$ then $P_{\W,1}=P_{\W,2}$ and $C_{\W,1}=C_{\W,2}$,(usually studied case) as expected since
in such a case the $LU$ factorization is indeed a Cholesky factorization.

\subsection{Christoffel-Darboux Kernels}
The Christoffel-Darboux and Cauchy kernels will play a crucial role in the following considerations. We present here their formal definition in our context.
\begin{definition}
We introduce the  Christoffel--Darboux kernel, the Cauchy kernel, and the first and second kind mixed Christoffel--Darboux kernels, given by\\
\begin{itemize}
\item Christoffel--Darboux kernel
\begin{align*}
  K^{[l]}(x,y)&:=\sum_{k=0}^{l-1}P_{2k}(x) h_k^{-1} P_{1k}(y)=[P_2(x)^{\top}]^{[l]} \left(H^{-1}\right)^{[l]} [P_1(y)]^{[l]}
  =\left(\chi(x)^{[l]}\right)^{\top}\left(G^{[l]}\right)^{-1}\chi(y)^{[l]} \ ,
\end{align*}
\item Cauchy CD kernel
\begin{align*}
  Q^{[l]}(x,y)&:=\sum_{k=0}^{l-1}C_{2k}(x) h_k^{-1} C_{1k}(y)=[C_2(x)^{\top}]^{[l]} \left(H^{-1}\right)^{[l]} [C_1(y)]^{[l]}
  =\left(\chi^*(x)^{[l]}\right)^{\top}\left(G^{[l]}\right)\chi^*(y) \ ,
 \end{align*}
\item Mixed 1st CD kernel
\begin{align*}
  \mathcal{K}_1^{[l]}(x,y)&:=\sum_{k=0}^{l-1}C_{2k}(x) h_k^{-1} P_{1k}(y)=[C_2(x)^{\top}]^{[l]} \left(H^{-1}\right)^{[l]} [P_1(y)]^{[l]}
  =\left(\chi(x)^*\right)^{\top}\begin{pmatrix}
                                           \mathbb{I}_{l\times l} \\
                                           \hline
                                           (S_1^{-1})^{[\geq l, j]}(S_1^{-1})^{[l]}
                                          \end{pmatrix} \chi^{[l]}(y) \ ,
\end{align*}
\item Mixed 2nd CD kernel
\begin{align*}
  \mathcal{K}_2^{[l]}(x,y)&:=\sum_{k=0}^{l-1}P_{2k}(x) h_k^{-1} C_{1k}(y)=[P_2(x)^{\top}]^{[l]} \left(H^{-1}\right)^{[l]} [C_1(y)]^{[l]}
  =\left(\chi(x)^{[l]}\right)^{\top}\left(\begin{array}{c|c}
                                           \mathbb{I}_{l\times l} & (S_2^{\top})^{[l]}([S_2^{\top}]^{-1})^{[j,\geq l]}
                                          \end{array}\right) \chi^*(y) \ .
\end{align*}
\end{itemize}
\end{definition}
%The CD kernel has the reproducing property and acts as a projector (best approximation). For those cases where $\W$ is symmetric only one
%of the first or second mixed kernels is needed. No Bezoutian expressions for neither of them in principle.

%***************************************************************************\\
%****************************************************************************\\
%***************************************************************************\\
\begin{remark}
In the previous definition, the expressions of the standard Bezoutian kernels are not present. They would involve only two consecutive
orthogonal polynomials (or second kind functions) instead of all polynomials up to the degree of the kernel.
The lack of this expression is not surprising, since the Bezoutian kernels would correspond to having a three term recurrence relation for the
orthogonal polynomials (and second kind functions), that in principle is missing.
Despite that, all of the expected properties of the CD kernel still hold. This is, the CD Kernel still has the reproducing property,

\begin{align*}
 \left(K^{[l]}(x,z),K^{[l]}(z,y)\right)_{\W}=
 \left(\chi^{[l]}(x)\right)^{\top} \left(G_{\W}^{[l]}\right)^{-1}
 \left[\int_{\Omega} \boldsymbol{\chi}^{[l]}(z) \W(z) \left(\boldsymbol{\chi}^{[l]}(z)\right)^{\top} \right]
 \left(G_{\W}^{[l]}\right)^{-1} \chi^{[l]}(y)=K^{[l]}(x,y)
\end{align*}
and acts as a projector onto the basis of the SBPS. Therefore, given any function $f(x)\in \A^{\mathcal{N}}_{\W}(\Omega)$, one has
\begin{align*}
 \Pi_1^{[l]}[f(y)]&=\left(f(x),K^{[l]}(x,y)\right)_{\W}=\sum_{k}^{l-1} \left[\left(f,P_{2,k}\right)_{\W} h_k^{-1}\right]P_{1,k}(y) \\
 \Pi_2^{[l]}[f(x)]&=\left(K^{[l]}(x,y),f(y)\right)_{\W}=\sum_{k}^{l-1} P_{2,k}(x) \left[h_k^{-1} \left(P_{1,k},f\right)_{\W} \right]
\end{align*}
where we call $\Pi_{\alpha}^{[l]}[f(x)]$ the best
approximation of $f$ (in $(*,*)_{\W}$) in the basis $\{P_{\alpha,l}\}_{k=0}^{(l-1)}$ for $\alpha=\{1,2\}$.
Notice also that when $\W$ is symmetric, only one of the two mixed kernels is needed (no distinction between subindices $1,2$ exists).
\end{remark}
%***********************************************************\\
%*************************************************************\\
%**************************************************************\\
\section{Additive perturbations of the measure matrix}

In this section, we are interested in the following problem: Given the pairing $(G, g)$, where $G$ is a moment matrix whose associated SBPS is known, and $g$ is another matrix, find the SBPS associated to the new
moment matrix $\breve{G}=G+g$.
\par
The same problem, although from a different point of view, was also studied in \cite{Car3Tr}. The results proposed in the present work, when they are equivalent, possess alternative
proofs. At the same time, they are suited for the Sobolev context.

%not only is equivalent to that one but also presents an alternative proof and some illustrative examples in the Sobolev context.\\
%************************************************************************\\
%**************************************************************************\\

Generally speaking, the solution to this problem leads to interesting cases when we require that $g$ has some special features. Two nontrivial examples
are indeed the cases of \textit{coherent pairs} and of \textit{discrete Sobolev bilinear functions}.

Since their appearance \cite{Coherent}, coherent pairs have been largely investigated in the literature. In this work, we will limit ourselves to show how coherent pairs fit within our framework. Instead, the discrete Sobolev bilinear forms will be of considerable relevance in our subsequent discussion; therefore, we will pay special attention to them.

%************************************************************************************\\
%*************************************************************************************\\
%Coherent pairs have been largely investigated in the literature, so instead of devoting more time to them, we will just show how these fit
%in the framework of this paper.
%On the other hand, the discrete Sobolev bilinear forms will be of great aid in
%the following section so we will pay special attention to them. \\
%************************************************************************************\\
%*************************************************************************************\\

As a starting point of our analysis, suppose that our moment matrix can be written as $\breve{G}=G+g$. Since we assume that
$G$ has an associated SBPS, then it must be $LU$-factorizable; at the same time, the requirement that the SBPS associated to $\breve{G}$ exists implies that the latter matrix should be
$LU$-factorizable too. Therefore, we deduce the relation
\begin{align}
\breve{S}_1^{-1} \breve{H} \left(\breve{S}_2^{-1}\right)^{\top}=S_1^{-1} H \left(S_{2}^{-1}\right)^{\top}+g.
\end{align}
This motivates the following
\begin{definition}\label{dA} We introduce the matrices
\begin{align*}
 A&:=S_1 g S_2^{\top}  &  M_1&:=\breve{S}_1 S_1^{-1}    &  M_2&:=\breve{S}_2 S_2^{-1}
\end{align*}
\end{definition}

\begin{pro}
The matrices $M_1, M_2$ are the connection matrices between old and new polynomials
\begin{align*}
 M_1 P_1(x)&=\breve{P}_1(x)    &    M_2 P_2(x)&=\breve{P}_2(x)
\end{align*}
and provide an LU factorization of the matrix $H+A$:
\begin{align*}
 M_1^{-1} H \left(M_2^{-1}\right)^{\top}=H+A
\end{align*}
\end{pro}
\begin{proof}
The result follows from the requirement that both $\breve{G},G$ admit an LU factorization and from the observation that, by definition, both $M_1, M_2$ are lower uni-triangular.
\end{proof}
This last proposition allows to derive directly the following consequence.
\begin{pro}\label{ASBPSNA}
 The basis change from the old SBPS to the new one is given
\begin{align*}
 \breve{P}_{1,k}(x)&=\Theta_*\left[\begin{array}{c c c c c|c}
	        &         & (H+A)^{[k]} &  &           & \begin{matrix}  P_{1,0}(x) \\ P_{1,1}(x) \\ \vdots\\ P_{1,k-1}(x) \end{matrix}\\
	\hline
	(A)_{k,0} & (A)_{k,1} & \dots   &  & (A)_{k,k-1} & P_{1,k}(x) \end{array}\right], \\
\breve{P}_{2,k}(x)&=\Theta_*\left[\begin{array}{c c c c c|c}
	        &         & (H+A)^{[k]} &  &           & \begin{matrix}  A_{0,k}(x) \\ A_{1,k}(x) \\ \vdots\\ A_{k,k-1}(x) \end{matrix}\\
	\hline
	P_{2,0} & P_{2,1} & \dots   &  & P_{2,k-1} & P_{2,k}(x) \end{array}\right], \\
 \breve{h}_{k}&=\Theta_*\left[\begin{array}{c c c c c|c}
	        &         & (H+A)^{[k]} &  &           & \begin{matrix}  (A)_{0,k} \\ (A)_{1,k} \\ \vdots\\ (A)_{k-1,k} \end{matrix}\\
	\hline
	(A)_{k,0} & (A)_{k,1} & \dots   &  & (A)_{k,k-1} & (H+A)_{k,k} \end{array}\right]
\end{align*}
\end{pro}

We shall use this result and focus now on three cases. Firstly we will deal with the situation where $G$ and $g$ are the moment matrices
associated to a pair of related classical measures. Secondly we will consider the case when $g=\lambda Dg_2D^{\top}$ and $G=g_1$
where $g_1,g_2$ are the moment matrices associated to a couple of measures that form a coherent pair.
Finally we will study the case where $g$ is associated to a discrete Sobolev bilinear function.\\

\subsection{A first relation with classical OPS}
It is a well known fact that classical orthogonal polynomials can be regarded as a very specific case of SOPS.
As we are about to see, a consequence of this is that the previous relations become almost trivial when choosing the right measures.

%Let us first make a fast introduccion to the classical measures just pointing out some properties of them that will be usefull to ilustrate the example.

If we denote the classical measures by $u_{\gamma}$, where $\gamma$ refers to the parameters that define them, they are
\begin{itemize}
\item Hermite $u(x)=e^{-x^2},\,\,\,x\in \mathbb{R}$ ; ($\gamma=\{ \varnothing \}$).
\item Laguerre $u_{\alpha}(x)=x^{\alpha} e^{-x},\,\,\,\alpha >-1,\,\,\,x\in \mathbb{R}_+$ ; ($\gamma=\{ \alpha \}$).
\item Jacobi $u_{\alpha,\beta}(x)=(1-x)^{\alpha}(1+x)^{\beta},\,\,\,\alpha,\beta > -1,\,\,\,x\in(-1,1)$ ; ($\gamma=\{ \alpha,\beta \}$).
\end{itemize}
We will use $P_{\gamma}(x)=S_{\gamma} \chi(x)$ to denote the monic orthogonal polynomials $\{P_{\gamma,n} \}_n$ associated to each of them in terms of the
LU factorization matrices $S_{\gamma}$ of the corresponding moment matrix $g_{\gamma}$.\\
There are many ways to characterize classical measures; the one that is suited for our purposes is to express them in terms of a Pearson differential equation:
\begin{align*}
 p_2(x)\frac{\d u_{\gamma}}{\d x}&= p_{1,\gamma}(x) u_{\gamma} &
 p_2^k(x) u_{\gamma}&= u_{\gamma+k} &
&\mbox{where $deg[p_2]\leq 2$ and $deg[p_{1,\gamma}]=1$}.
\end{align*}
%$p_2(x) \frac{\d}{\d x} u_{\gamma} =p_{1,\gamma} u_{\gamma}$.
\begin{itemize}
 \item Hermite $p_{1}=-2x$, $p_2=1$.
 \item Laguerre $p_{1,\alpha}=(\alpha-x)$, $p_2=x$.
 \item Jacobi $p_{1,\alpha,\beta}=-[(\alpha-\beta)+(\alpha+\beta)x]$, $p_2=1-x^{2}$.
\end{itemize}

This equation is relevant in the discussion of many properties of the associated OPS. In particular, it implies that
$P_{(\gamma+1),n}(x)=\frac{P'_{\gamma,n+1}(x)}{n+1}$, which in matrix form gives the crucial relation $D=S_{\gamma} D S_{\gamma+1}^{-1}$.\\

As a simple example, consider the following Sobolev inner product
\begin{align*}
(f,h)&=\int f(x) h(x) u_{\gamma}(x) \d x + \lambda \int f'(x) h'(x) u_{\gamma+1}(x) \d x  &  \lambda&>0,
\end{align*}
that we wish to interpret as an additive perturbation $\breve{G}=G+g$ with the identifications
$G=g_{\gamma}$ and $g=\lambda Dg_{\gamma+1} D^{\top}$. The crucial relation
$D=S_{\gamma} D S_{\gamma+1}^{-1}$ implies for $A$ the particularly simple form
\begin{align*}
 A=\lambda S_{\gamma} D S_{\gamma+1}^{-1} H_{\gamma+1} \left( S_{\gamma} D S_{\gamma+1}^{-1} \right)^{\top}=\lambda D H_{\gamma+1} D^{\top}=
 \lambda \begin{pmatrix}
  0 &                    &                  &       &                    &\\
    &1^{2}h_{\gamma+1,0} &                  &       &                    &\\
    &                    &2^2h_{\gamma+1,1} &       &                    & \\
    &                    &                  &\ddots &                    & \\
    &                    &                  &       &k^2h_{\gamma+1,k-1} &    \\
    &                    &                  &       &                    & \ddots
 \end{pmatrix},
\end{align*}
which makes the quasi-determinantal expressions in Proposition \ref{ASBPSNA} almost trivial.
\begin{cor}
The SBPS $\breve P_k$ and norms $\breve{h}_k$ for the following inner product
\begin{align*}
(f,h)&=\int f(x) h(x) u_{\gamma}(x) \d x + \lambda \int f'(x) h'(x) u_{\gamma+1}(x) \d x & \lambda&>0
\end{align*}
are given by
\begin{align*}
 \breve P_k(x)&=P_{\gamma,k}(x)    &   \breve{h}_{k}&=h_{\gamma,k}+\lambda k^2 h_{\gamma+1,k-1}
\end{align*}
\end{cor}
\noi For future reference we shall also discuss here a couple of additional properties of classical OPS.
They will be useful below, in relation with the study of equivalence classes of measure matrices.

\noi \textit{i}) It is almost straightforward to see that
\begin{align*}
 p_2^k(x) u_{\gamma}\mid_{\partial \Omega}&= u_{\gamma+k}\mid_{\partial \Omega}=0 &
 &\forall k\geq1 \ .
\end{align*}
In the previous relation, one can allow for even smaller values of $k$, depending on the value of $\gamma$. However, to avoid
the worst case  $\gamma=-1$, taking $k\geq1$ will be sufficient in all situations.

\noi \textit{ii}) A less trivial property is expressed by the following
\begin{pro}\label{ClSSW}
The measure $u_{\gamma+k}$ satisfies the relations %behaves as the $\omega_k$
\begin{align*}
 \frac{\d^r}{\d x^r}\left( p_2^k u_{\gamma} \right)&=\varphi_{k,r}(x)u_{\gamma} &  & 0\leq r\leq k \\
 \varphi_{k,r}u_{\gamma}\mid_{\partial \Omega}&=0                               &  &0\leq r \leq (k-1)
\end{align*}
where $\varphi_{k,r}(x)$ is a suitable polynomial.
\end{pro}
\begin{proof}
Let $Q$ be a $\mathcal{C}^{k}$ function; and taking into account the Pearson equation it is easy to see that
 \begin{align*}
  \frac{\d}{\d x}\left[Qp_2^k u_{\gamma}\right]=Q'p_2^ku_{\gamma}+Qkp'_2 p_2^{k-1}u_{\gamma}+Qp_2^{k}\frac{p_{1,\gamma}}{p_2}u_{\gamma}
  =\mathcal{O}_k\left[Q \right]p_2^{k-1}u_{\gamma} \ ,
 \end{align*}
where $\mathcal{O}_k:=p_2\frac{\d}{\d x}+[kp'_2+p_{1,\gamma}]$ is a first order linear differential operator. Differentiating the previous relation we have
 \begin{align*}
  \frac{\d^2}{\d x^2}\left[ Q p_2^k u_{\gamma} \right]=\frac{\d}{\d x}\left[\mathcal{O}_k[Q]p_2^{k-1} u_{\gamma}\right] \ .
  =\mathcal{O}_{k-1}\circ \mathcal{O}_k [Q] p_2^{k-2} u_{\gamma}
 \end{align*}
 Therefore, differentiating $r$ times one gets
 \begin{align*}
  \frac{\d^r}{\d x^r}\left[Q p_2^k u_{\gamma} \right]=
  \mathcal{O}_{k-(r-1)}\circ \mathcal{O}_{k-(r-2)}\dots \circ \mathcal{O}_{k} [Q] p_2^{k-r} u_{\gamma}:=\mathcal{O}_{k}^{k-(r-1)}[Q] p_2^{k-r} u_{\gamma}
 \end{align*}
 Notice that we have defined the operator $\mathcal{O}_k^{j}[f]$, $\forall j\leq k$, but in order to make our notation a bit more compact let
 us add to this definition the case $\mathcal{O}_k^{k+1}[f]:=f$ as the identity operator,
 this way for $ 0 \leq r \leq k $
 \begin{align*}
  \varphi_{k,r}:=\mathcal{O}_{k}^{k-(r-1)}[1] p_2^{k-r}
 \end{align*}
 (note that according to the definition of $\mathcal{O}_k^{k+1}[f]:=f$ we would have $\varphi_{k,0}=p_2^{k}$)
 and now from \noi \textit{i}) the proposition is proven.

\end{proof}
\subsection{Coherent Pairs}
We are interested in obtaining the SBPS associated to the inner product
\begin{align} \label{CP}
 (f,h)_c&:=\int f(x) h(x) \d \mu_1(x) + \lambda \int f'(x) h'(x) \d \mu_2(x)  & \lambda&>0,
\end{align}
where $\d \mu_1(x)$ and $\d \mu_2(x)$ form a \textit{coherent pair of measures}. This inner product, in terms of moment matrices reads
\begin{align*}
 \breve{G}=g_1+\lambda D g_2 D^{\top}
\end{align*}
and therefore can be studied from the additive perturbation approach.
Let us introduce some notation for the moment matrices, their factorization and corresponding OPS. For each of the two involved measures we will denote:
\begin{align*}
 \d \mu_1(x) \longrightarrow g_1=S^{-1}H \left(S^{-1}\right)^{T} \longrightarrow P(x)=S \chi(x) \\
 \d \mu_2(x) \longrightarrow g_2=Z^{-1}K \left(Z^{-1}\right)^{T} \longrightarrow Q(x)=Z \chi(x) \ .
\end{align*}
One of the possible characterizations of a coherent pair is given in terms of a relation between the OPS associated to each of the measures. Precisely,
it is said that $\d \mu_1(x)$ and $\d \mu_2(x)$ form a coherent pair if there exist some non zero constants $\{r_k\}_{k=1}^{\infty}$ such that
\begin{align*}
 Q_k(x)&=\frac{1}{k+1}P_{k+1}'(x)-\frac{r_k}{k}P_k'(x)        &      \forall k=1,2,\dots
\end{align*}
It is worth pointing out that the coefficient that goes with $P_{k+1}'(x)$ is chosen according to the fact that we wish to generate monic orthogonal
polynomials, while both the sign and coefficient that go with $r_k P_k'(x) $ are selected for convenience.
To interpret this construction as an additive perturbation and using the notation presented in definition \ref{dA} we have
\begin{align*}
A=\lambda \left(S D Z^{-1}\right) K \left(S D Z^{-1} \right)^{\top}
\end{align*}
We introduce the lower matrix $R^{-1}$ according to the formulae
\begin{align*}
S D Z^{-1}&=\begin{pmatrix}
            0      & 0     & 0  &  \dots      \\
            1      & 0     & 0  &  \dots      \\
            *      & 2     & 0  &   \dots     \\
            *      & *     & 3  &    0     \\
            \vdots &\vdots &    & \ddots
           \end{pmatrix}=
           \begin{pmatrix}
            0                    & 0                   & 0                    &  \dots      \\
            (R^{-1})_{0,0}       & 0                   & 0                    &  \dots      \\
            (R^{-1})_{1,0}       & (R^{-1})_{1,1}      & 0                    &   \dots     \\
            (R^{-1})_{2,0}       &(R^{-1})_{2,1}       & (R^{-1})_{2,2}       &    \ddots    \\
                \vdots           & \vdots              &   \vdots             &
           \end{pmatrix}, &
  R^{-1}:=  \begin{pmatrix}
            1                  & 0                 & 0  &  \dots      \\
            (R^{-1})_{1,0}     & 2                 & 0  &  \dots      \\
            (R^{-1})_{2,0}     &(R^{-1})_{2,1}     & 3  &   \dots     \\
            \vdots &\vdots &    & \ddots
           \end{pmatrix}
\end{align*}
This last definition is motivated by the fact that it allows to write the truncations of $A$ as
\begin{align*}
 A^{[k]}=\left(\begin{array}{c|c}
          0 &  \boldsymbol{0} \\
          \hline
          \boldsymbol{0}^{\top} & \lambda \left(R^{-1}K \left(R^{-1}\right)^{T} \right)^{[k-1]}
         \end{array}\right)=
         \left(\begin{array}{c|c}
          0 &  \boldsymbol{0}\\
          \hline
          \boldsymbol{0}^{\top} & \lambda \left(R^{[k-1]}\right)^{-1} K^{[k-1]} \left(\left(R^{[k-1]}\right)^{-1}\right)^{T}
         \end{array}\right)
\end{align*}

We have used here $\boldsymbol{0}$ for a row of zeroes. The second equality holds due to the lower
triangular shape of $R$. By means of Proposition \ref{ASBPSNA}, we deduce the following expression for the SBPS:
\begin{align*}
 \breve{P}_k&=P_k(x)
 -\lambda \begin{pmatrix}
 \left(R^{-1}K \left(R^{-1}\right)^{T} \right)^{[k]}_{k-1,0} & \dots &
 \left(R^{-1}K \left(R^{-1}\right)^{T} \right)^{[k]}_{k-1,k-2}
         \end{pmatrix}
\left[\left(R^{-1}K \left(R^{-1}\right)^{T} \right)^{[k-1]}+\tilde{H}^{[k-1]}\right]^{-1}
\begin{pmatrix}
 P_1(x) \\ P_2(x) \\ \vdots \\ P_{k-1}(x)
\end{pmatrix}
\end{align*}

%\begin{align*}
% \breve{P}_k(x)&=P_k(x)-\lambda \begin{pmatrix}
% \left(R^{-1}K R^{-T} \right)^{[k]}_{k-1,0} & \left(R^{-1}K R^{-T} \right)^{[k]}_{k-1,1} & \dots
% \left(R^{-1}K R^{-T} \right)^{[k]}_{k-1,k-2}
%         \end{pmatrix} \cdot \\
%&\cdot\left[\left(R^{-1}K R^{-T} \right)^{[k-1]}+\tilde{H}^{[k-1]}\right]^{-1}
%\begin{pmatrix}
% P_1(x) \\ P_2(x) \\ \vdots \\ P_{k-1}(x)
%\end{pmatrix}
%\end{align*}

Here $\tilde{H}^{[k-1]}:=diag \{H_1,H_2,\dots,H_{k-2}\}$. This is a general result that would be valid for any inner product of the form \eqref{CP}. In order to simplify it, we will use the fact that we are working with coherent pairs in order to find a simple expression
for $\left(R^{-1}K R^{-T} \right)^{[k]}$. To this aim, remember that
\begin{align*}
 SDZ^{-1}Q(x)&=P'(x) & &\Rightarrow  &  R^{-1} \begin{pmatrix}
                                                Q_0 \\ Q_1 \\ Q_2 \\ \vdots
                                               \end{pmatrix}&=\begin{pmatrix}
                                                P'_1 \\ P'_2 \\ P'_3 \\ \vdots
                                               \end{pmatrix}  & &\Rightarrow  &
                                               \begin{pmatrix}
                                                Q_0 \\ Q_1 \\ Q_2 \\ \vdots
                                               \end{pmatrix}&=R \begin{pmatrix}
                                                P'_1 \\ P'_2 \\ P'_3 \\ \vdots
                                               \end{pmatrix}
\end{align*}
At the same time, due to coherence property, we know that $R$ has a particularly simple lower bi-diagonal shape
\begin{align*}
 R=\begin{pmatrix}
    1              &                &               &         \\
    -\frac{r_1}{1} &   \frac{1}{2}  &               &         \\
                   & -\frac{r_2}{2} &  \frac{1}{3}  &         \\
                   &                & -\frac{r_3}{3}& \ddots  \\
                   &                &               & \ddots  \\
   \end{pmatrix}
\end{align*}
It is  now easy to see that after introducing the matrices
\begin{align*}
 r&:=\begin{pmatrix}
    0              &                &               &         \\
    r_1            &   0            &               &         \\
                   & r_2            &  0            &         \\
                   &                & r_3          & \ddots  \\
                   &                &               & \ddots  \\
   \end{pmatrix} &
  N &:=\begin{pmatrix}
    1              &                &               &         \\
                   &   2            &               &         \\
                   &                &  3            &         \\
                   &                &               & \ddots  \\
   \end{pmatrix}
\end{align*}
one obtains
\begin{align*}
 RN&=\left(\mathbb{I}-r\right) &  &\Longrightarrow  &
 R^{-1}=N\left(\mathbb{I}-r\right)^{-1}=N\left(\mathbb{I}+r+r^{2}+\dots \right) &  &\Longrightarrow  &
 \left(R^{[k]}\right)^{-1}=N^{[k]}\left(\mathbb{I}^{[k]}+r^{[k]}+\dots+ (r^{k-1})^{[k]}\right)
\end{align*}
Therefore
\begin{align*}
 \lambda \left(R^{-1}K \left(R^{-1}\right)^{T} \right)^{[k]}=
 \lambda N^{[k]}\left(\mathbb{I}^{[k]}+r^{[k]}+(r^{2})^{[k]}+\dots+ (r^{k-1})^{[k]}\right) K^{[k]}
                \left(\mathbb{I}^{[k]}+r^{[k]}+(r^{2})^{[k]}+\dots+ (r^{k-1})^{[k]}\right)^{\top} N^{[k]}
\end{align*}
which finally implies that the $\breve{P}_k$ depend only on the first $k-1$ parameters $\{r_1,r_2,\dots,r_{k-1}\}$ that
characterized the coherence and the norms of the original polynomials. For instance, consider
\begin{align*}
 \lambda \left(R^{-1}K \left(R^{-1}\right)^{T} \right)^{[3]}=
 \lambda \begin{pmatrix}
           K_0 &  2r_1 K_0                 & 3 r_2 r_1 K_0               \\
      2r_1 K_0 & 2^2(r_1^{2}K_0+K_1)       & 2\cdot3(r_1^2r_2K_0+r_2K_1) \\
  3 r_2 r_1 K_0&2\cdot3(r_1^2r_2K_0+r_2K_1)& 3^2(r_1^2r_2^2K_0+r_2^2K_1+K_2)
         \end{pmatrix}
\end{align*}
which yields
\begin{align*}
 \breve{P}_0&=P_0 \\
 \breve{P}_1&=P_1 \\
 \breve{P}_2&=P_2-\lambda (2r_1K_0)[\lambda K_0+H_1]^{-1}P_1\\
 \breve{P}_3&=P_3-\lambda \begin{pmatrix}
                           3 r_2 r_1 K_0&2\cdot3(r_1^2r_2K_0+r_2K_1)
                          \end{pmatrix}
                          \begin{pmatrix}
                           K_0+H_1 &  2r_1 K_0 \\
                           2r_1 K_0 & 2^2(r_1^{2}K_0+K_1)+H_2
                          \end{pmatrix}^{-1}
                         \begin{pmatrix}
                          P_1 \\ P_2
                         \end{pmatrix}
\end{align*}

Observe that the previous nice expressions for the Sobolev polynomials are just a consequence of the lower bi--diagonal structure of $R$
(which came from the characterization of the coherent pair $\{\d\mu_1,\d\mu_2 \}$ in terms
of their associated OPS).

A possible generalization of the notion of coherent pairs can be obtained by considering a bi--$m \times m$ block diagonal $R$ and proceeding in the same way. This suggests the following
\begin{definition}
 We shall say that $\{\d\mu_1,\d\mu_2 \}$ form a $m \times m$ block coherent pair if their associated OPS are related as follows
\begin{align*}
 \begin{pmatrix}
  Q_0 \\ Q_1 \\ \vdots \\ Q_{m-1}
 \end{pmatrix}&=\left(R_{m}\right)_{[0][0]}
\begin{pmatrix}
  P'_1 \\ P'_2 \\ \vdots \\ P'_{m}
 \end{pmatrix}, &
 \begin{pmatrix}
  Q_{km} \\ Q_{km+1} \\ \vdots \\ Q_{km+m-1}
 \end{pmatrix}&=\left(R_{m}\right)_{[k][k-1]}
\begin{pmatrix}
  P'_{(k-1)m+1} \\ P'_{(k-1)m+2} \\ \vdots \\ P'_{(k-1)m+m}
 \end{pmatrix}+
 \left(R_{m}\right)_{[k][k]}
\begin{pmatrix}
  P'_{km+1} \\ P'_{km+2} \\ \vdots \\ P'_{km+m}
 \end{pmatrix} &
 & \forall k \geq 1
\end{align*}
where $\left(R_{m}\right)_{[k][k-1]},\left(R_{m}\right)_{[k][k]} \in \mathbb{M}_m$ and
\begin{align*}
\left(R_{m}\right)_{[k][k]}=\begin{pmatrix}
  \frac{1}{km+1}   &                &               &         \\
    *              & \frac{1}{km+2} &               &         \\
      \vdots       &    \vdots      &  \ddots       &         \\
       *           &      *         &  \dots        & \frac{1}{(k+1)m}
   \end{pmatrix} \ .
\end{align*}
\end{definition}
Note that the case $m=1$ reproduces just the standard concept of coherent pairs that we treated before. The case $m=2$ contains as a particular case
the symmetrically coherent pairs since
\begin{align*}
 \begin{pmatrix}
  Q_{2k} \\ Q_{2k+1}
 \end{pmatrix}&=
 \begin{pmatrix}
  * & 0 \\
  0 & *
 \end{pmatrix}
\begin{pmatrix}
  P'_{2k-1} \\ P'_{2k}
 \end{pmatrix}+
 \begin{pmatrix}
  \frac{1}{2k+1} & 0 \\
  0              & \frac{1}{2k+1}
 \end{pmatrix}
\begin{pmatrix}
  P'_{2k+1} \\ P'_{2k+2}
 \end{pmatrix}
\end{align*}
The way to proceed for a general $m$ would follow the same steps as the case $m=1$. Firstly define
\begin{align*}
 N_m &:=\begin{pmatrix}
       \left(R_{m}\right)_{[0][0]}^{-1}   &                                  &          \\
                                          &\left(R_{m}\right)_{[1][1]}^{-1}  &          \\
                                          &                                  & \ddots
      \end{pmatrix}, &
 r_m &:=\begin{pmatrix}
       0                                  &                                  &          \\
     \left(r_m\right)_{[1][0]}            &    0                             &          \\
                                          &  \left(r_m\right)_{[2][1]}       &   0        \\
                                          &                                  &  \ddots
      \end{pmatrix}, &
			\end{align*}
\[			
\left(r_m\right)_{[k][k-1]}=-\left(R_m\right)_{[k][k-1]}\left(R_m\right)_{[k-1][k-1]}^{-1} \ .
\]
Therefore $R_m N_m= \mathbb{I}-r_m$ and taking the inverse of its truncations one obtains
\begin{align*}
  \left(R_m^{[km]}\right)^{-1}=N_m^{[km]}\left(\mathbb{I}^{[km]}+r_m^{[km]}+(r_m^{2})^{[km]}+\dots+ (r_m^{k-1})^{[km]}\right) \ ,
\end{align*}
which would allow us to write the associated SOPS only in terms of the entries of the matrices that characterized the
$m \times m$ block coherent pair.

An open problem is to construct examples of $m \times m$ block coherent pairs.
An illustrative example is offered by the previously mentioned symmetrically coherent pair case, which has $m=2$. Let us take
$k=2$. In this case we have
\begin{align*}
\left(R_2^{[2\cdot 2]}\right)^{-1}&=
\begin{pmatrix}
 1   &    &    &    \\
     & 2  &    &    \\
     &    & 3  &    \\
     &    &    & 4
\end{pmatrix}\left[ \mathbb{I}_{4\times 4}+
\begin{pmatrix}
 0   &   0   &  0  &  0 \\
 0   &   0   &  0  &  0 \\
 r_2 &   0   &  0  &  0 \\
 0   &   r_3 &  0  &  0
 \end{pmatrix}\right] \\
&\Longrightarrow
 \lambda \left(R^{-1}K \left(R^{-1}\right)^{T} \right)^{[4]}=
 \lambda \begin{pmatrix}
           K_0 &  0                        & 3  K_0r_2        &  0           \\
      0        & 4K_1                      & 0                &   8K_1r_3             \\
  3  K_0r_2    & 0                         & 9(K_2+K_0r_2^{2})&    0             \\
  0            &  8K_1r_3                  &     0            &   16(K_3+K_1r_3^{2})
         \end{pmatrix}
\end{align*}
whence we deduce
\begin{align*}
 \breve{P}_0&=P_0  & \breve{P}_1&=P_1   & \breve{P}_2&=P_2
\end{align*}
\begin{align*}
\breve{P}_3&=P_3-\lambda \begin{pmatrix}3K_0r_2 & 0 \end{pmatrix}
\left[\begin{pmatrix}
       K_0  &  0  \\
       0    &  4K_1
      \end{pmatrix}+
      \begin{pmatrix}
       H_1  &  0  \\
       0    &  H_2
      \end{pmatrix}
 \right]^{-1}\begin{pmatrix}P_1 \\ P_2 \end{pmatrix}=P_3-\lambda \beta_1 P_1 \\
 \breve{P}_4&=P_4-\lambda \begin{pmatrix}0 & 8K_1r_3 & 0 \end{pmatrix}
\left[\begin{pmatrix}
       K_0     &  0    &  3K_0r_2 \\
       0       &  4K_1 & 0   \\
       3K_0r_2 & 0     & 9(K_2+k_0r_2^2)
      \end{pmatrix}+
      \begin{pmatrix}
       H_1  &  0   &  0 \\
       0    &  H_2 &  0 \\
       0    &      &  H_3
      \end{pmatrix}
 \right]^{-1}\begin{pmatrix}P_1 \\ P_2 \\ P_3 \end{pmatrix}=P_4-\lambda \beta_2 P_2
\end{align*}
where the $\beta$'s are given in terms of $K,H,r$. A thorough treatment of this approach is beyond the scope of this work and will be studied elsewhere.

\subsection{Discrete Sobolev bilinear forms}
The above definition of the Sobolev bilinear function has been proposed in full generality, i.e.
without any reference to the explicit expressions of the entries $\d \mu_{i,j}$ in $\W$.
A particularly interesting case is obtained when the entries are allowed to be Dirac's $\delta$ distributions.

In this perspective, we shall call the part of the Sobolev bilinear function involving a continuous support
the \textit{continuous part} of the bilinear function, and that involving a discrete support its \textit{discrete part}.
Thus, once we split a Sobolev bilinear function into its continuous and discrete parts, we can consider the former as an additive
perturbation of the latter. According to this philosophy, given a set of nodes and their
multiplicities $\{x_i,n_i,m_i \}_{i=1}^{s}$ let us study the following Sobolev bilinear function
\begin{align*}
 \left(f,h \right)_{\breve{\W}}&:=
 \left(f,h \right)_{\W}+\sum_{i=1}^{s} \sum_{k=0}^{n_i-1} \sum_{j=0}^{m_i-1} \xi^{(i)}_{k,j} h^{(k)}(x_i) f^{(j)}(x_i) &
 & \Longrightarrow   &  \breve{G}=G+g
\end{align*}
Notice that the function space on which this Sobolev bilinear form is defined will be
$\A^{\breve{\mathcal{N}}}_{\breve{\W}}(\breve{\Omega})\subseteq \A^{\mathcal{N}}_{\W}({\Omega})$
where $\breve{\Omega}=\Omega \bigcup_i x_i$ and $\breve{N}=max\big\{N,\{(n_i-1)\}_i,\{(m_i-1)\}_i \big\}$.
In order to see how the matrix $A$ looks like in this case, we propose the following
\begin{definition}
Given a function $f\in \A^{\breve{\mathcal{N}}}_{\breve{\W}}(\breve{\Omega})$, we introduce the vectors
\begin{align*}
 N[f(x)]&:=\left(f(x_1),f'(x_1),\dots,f^{(n_1-1)},f(x_2),f'(x_2),\dots,f^{(n_2-1)},\dots,f(x_s),f'(x_s),\dots,f^{(n_s-1)} \right) \\
 M[f(x)]&:=\left(f(x_1),f'(x_1),\dots,f^{(m_1-1)},f(x_2),f'(x_2),\dots,f^{(m_2-1)},\dots,f(x_s),f'(x_s),\dots,f^{(m_s-1)} \right)
\end{align*}
and the following matrix $\Xi\in \sum_i n_i \times \sum_i m_i$
\begin{align*}
\Xi&:=\begin{pmatrix}
       \xi^{(1)} &         &       &    \\
                 &\xi^{(2)}&       &    \\
                 &         &\ddots &    \\
                 &         &       &\xi^{(2)}
      \end{pmatrix} &
\xi^{(i)}&:= \begin{pmatrix}
            \xi^{(i)}_{0,0} & \xi^{(i)}_{0,1} &  \dots &  \xi^{(i)}_{0,m_i-1} \\
            \xi^{(i)}_{1,0} &               &        &                    \\
            \vdots        &               &        &                     \\
           \xi^{(i)}_{n_i-1}&               &        & \xi^{(i)}_{n_i-1,m_i-1}
           \end{pmatrix}
\end{align*}
\end{definition}
\begin{pro}
Given an additive perturbation of a discrete Sobolev type form, the matrix $A$ can be written in terms of the old polynomials as
\begin{align*}
 A^{[k]}=N[P_1^{[k]}]\left(\Xi\right) M[P_2^{[k]}]^{\top} \ .
 \end{align*}
\end{pro}
\begin{proof}
The proposition follows easily from the relations
\begin{align*}
 g&=N[\chi] \left( \Xi \right) M[\chi]^{\top} & A^{[k]}&=S_{1}^{[k]} g^{[k]} \left( S_{2}^{[k]}\right)^{\top}  &
 S_{1}^{[k]}N[\chi]&=N[P_1^{[k]}]  &  S_{2}^{[k]}M[\chi]&=N[P_2^{[k]}]
\end{align*}
\end{proof}
It is  useful to define the following $\sum_i n_i \times \sum_i m_i$ matrix, suitable for the discrete Sobolev problem at hand, whose entries are the derivatives of the
CD Kernel evaluated at the points $\{x_i\}$ up to $\{(n_i-1),(m_i-1)\}$ times.
\begin{definition}\label{mathKK}
We introduce the CD matrix
\begin{align*}
\mathbb{K}^{[k]}&:=\left(M[P_2^{[k]}]\right)^{\top}\left(H^{[k]}\right)^{-1} \left( N[P_1^{[k]}] \right)=
\begin{pmatrix}
 \mathbb{K}^{[k]}_{[1][1]}  & \mathbb{K}^{[k]}_{[1][2]} &  \dots &  \mathbb{K}^{[k]}_{[1][s]} \\
 \mathbb{K}^{[k]}_{[2][1]}  & \mathbb{K}^{[k]}_{[2][2]} &  \dots &  \mathbb{K}^{[k]}_{[2][s]} \\
                            &                           &        &                            \\
 \mathbb{K}^{[k]}_{[s][1]}  & \mathbb{K}^{[k]}_{[s][2]} &  \dots &  \mathbb{K}^{[k]}_{[s][s]}
\end{pmatrix} \\
where\\
\mathbb{K}^{[k]}_{[i][j]}&:=
\begin{pmatrix}
 \left(K^{[k]}(x_i,x_j)\right)^{(0,0)} &  \left(K^{[k]}(x_i,x_j)\right)^{(0,1)} &  \dots & \left(K^{[k]}(x_i,x_j)\right)^{(0,n_j-1)} \\
 \left(K^{[k]}(x_i,x_j)\right)^{(1,0)} &  \left(K^{[k]}(x_i,x_j)\right)^{(1,1)} &  \dots & \left(K^{[k]}(x_i,x_j)\right)^{(1,n_j-1)} \\
                                       &                                        &        &                                           \\
 \left(K^{[k]}(x_i,x_j)\right)^{(m_i-1,0)} &  \left(K^{[k]}(x_i,x_j)\right)^{(m_i-1,1)} &  \dots & \left(K^{[k]}(x_i,x_j)\right)^{(m_i-1,n_j-1)}  \\
\end{pmatrix} \ .
\end{align*}
Here we have used the notation $\left(K^{[k]}(x_i,x_j)\right)^{(t,d)}:=\frac{\partial^{t+d} K^{[k]}(x,y)}{\partial x^t \partial y^d}|_{(x,y)=(x_i,x_j)}$
\end{definition}
The previous definitions and analysis allow us to state the main result of this section.
\begin{pro}
The discrete part of a Sobolev bilinear function is as an additive perturbation of its continuous counterpart.
Also, the SBPS associated with the Discrete$+$Continuous part can be represented in terms of the
following quasi-determinantal formulas involving only the continuous part of the Sobolev bilinear function.
\begin{align} \label{P12k}
 \breve{P}_{1,k}(x)&=\left( \begin{array}{c|c}
                 \mathbb{I}+\mathbb{K}^{[k]}\Xi & M[K^{[k]}(\cdot,x)]^{\top} \\
                 \hline
                 N[P_{1,k}] \Xi                      & P_{1,k}(x)
                \end{array} \right)\ , &
 \breve{P}_{2,k}(x)&=\left( \begin{array}{c|c}
                 \mathbb{I}+\Xi\mathbb{K}^{[k]} & \Xi  M[P_{2,k}]^{\top} \\
                 \hline
                 N[K^{[k]}(x,\cdot)]                      & P_{2,k}(x)
                \end{array} \right)\ .
\end{align}
Here the expression $M[K^{[k]}(\cdot,x)]$ $(N[K^{[k]}(x,\cdot)])$ stands for the action of the operator $M$ (respectively $N$), on the first (second) variable of $K$.
Alternatively the previous formulas can be rewritten in terms of the original polynomials as follows
\begin{align}
 \breve{P}_{1,k}(x)&=\label{yes}
 \left(\begin{array}{c|c} -N[P_{1,k}] \Xi \left(\mathbb{I}+\mathbb{K}^{[k]}\Xi\right)^{-1}\left(M\left[(P_2^{[k])^{\top}}\right]\right)^{\top}\left(H^{[k]}\right)^{-1}& 1 \end{array} \right)
 \begin{pmatrix}P_{1}^{[k]}(x)\\ \hline P_{1,k}(x) \end{pmatrix}\ ,\\
 \breve{P}_{2,k}(x)&=\label{no}
 \left(\begin{array}{c|c} \left( P_{2}^{[k]}(x) \right)^{\top} & P_{2,k}(x) \end{array}\right)
 \begin{pmatrix}  -\left(H^{[k]}\right)^{-1}N[P_1^{[k]}] \left(\mathbb{I}+\Xi\mathbb{K}^{[k]}\right)^{-1}\Xi M[P_{2,k}]^{\top}\\ \hline 1 \end{pmatrix}\ .
\end{align}
\end{pro}

\begin{proof}
Let us write the expression of the inverse of the matrix $\left(H+A\right)^{[k]}$. By using Definition \ref{mathKK}, one can check
%Firstly realize that definition \ref{mathKK} is motivated when inverting the matrix $\left(H+A\right)^{[k]}$. One can check that
the equalities
 \begin{align*}
 &\left[\left(H+A\right)^{[k]}\right]^{-1}=(H^{[k]})^{-1} \left[\left(\mathbb{I}+AH^{-1}\right)^{[k]}\right]^{-1}=
 (H^{[k]})^{-1} \left(\mathbb{I}+N[P_1^{[k]}]\Xi M[P_2^{[k]}]^{\top} (H^{[k]})^{-1} \right)^{-1} \\
 &=(H^{[k]})^{-1}\left(\mathbb{I}-N[P_1^{[k]}]\Xi M[P_2^{[k]}]^{\top} (H^{[k]})^{-1}+
 N[P_1^{[k]}]\Xi M[P_2^{[k]}]^{\top} (H^{[k]})^{-1}N[P_1^{[k]}]\Xi M[P_2^{[k]}]^{\top} (H^{[k]})^{-1}-\dots \right) \\
 &=(H^{[k]})^{-1}-(H^{[k]})^{-1}N[P_1^{[k]}]\Xi \left(\mathbb{I}-\mathbb{K}^{[k]}\Xi+(\mathbb{K}^{[k]}\Xi)^2-\dots \right)M[P_2^{[k]}]^{\top} (H^{[k]})^{-1} \ .
 \end{align*}
Consequently, we get the following expression, assuming that the formal series converges
 \begin{align} \label{convseries}
  \left[\left(H+A\right)^{[k]}\right]^{-1}=(H^{[k]})^{-1}-
  (H^{[k]})^{-1}N[P_1^{[k]}] \Xi \left( \mathbb{I}+\mathbb{K}^{[k]}\Xi \right)^{-1} M[P_2^{[k]}]^{\top}(H^{[k]})^{-1} \ .
 \end{align}
To prove the second statement, observe that
\begin{align*}
 \begin{pmatrix}
  A_{k,0} & A_{k,1} & \dots & A_{k,k-1}
 \end{pmatrix}&=N[P_{1,k}]\Xi M[P_2^{[k]}]^{\top}  &
 \begin{pmatrix}
  A_{0,k} \\ A_{1,k} \\ \vdots \\ A_{k-1,k}
 \end{pmatrix}&=N[P_{1}^{[k]}]\Xi M[P_{2,k}]^{\top}
\end{align*}
Once we substitute these expressions in the quasi-determinantal formulae given in proposition \ref{ASBPSNA}, we obtain the
relations \eqref{P12k}. The expressions \eqref{yes} and \eqref{no} follow from those in \eqref{P12k} by just expanding the quasi-determinants and
the CD kernels.
\end{proof}
\begin{remark}
Whenever the convergence of the series \eqref{convseries} is not fulfilled, no orthogonal polynomial sequences arises. This implies that the LU-factorization assumption for the moment matrix was not satisfied in the specific example considered
\end{remark}

%*********************************************************\\
%***********************************************************\\
%**************************************************************\\

Let us define the following polynomial, which will be useful in dealing with the additive discrete part of a bilinear Sobolev function.
\begin{definition} We define the auxiliary polynomial
\begin{align} \label{Wxx}
 W(x)&:=\prod_{i=1}^{s}(x-x_i)^{max\{n_i,m_i\}}
\end{align}
\end{definition}
The auxiliary polynomial \eqref{Wxx} is the keystone for the following result in concordance with \cite{NtermR} and slightly
generalizing \cite{MomentPrb}.
\begin{pro}
Given a non-Sobolev inner product $\langle *,* \rangle$, consider the bilinear form
\begin{align*}
 \left(f,h \right)_{\breve{\W}}&:=
 \langle f,h \rangle +\sum_{i=1}^{s} \sum_{k=0}^{n_i-1} \sum_{j=0}^{m_i-1} \xi^{(i)}_{k,j} h^{(k)}(x_i) f^{(j)}(x_i)
\end{align*}
obtained by adding a discrete Sobolev part to the original standard inner product. Then,
the SBPS associated with the new bilinear function $\left(*,* \right)_{\breve{\W}}$ satisfies a
$\big(2\left[\deg W(x) \right]+1\big)$-term recurrence relation, which in matrix form reads
\begin{align*}
 R_{\alpha} \breve{P}_{\alpha}(x)&=W(x)\breve{P}_{\alpha}    &       \alpha&=1,2 \ .
\end{align*}
Here $R_{\alpha}$ are $\big(2\left[\deg W(x) \right]+1\big)$ banded matrices, related to each other,
$R_1=\breve{H}R_2^{\top} \breve{H}^{-1}$, and can be written as
\begin{align*}
 R_{\alpha}&=M_{\alpha} W(J) M_{\alpha}^{-1} \ .
\end{align*}
This expression involves the connection matrices $M_{\alpha}P=\breve{P}_{\alpha}$, whose rows, according to \eqref{yes}, \eqref{no} read
\begin{align*}
 \left(\begin{array}{cccc|c} (M_1)_{k,0} & (M_1)_{k,1} & \dots & (M_1)_{k,k-1} & (M_1)_{k,k}  \end{array}\right)&=
  \left(\begin{array}{c|c} -N[P_{k}] \Xi \left(\mathbb{I}+\mathbb{K}^{[k]}\Xi\right)^{-1}\left(M\left[(P^{[k])^{\top}}\right]\right)^{\top}\left(H^{[k]}\right)^{-1}& 1 \end{array} \right)
 \\
 \left(\begin{array}{cccc|c} (M_2)_{k,0} & (M_2)_{k,1} & \dots & (M_2)_{k,k-1} & (M_2)_{k,k}  \end{array}\right)&=
 \begin{pmatrix}  -\left(H^{[k]}\right)^{-1}N[P^{[k]}] \left(\mathbb{I}+\Xi\mathbb{K}^{[k]}\right)^{-1}\Xi M[P_{2}]^{\top}\\ \hline 1 \end{pmatrix}^{\top}
\end{align*}
and the Jacobi matrix $J:=S \Lambda S^{-1}$
of the non perturbed initial inner product $\langle * ,* \rangle$ (responsible for their three term recurrence relation
$J P(x)=x P(x) $).
\end{pro}
\begin{proof}
It is  straightforward to see that
\begin{align*}
(Wf,h)_{\breve{\W}}=\langle Wf,h \rangle=\langle f, W h \rangle=(f,Wh)_{\breve{\W}} \ .
\end{align*}
Thus, the moment matrix satisfies
\begin{align*}
 W(\Lambda) \breve{G}=\breve{G}  W(\Lambda^{\top}) \ .
\end{align*}
Taking into account the $LU$ factorization of $\breve{G}$ and the definitions for the connection matrices the proposition follows.
\end{proof}

\section{Equivalence classes of measure matrices}\label{Eq}

A natural question arising from the theory previously developed is the following. Consider two measure matrices $\W_1(\Omega) \neq \W_2(\Omega)$ over the same $\Omega$. Assume that the equality $G_{\W_1}=G_{\W_2} $ holds, or equivalently $(p,q;\W_1)=(p,q;\W_2)$ $\forall p,q \in \R[x]$. Notice that, despite sharing the same moment matrix, and hence the
same SBPS, in principle $\A^{\mathcal{N}_1}_{\W_1}(\Omega)\neq \A^{\mathcal{N}_2}_{\W_2}(\Omega)$. %and therefore some pairings in the bilinear  function associated to $\W_1$ may not be even defined for the one associated to $\W_2$ and viceversa.
At the same time,
$\R[x]\in \A^{\mathcal{N}_1}_{\W_1}(\Omega) \cap \A^{\mathcal{N}_2}_{\W_2}(\Omega)$ and for every $f,g$ in this intersection, the
equality $(f,g;\W_1)=(f,g;\W_2)$ will hold. These considerations suggest to introduce the notion of equivalence class of measure.

\begin{definition}
 We shall say that two measure matrices $\W_{a}$ and $\W_{b}$ are equivalent, and we write $\W_{a}\sim \W_{b}$, if
 $(p,q;\W_{a})=(p,q;\W_{b})$ for every $p,q\in \R[x]$. We shall denote by $[\W_{a}]=\{\W_{b} \setminus \W_{b}\sim \W_{a} \}$ the equivalence class of measure matrices equivalent to a given matrix $\W_{a}$. Two matrices belonging to the same equivalence class
will be said similar.
\end{definition}
In other words, equivalent measure matrices share the same moment matrix. We will use the symbol $G_{[\W_a]}$ to denote the common moment matrix of a given equivalent class.

In this section we will address the equivalence problem, by showing how elements of the same matrix class are related. To this aim, we have to study preliminarily how a measure matrix changes under integration by parts manipulations.
Let us focus on the $(i,j)$-th entry of a given measure matrix and take it to be an absolutely continuous measure, this is,
$\d \mu_{i,j}(x)=\omega_{i,j}(x)\d x$.
We adopt the notation $ I\omega_{i,j}:=\mu_{i,j}$ for the anti-derivative or primitive of the absolutely continuous measure $\d \mu_{i,j}$.

Two possibilities arise.
\begin{itemize}
\item If $\omega_{i,j}\in \mathcal{C}^{1}(\Omega_{i,j})$, %(weak derivatives may also be another possibility),
\begin{align*}
 \noi \int_{\Omega_{i,j}} \chi^{(i)} \omega_{i,j} \left(\chi^{(j)}\right)^{\top} \d x=
 \begin{cases}
 \int_{\Omega_{i,j}} \chi^{(i-1)} \left[\delta \omega_{i,j}-\frac{\d \omega_{i,j}}{\d x} \right] \left(\chi^{(j)}\right)^{\top} \d x -
 \int_{\Omega_{i,j}} \chi^{(i-1)} \omega_{i,j} \left(\chi^{(j+1)}\right)^{\top} \d x \\
 \int_{\Omega_{i,j}} \chi^{(i)} \left[\delta \omega_{i,j}-\frac{\d \omega_{i,j}}{\d x} \right] \left(\chi^{(j-1)}\right)^{\top} \d x -
 \int_{\Omega_{i,j}} \chi^{(i+1)} \omega_{i,j} \left(\chi^{(j-1)}\right)^{\top} \d x
  \end{cases}
\end{align*}
\item For the primitive $\mu_{i,j}$,
\begin{align*}
  \int_{\Omega_{i,j}} \chi^{(i)} \d \mu_{i,j} \left(\chi^{(j)}\right)^{\top} =
  \int_{\Omega_{i,j}} \chi^{(i)} \delta I\omega_{i,j} \left(\chi^{(j)}\right)^{\top} \d x-
  \int_{\Omega_{i,j}} \chi^{(i+1)} I\omega_{i,j} \left(\chi^{(j)}\right)^{\top} \d x-
  \int_{\Omega_{i,j}} \chi^{(i)} I\omega_{i,j} \left(\chi^{(j+1)}\right)^{\top} \d x
\end{align*}

\end{itemize}
where we have introduced the operator ``$\delta$'' that turns the continuous measure into a discrete one on the boundary of its support
\begin{align*}
\int_{\Omega_{i,j}} \delta \omega_{i,j}(x) f(x)\d x&:=\left(\omega_{i,j}(x)f(x)\right)\biggr\rvert_{\partial \Omega_{i,j}}
\end{align*}
Therefore, we have found the relations among similar measure matrices that arise throughout integrations by parts manipulations.
\begin{pro} The following elementary transformations characterize an equivalent class of measure matrices:
\begin{align}\label{llech}
 \begin{pmatrix}
   \d \mu_{i-1,j-1}  & \d \mu_{i-1,j} & \d \mu_{i-1,j+1}   \\
   \d \mu_{i,j-1}    & \boldsymbol{\omega}_{i,j}\d x   & \d \mu_{i,j+1}   \\
   \d \mu_{i+1,j-1}  & \d \mu_{i+1,j} & \d \mu_{i+1,j+1}
 \end{pmatrix}\sim
 \begin{cases}
 \begin{pmatrix}
  \d \mu_{i-1,j-1} & \left(\d \mu_{i-1,j}-\left[\frac{\d \boldsymbol{\omega}_{i,j}}{\d x} \right]\d x\right) & \left(\d \mu_{i-1,j+1}-\boldsymbol{\omega}_{i,j}\d x\right)  \\
   \d \mu_{i,j-1}    & 0              & \d \mu_{i,j+1}   \\
   \d \mu_{i+1,j-1}  & \d \mu_{i+1,j} & \d \mu_{i+1,j+1}
 \end{pmatrix}+
 \begin{pmatrix}
    0 & \delta \boldsymbol{\omega}_{i,j}\d x & 0   \\
    0 & 0  & 0   \\
    0 & 0 & 0
 \end{pmatrix}\\
 \begin{pmatrix}
  \d \mu_{i-1,j-1} & \d \mu_{i-1,j} & \d \mu_{i-1,j+1}   \\
  \left(\d \mu_{i,j-1}-\left[\frac{\d \boldsymbol{\omega}_{i,j}}{\d x} \right]\d x\right) & 0 & \d \mu_{i,j+1}   \\
  \left(\d \mu_{i+1,j-1}-\boldsymbol{\omega}_{i,j}\d x\right) & \d \mu_{i+1,j} & \d \mu_{i+1,j+1}
 \end{pmatrix}+
 \begin{pmatrix}
   0 & 0 & 0   \\
   \delta \boldsymbol{\omega}_{i,j}\d x  & 0 & 0   \\
   0 & 0 & 0
 \end{pmatrix}\\
 \begin{pmatrix}
   \d \mu_{i-1,j-1} & \d \mu_{i-1,j} & \d \mu_{i-1,j+1}   \\
   \d \mu_{i,j-1} & 0  & \left(\d \mu_{i,j+1}-I\boldsymbol{\omega}_{i,j}\d x\right)   \\
   \d \mu_{i+1,j-1} & \left(\d \mu_{i+1,j}-I\boldsymbol{\omega}_{i,j}\d x\right) & \d \mu_{i+1,j+1}
 \end{pmatrix}+
 \begin{pmatrix}
    0 & 0 & 0    \\
    0 & \delta I\boldsymbol{\omega}_{i,j}\d x  & 0   \\
    0 & 0 & 0
 \end{pmatrix}
 \end{cases}
\end{align}
\end{pro}
Iterations of these transformations are obviously allowed. Notice the split between the continuous and discrete parts.
The previous transformations can be performed over every entry $i,j$ in the measure matrix as long as $\omega_{i,j}$
can be derived or integrated. This leads to a huge amount of equivalent matrices in $[\W_{a}]$.

\begin{align*}
 \begin{pmatrix}
   &* &* &*  &*        &*  &*  &*  & * &      \\
   &* &* &*  &*        &*  &*  & * &   &  \\
   &* &* &*  &*        &*  &*  &   &   &  \\
   &* &* &*  &*        & * &   &   &   &  \\
   &* &* &*  &\bigstar & * &*  &*  & * &  \\
   &* &* &*  &*        & * &*  &*  & * &  \\
   &* &* &   &*        & * &*  &*  &*  &  \\
   &* &  &   &*        & * &*  &*  &*  &  \\
   &  &  &   &*        & * &*  &*  &*  &  \\
   &  &  &   &*        & * &*  &*  &*  &
 \end{pmatrix}
\end{align*}

Upper (lower) anti-diagonal terms come from taking derivatives (integrals) of the entry in the star location. \\

Each iteration produces a discrete term. Once gathered together in a matrix, these terms will define an additive discrete perturbation of the measure matrix.

According to the previous discussion, if we can obtain the SBPS associated to the continuous part, the SBPS associated to the
whole bilinear form can also be obtained with the aid of the CD kernels of the continuous part. However, under certain conditions imposed
on the $\omega_{i,j}$ one can get rid of the discrete part.\\
\begin{definition}\label{domeg}
 Let us denote by $\tilde{\omega}_k$ any weight with finite moments on $\Omega$ having the following property	
 \begin{align} \label{omegaprop}
         \delta \tilde{\omega}_{k}^{(t)}&=0  &
         t&=0,1,2,\dots,(k-1)
 \end{align}
\end{definition}
According to proposition \ref{ClSSW}, the classical measure $u_{\gamma+k}$ is a particular example of $\tilde{\omega}_k$.
\begin{pro}\label{tildet}
Let $\W$ be a $(\mathcal{N}+1) \times (\mathcal{N}+1)$ measure matrix such that $\d \mu_{i,j}=\omega_{i,j}\d x$
and each $\omega_{i,j}$ is a function of class $\mathcal{C}^{|i-j|}$.
\begin{itemize}
\item If $\W=\W^{\top}$ then $\W$ is similar to the sum of a diagonal measure matrix and a discrete measure matrix.\\
\item If $\W=\W^{\top}$ and additionally each entry $\omega_{i,j}$ can has the property of $\tilde{\omega}_{|i-j|} $ $\forall i,j$,
as in \eqref{omegaprop} then $\W$ is similar to a diagonal measure matrix.
\end{itemize}
\end{pro}

\begin{proof}
We shall formulate an inductive procedure to prove the first statement of the proposition. Given any $(\mathcal{N}+1)\times(\mathcal{N}+1)$ symmetric measure matrix $\W$
\begin{align*}
 \W=\begin{pmatrix}
  \omega_{0,0}        & \omega_{1,0}     & \omega_{2,0}  &  \dots & \omega_{\mathcal{N}-1,0}    &  \omega_{\mathcal{N},0}     \\
  \omega_{1,0}    & \omega_{1,1}         & \omega_{2,1}  &  \dots & \omega_{\mathcal{N}-1,1}    &  \omega_{\mathcal{N},1} \\
  \omega_{2,0}    & \omega_{2,1}     & \omega_{2,2}      &  \dots &                           &                       \\
                      &                      &                   &        &   \vdots                  &  \vdots                \\
       \vdots         &      \vdots          &    \vdots         & \ddots &                           &                        \\
\omega_{\mathcal{N}-1,0}&\omega_{\mathcal{N}-1,1}&    \dots          &        &  \omega_{\mathcal{N}-1,\mathcal{N}-1}         &\omega_{\mathcal{N},\mathcal{N}-1}  \\
  \omega_{\mathcal{N},0}    &\omega_{\mathcal{N},1}  &    \dots          &        &\omega_{\mathcal{N},\mathcal{N}-1}       &  \omega_{\mathcal{N},\mathcal{N}}
 \end{pmatrix}\d x
\end{align*}
one can use the first similarity relation stated in \eqref{llech} for each entry of the last row of $\W$ and the second one for each entry
of its last column. In this way, one obtains
\begin{align*}
 \W &\sim \begin{pmatrix}
  \omega_{0,0}        & \omega_{1,0}     & \omega_{2,0}  &  \dots &\omega_{\mathcal{N}-1,0}-\omega_{\mathcal{N},0}^{'}    &  0                   \\
  \omega_{1,0}    & \omega_{1,1}         & \omega_{2,1}  &  \dots &\omega_{\mathcal{N}-1,1}-\omega_{\mathcal{N},0}-\omega_{\mathcal{N},1}^{'}    &  0                    \\
  (\omega_2)_{2,0}    & (\omega_1)_{2,1}     & \omega_{2,2}      &  \dots &                           &                       \\
                      &                      &                   &        &   \vdots                  &  \vdots                \\
       \vdots         &      \vdots          &    \vdots         & \ddots &                           &                        \\
\omega_{\mathcal{N}-1,0}-\omega_{\mathcal{N},0}^{'}&\omega_{\mathcal{N}-1,1}-\omega_{\mathcal{N},0}-\omega_{\mathcal{N},1}^{'}&    \dots     &        &\omega_{\mathcal{N}-1,\mathcal{N}-1}-2\omega_{\mathcal{N},\mathcal{N}-2}-\omega_{\mathcal{N},\mathcal{N}-1}^{'}         &         0              \\
  0                   &    0             &    \dots          &             &      0                    &  \omega_{\mathcal{N},\mathcal{N}}
 \end{pmatrix}\d x\\ \\ &+
 \begin{pmatrix}
  0    & 0     & 0  &  \dots & \delta\omega_{\mathcal{N},0}   &  0                   \\
  0    & 0     & 0  &  \dots &\delta\omega_{\mathcal{N},1}     &  0                    \\
  0    & 0     & 0  &  \dots &                           &                       \\
                      &                      &                   &        &   \vdots                  &  \vdots                \\
       \vdots         &      \vdots          &    \vdots         & \ddots &                           &                        \\
\delta\omega_{\mathcal{N},0} &\delta\omega_{\mathcal{N},1}  &    \dots          &        &\delta\omega_{\mathcal{N},\mathcal{N}-1}        &         0              \\
  0                   &    0             &    \dots          &             &      0                    &  0
 \end{pmatrix}\d x
\end{align*}
This new, equivalent measure matrix is still symmetric. Therefore,
the whole procedure can be repeated up to $\mathcal{N}$ times, until the diagonal form is achieved, jointly with the discrete terms that will
appear each time.\\
The second statement of the proposition is just a corollary of the first one since the definition \ref{domeg} is suited to make the discrete
terms disappear.
\end{proof}

Let us consider the example $\mathcal{N}=3$.

\begin{align*}
%i) \\
% &\begin{pmatrix}
% \omega_{0,0} & \omega_{1,0} \\
%  \omega_{1,0} & \omega_{1,1}
% \end{pmatrix}\sim
%\begin{pmatrix}
%  \omega_{0,0}-\omega_{1,0}^{'} & 0 \\
%  0 & \omega_{1,1}
% \end{pmatrix}+
% \begin{pmatrix}
%  \delta\omega_{1,0} & 0 \\
%  0 & 0
% \end{pmatrix} \\
%ii) \\
% &\begin{pmatrix}
%  \omega_{0,0} & \omega_{1,0} & \omega_{2,0}\\
%  \omega_{1,0} & \omega_{1,1} & \omega_{2,1} \\
%  \omega_{2,0} & \omega_{2,1} & \omega_{2,2}
% \end{pmatrix}\sim
%\begin{pmatrix}
%  \omega_{0,0}-\omega_{1,0}^{'}+\omega_{2,0}^{''} & 0 & 0\\
%  0 & \omega_{1,1}-\omega_{2,1}^{'}-2\omega_{2,0} & 0 \\
%  0 & 0 & \omega_{2,2}
% \end{pmatrix}
% +\begin{pmatrix}
%  \delta[\omega_{1,0}-\omega_{1,0}'] & \delta\omega_{2,0} & 0\\
%  \delta\omega_{2,0} & \delta\omega_{2,1} & 0 \\
%  0 & 0 & 0
% \end{pmatrix} \\
%iii) \\
 &\begin{pmatrix}
  \omega_{0,0} & \omega_{1,0} & \omega_{2,0} & \omega_{3,0}\\
  \omega_{1,0} & \omega_{1,1} & \omega_{2,1} & \omega_{3,1}\\
  \omega_{2,0} & \omega_{2,1} & \omega_{2,2} & \omega_{3,2} \\
  \omega_{3,0} & \omega_{3,1} & \omega_{3,2} & \omega_{3,3}
 \end{pmatrix}\sim \\
 &\begin{pmatrix}
  \omega_{0,0}-\omega_{1,0}^{'}+\omega_{2,0}^{''}+\omega_{3,0}^{'''} & 0 & 0 & 0\\
  0 & \omega_{1,1}-\omega_{2,1}^{'}+\omega_{3,1}^{''}-2\omega_{2,0}+3\omega_{3,0}^{'} &  0 & 0\\
  0 & 0 & \omega_{2,2}-\omega_{2,3}^{'}-2\omega_{3,1} & 0 \\
  0 & 0 & 0 & \omega_{3,3}
 \end{pmatrix} \\
 &+\begin{pmatrix}
  \delta[\omega_{1,0}-\omega_{2,0}'-\omega_{3,0}''] & \delta[\omega_{2,0}-\omega_{3,0}'] &  \delta\omega_{3,0} & 0\\
  \delta[\omega_{2,0}-\omega_{3,0}'] & \delta[\omega_{2,1}-\omega_{3,0}-\omega_{3,1}'] &   \delta\omega_{3,1} & 0\\
   \delta\omega_{3,0} &  \delta\omega_{3,1} &  \delta\omega_{3,2} & 0 \\
  0 & 0 & 0 & 0
 \end{pmatrix}
\end{align*}

\subsection{Sobolev inner products involving classical measures}
When dealing with classical measures $u_{\gamma}$, the construction of equivalence classes of measure matrices appears
to be particularly simple and neat. The reason resides in the possibility of generating equivalence classes
without having to deal with any discrete parts (boundary terms). We summarize this properties in the next
\begin{pro}\label{canana}
Let
\begin{align*}
 \W=\begin{pmatrix}
     \omega^0_1  & 0  & \dots   &   \\
     0           & 0  &   &   \\
     \vdots      &    &\ddots   &   \\
                 &    &   &
    \end{pmatrix}u_{\gamma+0}\d x+
    \begin{pmatrix}
     0           & \omega^1_2 & 0     & \dots   \\
     \omega^1_2  & \omega^1_1 &   &   \\
     0           &            & 0   &   \\
      \vdots     &            &    & \ddots
    \end{pmatrix}u_{\gamma+1}\d x+\dots+
    \begin{pmatrix}
     0             & \dots       & 0     & \omega^n_{n+1}  \\
     \vdots        & \ddots      &       &  \omega^n_{n} \\
     0             &             &  0    & \vdots   \\
     \omega^n_{n+1}&\omega^n_{n} & \dots & \omega^n_{1}
    \end{pmatrix}u_{\gamma+n}\d x
\end{align*}
be a measure matrix such that each $\{\omega_j^{r}u_{\gamma+r}\}_{j=1}^{r+1}$ is of type $\tilde{\omega}_r$ $\forall r=0,1,\dots,n$. Then if
$\W$ determines a $SOPS$, then there exist
linear differential operator $\mathbf{F}$ and constants $\{\alpha_{k,j},\beta_{k,j}\}$ such that
\begin{align*}
 \left(f,h;\W \right)&=\langle \mathbf{F}[f]h,u_{\gamma} \rangle=
 \langle f\mathbf{F}[h],u_{\gamma} \rangle  & \Longrightarrow &  &
 \mathbf{F}[P_{\W,k}]&=\sum_{j=k}^{r}\alpha_{k,j}P_{\gamma,j} \\
 \left(\mathbf{F}[f],h;\W \right)&=\left(f,\mathbf{F}[h];\W \right) & \Longrightarrow &  &
 \mathbf{F}[P_{\W,k}]&=\sum_{j=k-r}^{k+r}\beta_{k,j}P_{\W,j}
\end{align*}
\end{pro}

\begin{proof}
Since the selected measure matrix $\W$ satisfies the conditions in
 proposition \ref{tildet}, using also proposition \ref{ClSSW} it is not hard to see that
 \begin{align*}
  \W&\sim \begin{pmatrix}
          v_0 u_{\gamma} &                &       &       \\
                         &v_1 u_{\gamma+1}&       &       \\
                         &                &\ddots &       \\
                         &                &       & v_n u_{\gamma+n}
         \end{pmatrix}&   &\Longrightarrow   &
         \left(f,h;\W\right)&=\sum_{r=0}^{n}\langle f^{(r)}h^{(r)},v_ru_{\gamma+r} \rangle
 \end{align*}
 where the $\{v_r\}_{r=0}^{n}$ are functions that depend on the $\omega$ and their derivatives and
 $v_r u_{\gamma+r}$ are of type $\tilde{\omega}_r$ due to the conditions that the proposition imposes on the $\omega^r_j u_{\gamma+r}$.
 Using proposition \ref{ClSSW} (in which the operator $\mathcal{O}_{r}^{j}$ was defined)
 for the $r$-th term of the sum, the following chain of equalities follow
 \begin{align*}
 \langle f^{(r)}h^{(r)},v_ru_{\gamma+r} \rangle=
 (-1)^r \sum_{j=0}^{r} \begin{pmatrix}r \\j\end{pmatrix}\langle h^{(0)} (f^{(r)})^{(j)},\left(v_r u_{\gamma+r} \right)^{(r-j)}\rangle=
 (-1)^r \sum_{j=0}^{r} \begin{pmatrix}r \\j\end{pmatrix}\langle h^{(0)} \left(\mathcal{O}_r^{j+1}[v_r]p_2^{j}\right)f^{(r+j)},u_{\gamma}\rangle
 \end{align*}
 this has to be added for each $r$, after doing so the differential operator $\mathbf{F}$ can be finally defined as
 \begin{align*}
  \mathbf{F}:=\sum_{r=0}^{n}(-1)^r \sum_{j=0}^{r} \begin{pmatrix}r \\j\end{pmatrix}
  \left(\mathcal{O}_r^{j+1}[v_r](x)p_2^{j}(x) \right)\frac{\d^{r+j}}{\d x^{r+j}} \ .
 \end{align*}
Consequently, $\left(f,h;\W\right)=\langle h \mathbf{F}[f],u_{\gamma}\rangle $. By acting on $h$ instead of $f$, we also get $\left(f,h;\W\right)=\langle \mathbf{F}[h] f,u_{\gamma}\rangle$. As we already
 know, this equalities can be translated to relations between the moment matrices
 \begin{align*}
 \left(f,h;\W \right)&=\langle \mathbf{F}[f]h,u_{\gamma} \rangle=
 \langle f\mathbf{F}[h],u_{\gamma} \rangle  & \Longrightarrow &  &
 G_{\W}=Fg_{\gamma}=g_{\gamma}F^{\top}
 \end{align*}
 where the matricial representation of $\mathbf{F}$ is,
 $F:=\sum_{r=0}^{n}(-1)^r \sum_{j=0}^{r} \begin{pmatrix}r \\j\end{pmatrix}D^{r+j}
  \left(\mathcal{O}_r^{j+1}[v_r](\Lambda)p_2^{j}(\Lambda) \right)$. Now if $\W$ gives a SOPS then
  $G_\W$ must be LU factorizable, i.e.,
  \begin{align*}
   G_{\W}&=Fg_{\gamma}   & \Longrightarrow &  &
   U&:=S_{\W}F S_{\gamma}^{-1}=H_{\W}\left(S_{\gamma} S_{\W}^{-1}\right)^{\top}H_{\gamma}^{-1} & \Longrightarrow & &
   U P_{\gamma}&=\mathbf{F}[P_{\W}]
  \end{align*}
  The second set of equations imposes  an upper triangular form to $U$,  with a finite number $r$ of non vanishing super-diagonal terms only,
  that will depend on the differential operator.\\
  Multiplying the relation between moment matrices by $F$ and $F^{\top}$ and LU factorizing once more one obtains
  \begin{align*}
   FG_{\W}&=G_{\W}F^{\top}   & \Longrightarrow &  &
   J_{F}&:=S_{\W}F S_{\W}^{-1}=H_{\W}J_F^{\top}H_{\W}^{-1} & \Longrightarrow & &
   J_{F} P_{\W}&=\mathbf{F}[P_{\W}]
  \end{align*}
  This time, the second set of relations imposes a $2r+1$ diagonal structure to $J_F$ ($2r$ non vanishing diagonals above and below the main one).
\end{proof}

Some comments are in order.
\begin{itemize}
 \item The initial condition $\{\omega_j^{r}u_{\gamma+r}\}_{j=1}^{r+1}$ being of type $\tilde{\omega}_r$ is not so restrictive since
 $u_{\gamma+r}$ already is of type $\tilde{\omega}_r$. So the $\omega_j^{r}$ just must not spoil this property.
 \item A particularly simple example is to consider $\omega_j^{r}=0$ $\forall j \neq 1$ in which case $\omega^{r}_1=v_r$. Taking
 now $v_r=p_2^{n-r}\lambda_r$ with $\lambda_r>0$ $\forall r=0,1,\dots,n$ one is left with the following inner product and corresponding linear differential operator
 \begin{align*}
  \left(f,h;\W\right)&=\sum_{r=0}^{n}\lambda_r \langle f^{(r)} h^{(r)}, u_{\gamma+n} \rangle= \langle \mathbf{F}[f]h,u_{\gamma}\rangle  &
  \mathbf{F}=\sum_{r=0}^{n}(-1)^r \lambda_r \sum_{j=0}^{r} \begin{pmatrix}r \\j\end{pmatrix}
  \varphi_{n,r-j}(x)\frac{\d^{r+j}}{\d x^{r+j}}
 \end{align*}
 Although with small differences (the starting inner product's measure and the way the Pearson equation is used),
 this example is in agreement with the main ideas in \cite{GenMarc} and \cite{LagMarc}.
\end{itemize}

\section{Polynomial deformations of the measure matrix}

As we have seen, moment matrices arising from a diagonal $\W$ with positive definite measures (also symmetric $\W$ reducible to a diagonal shape)
are examples of Sobolev $LU$-factorizable moment matrices. In this section, we investigate deformations of a given factorizable case,
with the idea of exploring the possibility of new factorizable ones.
The deformations of the measure matrix we are interested in can be understood as deformations of the moment matrix,
which naturally translate into transformations of the associated bilinear form.
These transformations of the bilinear form are expressed in terms of linear differential operators
acting on each of its entries, but before studying the general case, we will start with the more simple and usual case of
deformations involving polynomials.\\

On one hand, in the standard case (corresponding to $\omega_{n,r}=0 \,\,\,\forall n,r>0$, which gives $(f,h;\W)=\langle f,h \rangle_{\omega_{0,0}}$)
we have the symmetry $\langle xf,h \rangle_{\omega_{0,0}}=\langle f,xh \rangle_{\omega_{0,0}}=\langle f,h \rangle_{x\omega_{0,0}}$.
This symmetry  is responsible, for instance, for the three term recurrence relation of the OPS, or the Hankel shape of the moment matrix.

On the other hand, given a measure matrix $\W$, in general $(xf,h;\W)\neq (f,xh;\W)\neq (f,h;x\W)$.
However, we can equivalently say that there exist new measure matrices $\W_2,\W_3$
such that $(xf,h;\W_1)=(f,h;\W_2)$ and $(f,xh;\W_1)=(f,h;\W_3)$. While multiplication of any of
the entries of the standard inner product by a polynomial produces another standard inner product, instead, the same operation in any
of the entries of a Sobolev-type bilinear function deforms the initial $\W$ giving a different one, probably spoiling the symmetries of
$\W$ if it had any.
%To understand this we need to make the following
%\begin{definition}
% \begin{align*}
%  \mathcal{X}:=\begin{pmatrix}
%   x & 1 & 0 & 0 & \dots\\
%   0 & x & 2 & 0 &\dots\\
%   0 & 0 & x & 3 &\dots\\
%   0 & 0 & 0 & x &\dots\\
%  \vdots&\vdots& \vdots& \vdots & \ddots
%\end{pmatrix}
% \end{align*}
%\end{definition}
%That allow us to state the next
\begin{theorem}\label{te1}
The operator $\mathcal{X}$ of multiplication by $x$,
 \begin{align*}
  \mathcal{X}:=\begin{pmatrix}
   x & 1 & 0 & 0 & \dots\\
   0 & x & 2 & 0 &\dots\\
   0 & 0 & x & 3 &\dots\\
   0 & 0 & 0 & x &\dots\\
  \vdots&\vdots& \vdots& \vdots & \ddots
\end{pmatrix}
 \end{align*}
once applied to any of the entries of a Sobolev bilinear function, provides the following deformation of the measure matrix
\footnote{Being the initial moment matrix $G_{\W}$ a $LU$-factorizable moment matrix does not imply the new moment matrix
$G_{\mathcal{X}\W}$ to be $LU$-factorizable as well.}
 \begin{align*}
  (xf,h;\W)&=(f,h;\mathcal{X}\W)  &   \Lambda G_{\W}&= G_{\mathcal{X}\W} \\
  (f,xh;\W)&=(f,h;\W (\mathcal{X})^{\top})  &   G_{\W} \Lambda^{\top}&= G_{\W (\mathcal{X})^{\top}}
 \end{align*}
\end{theorem}

\begin{proof}
Using the definition of the moment matrix and taking into account the commutation relations between $D^k$ and $\Lambda$, we get
\begin{align*}
 \Lambda G_{\W}& = \Lambda \boldsymbol{D} \left( \int_{\Omega} \chi(x) \W \chi(x)^\top  \right) \boldsymbol{D}^\top=
 \begin{pmatrix}
                        \Lambda\mathbb{I} & \Lambda D & \Lambda D^2 & \dots & \Lambda D^k & \dots
                       \end{pmatrix}
        \left( \int_{\Omega} \chi(x) \W \chi(x)^\top  \right) \boldsymbol{D}^\top=\\
        & \begin{pmatrix}
                        \Lambda & D\Lambda+\mathbb{I} &  D^2 \Lambda+2D & \dots & D^k\Lambda+kD^{k-1} & \dots
                       \end{pmatrix}
        \left( \int_{\Omega} \chi(x) \W \chi(x)^\top  \right) \boldsymbol{D}^\top= \\
        &\boldsymbol{D}  \begin{pmatrix}
   \Lambda & \mathbb{I} & 0 & 0 & \dots\\
   0 & \Lambda & 2\mathbb{I} & 0 &\dots\\
   0 & 0 & \Lambda & 3\mathbb{I} &\dots\\
   0 & 0 &  & \Lambda &\dots\\
  \vdots&\vdots& \vdots& \vdots & \ddots
\end{pmatrix}
                   \left( \int_{\Omega} \chi(x) \W \chi(x)^\top  \right) \boldsymbol{D}^\top=
      \boldsymbol{D}
                   \left( \int_{\Omega} \chi(x)
                   \begin{pmatrix}
   x & 1 & 0 & 0 & \dots\\
   0 & x & 2 & 0 &\dots\\
   0 & 0 & x & 3 &\dots\\
   0 & 0 &  & x &\dots\\
  \vdots&\vdots& \vdots& \vdots & \ddots
\end{pmatrix}
                   \W \chi(x)^\top  \right) \boldsymbol{D}^\top \ .
\end{align*}
\end{proof}
We can generalize the previous argument. First, let us compute the powers of $\mathcal{X}$.
Then one can observe that $\mathcal{X}^k $ is an upper triangular banded matrix, whose entries for $n=1,2,\dots$ are
\begin{align*}
 (\mathcal{X}^k)_{(n-1),(n-1)+i}&=\begin{pmatrix}
                          k \\
                          i
                         \end{pmatrix} (n)^{i} x^{k-i} &
                        0\leq i\leq k \\
 (\mathcal{X}^k)_{(n-1),(n-1)+i}&=0   & i > k
\end{align*}
In addition, due to the bilinearity of the function, we obtain the following
\begin{pro}\label{pro2}
Given two real polynomials $P(x)$ and $Q(x)$, the relations
\begin{align*}
 (P(x)f,Q(x)h;\W)& =(f,h; P(\mathcal{X})\W [Q(\mathcal{X})]^{\top})  &
  P(\Lambda)G_{\W}\left(Q(\Lambda)^{\top}\right)&= G_{P(\mathcal{X})\W \left( Q(\mathcal{X})^{\top}\right)}
\end{align*}
hold. If $deg\{P(x)\}=k$, then $P(\mathcal{X})$ is an upper triangular matrix whose entries are
\begin{align*}
 (P(\mathcal{X}))_{(n-1),(n-1)+i}& =\begin{cases} \frac{(n)^{i}}{i!} \frac{\d^i p(x)}{\d x^i} & 0 \leq i\leq k \\
                                      0  & i > k  \end{cases}  &
 P(\mathcal{X})&=\begin{pmatrix}
                 P(x) & P'(x) & P''(x) & P'''(x) & \dots &     \\% \begin{pmatrix} k \\ k \end{pmatrix}P^{(k)}(x)             &                                                 &                                                 &                                                 &         &                    \\
                      & P(x)  & 2P'(x) & 3P''(x) & \dots &     \\%                                                       & \begin{pmatrix} k+1 \\ k \end{pmatrix}P^{(k)}(x)&                                                 &                                                 &         &                     \\
                      &       & P(x)   & 3P'(x)  & \dots &     \\%                                                       &                                                 & \begin{pmatrix} k+2 \\ k \end{pmatrix}P^{(k)}(x)&                                                 &         &                     \\
                      &       &        & P(x)    & \dots &     \\%                                                       &                                                 &                                                 & \begin{pmatrix} k+3 \\ k \end{pmatrix}P^{(k)}(x)&         &                     \\
                      &       &        &         &\ddots &     \\%                                                       &                                                 &                                                 &                                                 & \ddots  &                     \\
                      &       &        &         &       &      %                                                      &                                                 &                                                 &                                                 &         &\begin{pmatrix} k+(n-1) \\ k \end{pmatrix}P^{(k)}(x)
                \end{pmatrix} \ .
\end{align*}
Thus, if $\W$ is a $(\mathcal{N}+1) \times (\mathcal{N}+1)$ measure matrix, then $P(\mathcal{X})\W [Q(\mathcal{X})]^{\top}$ will still
be a $(\mathcal{N}+1) \times (\mathcal{N}+1)$ measure matrix.
\end{pro}

The interest of the latter proposition relies on the fact that, although in principle there is no reason why
$G_{P(\mathcal{X})\W \left( Q(\mathcal{X})^{\top}\right)}$ should be $LU$-factorizable if $G_{\W}$ is so, there will be important cases, that we
we are about to study, where equations like the one in the right hand side of the proposition
will lead to relations between the SBPS associated to the deformed and non deformed measure matrices. Therefore, this proposition will be
keystone in order to
study a special case where the standard three term recurrence relation holds and to
generalize the concept of Darboux transformations \cite{GerDarb} to the Sobolev context.

\subsection{A special case where the standard three term recurrence relation holds}

As we have already pointed out, given an arbitrary measure matrix $\W$, in general  $(xf,h;\W)\neq(f,xh;\W)$. However, if we impose some
additional symmetry on $\W$, or we specialize it conveniently, we may get the desired equality.
\begin{definition}
We introduce the set of matrices
\begin{align*}
\W_x:=\{\W \setminus \mathcal{X}\W \sim \W \mathcal{X}^{\top}\}.
 \end{align*}
\end{definition}
\begin{theorem}
If $\W \in \W_x$ then $G_{\W}$ is Hankel and the associated SOPS satisfy the \textbf{standard} three term recurrence relation
\begin{align*}
  xP_n&=J_{n,n-1}P_{n-1}+J_{n,n}P_n+P_{n+1}  &  J_{n,n-1}&=\frac{h_{n}}{h_{n-1}} & J_{n,n}&=S_{n,n-1}-S_{n+1,n}
 \end{align*}
\end{theorem}

%The reader should notice that this theorem does not exclude the possible existence of some SBPS associated to a $\W \notin \W_x$ having a three term recurrence relation. What it states is that when $\W\in \W_x$ the recurrence relation terms are the \textbf{standard} ones, i.e (for monic orthogonal polynomials)
% \begin{align*}
%  xP_n&=J_{n,n-1}P_{n-1}+J_{n,n}P_n+P_{n+1}  &  J_{n,n-1}&=\frac{h_{n}}{h_{n-1}} & J_{n,n}&=S_{n,n-1}-S_{n+1,n}
% \end{align*}
\begin{proof}
 The condition $\mathcal{X}\W \sim \W \mathcal{X}^{\top}$,
 due to Theorem \ref{te1} is equivalent to
 $\Lambda G_{\W}=G_{\mathcal{X}\W}=G_{\W\mathcal{X}^{\top}}=G_{\W} \Lambda^{\top}$. This symmetry of the moment matrix leads
 to its Hankel shape and allows to construct the well known tri-diagonal Jacobi matrix ($J:=S\Lambda S^{-1}$) with its entries in terms of the elements of $S,h$.
 Note also that if $\W\in \W_x$ then $\mathcal{X}\W \in \W_x$ as well.
\end{proof}

\begin{theorem}
$\W_x$ is not an empty set.
\end{theorem}
\begin{proof}
We give here the following counterexample
 \begin{align}\label{Wx}
\W=\begin{pmatrix}
   \d\mu_{0}                                    & \begin{pmatrix}1 \\ 0 \end{pmatrix}\d\mu_{1} & \begin{pmatrix}2 \\ 0 \end{pmatrix}\d\mu_{2} & \begin{pmatrix}3 \\ 0 \end{pmatrix}\d\mu_{3} &\dots & \begin{pmatrix}\mathcal{N} \\ 0 \end{pmatrix}\d\mu_{\mathcal{N}}\\
   \begin{pmatrix}1 \\ 1 \end{pmatrix}\d\mu_{1} & \begin{pmatrix}2 \\ 1 \end{pmatrix}\d\mu_{2} & \begin{pmatrix}3 \\ 1 \end{pmatrix}\d\mu_{3} &                                              &      &  0   \\
   \begin{pmatrix}2 \\ 2 \end{pmatrix}\d\mu_{2} & \begin{pmatrix}3 \\ 2 \end{pmatrix}\d\mu_{3} &                                              &                                              &      &     \\
   \begin{pmatrix}3 \\ 3 \end{pmatrix}\d\mu_{3} &                                              &                                              &                                              &      &   \\
   \vdots                                       &\begin{pmatrix}\mathcal{N} \\\mathcal{N}-1 \end{pmatrix}\d\mu_{\mathcal{N}} &                                              &                                              &      &     \\
   \begin{pmatrix}\mathcal{N} \\ \mathcal{N} \end{pmatrix}\d\mu_{\mathcal{N}} &      0                                       &                                              &                                              &      &
\end{pmatrix}\in \W_x
 \end{align}
which is obtained by imposing $\mathcal{X}\W=\W \mathcal{X}^{\top}$ and yiels the following Sobolev inner product
\begin{align*}
 (f,h;\W)=\sum_{n=0}^{\mathcal{N}} \int_{\Omega_n} (fh)^{(n)} \d \mu_n(x)
\end{align*}
\end{proof}
For a SOPS associated with a measure matrix in $\W_x$, all of the results of the standard theory of orthogonal polynomial sequences hold: Three term
recurrence relation, Christoffel-Darboux formulae, the existence of $\tau$-functions, of associated integrable hierarchies, etc.. All of these properties are indeed a non-trivial consequence of
the symmetry $\Lambda G_{\W}=G_{\W} \Lambda^{\top}$.

A natural question is the relation between the previous result and the classical Favard theorem.  Essentially, Favard's theorem assures that given a set of polynomials,
satisfying certain initial conditions and a standard three term recurrence relation, there exists a measure $\mu$ with respect to which the set of polynomials is actually an OPS.
The SOPS associated to a $\W_x$ indeed satisfy the hypotheses of Favard's theorem.
Therefore, from both results we deduce that there must exist a measure $\d\mu$ such that $\d\mu E_{00}\sim \W_x$.
(Remember that similar measure matrices shared both the moment matrix and the orthogonal polynomial sequence).%The challenge is to find this $\mu$.

Let us consider a particular case of the given counterexample. Let us take $\mathcal{N}=1$ and $\Omega_n:=[x_1,x_2]$ for $n=0,1$ and,
using the iterations of Proposition \ref{llech} it is not hard to see that (at least in the function spaces where
the corresponding integration by parts makes sense)
\begin{align*}
 \begin{pmatrix}
  \d \mu_0   &  \d \mu_{1} \\
  \d \mu_{1} &  0
 \end{pmatrix}\sim
 \begin{pmatrix}
  \d \mu_0+(\d \mu_1)' &  0 \\
  0 &  0
 \end{pmatrix}+
 \begin{pmatrix}
  \delta \d \mu_{1}    &  0\\
  0          &  0
 \end{pmatrix}= \left[\d \mu_0+(\d \mu_1)'+\delta \d \mu_{1} \right] E_{00}
\end{align*}

\subsection{Darboux--Sobolev tranformations and quasi-recurrence relations}\label{Conflict}

In the next three sections we will proceed, with the aid of proposition \ref{pro2}
to deform the measure matrix by means of a (right or left) multiplication by a polynomial in
$\mathcal{X}$ or its inverse. Subsequently, we shall study the relation between the new and old SBPS associated to the deformed and non deformed
measure matrices, respectively. The reason for the name of these deformations is that whenever $\W=E_{0,0}\omega$ (the ``standard" case),
then our deformations reduce to the ``standard" Darboux transformations or linear spectral transformations. As already noticed in the
introduction, this section adapts and completes, for this particular Sobolev scalar case, the more general results given in Ref. \cite{Car3Tr}.

\subsection{Christoffel--Sobolev transformations}

Let us introduce the polynomial $R(x):=\prod_{i=1}^{d}(x-r_i)^{m_i}$  of degree $\sum_{i=1}^{d}m_i=M$.
\begin{definition}
The right and left Christoffel--Sobolev deformed measure matrices and moment matrices are
 \begin{align*}
  \hat{\W}_L& :=R(\mathcal{X})\W  & \hat{\W}_R& :=\W [R(\mathcal{X})]^{\top} \\
  R(\Lambda)G_{\W}&=G_{\hat{\W}_L}:=\hat{G}_L  &
  G_{\W}[R(\Lambda)]^{\top}&=G_{\hat{\W}_R}:=\hat{G}_R
 \end{align*}
The resolvents and adjoint resolvents are defined as
\begin{align*}
 (\hat{\omega}_L)&:=(\hat{S}_{L1})R(\Lambda)S_1^{-1}  &   (\hat{\Omega}_L)&:=S_2 (\hat{S}_{L2})^{-1} \\
 (\hat{\omega}_R)&:=(\hat{S}_{R2})R(\Lambda)S_2^{-1}  &   (\hat{\Omega}_R)&:=S_1 (\hat{S}_{R1})^{-1}
\end{align*}
\end{definition}

\begin{pro}
 The resolvents are related to the adjoint resolvents by the formulae
 \begin{align*}
  (\hat{\omega}_L)&=(\hat{H}_L)(\hat{\Omega}_L)^{\top}H^{-1} &  (\hat{\omega}_R)&=(\hat{H}_R)(\hat{\Omega}_R)^{\top}H^{-1}
 \end{align*}
 and have the following $(M+1)$ diagonal structure
\begin{align*}
 \hat{\omega}=\begin{pmatrix}
  \hat{\omega}_{0,0} & \hat{\omega}_{0,1} &     \dots         &  \hat{\omega}_{0,(M-1)} &  \hat{\omega}_{0,M}      &    0                 &                          &                      &        \\
           0         & \hat{\omega}_{1,1} &                   &                         & \hat{\omega}_{1,M}       &\hat{\omega}_{1,(M+1)}&        0                 &  \dots               &        \\
           0         &         0          & \ddots            &                         &                          &                      &     \ddots               &                      &        \\
                     &                    &                   &\hat{\omega}_{k,k}       &                          &                      &   \hat{\omega}_{k,k+M-1} &\hat{\omega}_{k,k+M}&   0    \\
                     &                    &                   &                         &     \ddots               &                      &                          &        &   \ddots
\end{pmatrix} \ ,
\end{align*}
where $\hat{\omega}_{k,k+M}=1$ and $\hat{\omega}_{k,k}=\frac{\hat{h}_k}{h_k}$.
\end{pro}
\begin{proof}
The previous relations follow from a $LU$-factorization of the expressions defining the Darboux--Sobolev deformed moment matrices.
\end{proof}
\noi Let us establish now some connection formulae relating deformed to non-deformed polynomials. They are based on the notion of resolvent, as clarified by the following
\begin{pro}
Deformed and non deformed polynomials are related by the resolvents
 \begin{align*}
  (\hat{\omega}_L)P_1(x)&=R(x)(\hat{P}_{L1})(x)  &  (\hat{\Omega}_L)(\hat{P}_{L2})(x)=P_2(x) \\
  (\hat{\omega}_R)P_2(x)&=R(x)(\hat{P}_{R2})(x)  &  (\hat{\Omega}_R)(\hat{P}_{R1})(x)=P_1(x)
 \end{align*}
while transformed and non transformed Christoffel--Darboux kernels are related as follows
\begin{align*}
 &K^{[n+1]}(x,y)=R(y)\hat{K}_L^{[n+1]}(x,y)-\\
 &\begin{pmatrix}(\hat{P}_{L2})_{n+1-M} & \dots & (\hat{P}_{L2})_{n} \end{pmatrix}
 \begin{pmatrix}
  (\hat{h}_L)_{n+1-M}^{-1} &       &                \\
                       &\ddots &                \\
                       &       & (\hat{h}_L)_{n}^{-1}
 \end{pmatrix}
 \begin{pmatrix}
  (\hat{\omega}_L)_{n+1-M,n+1} &       &        0        \\
         \vdots            &\ddots &                \\
  (\hat{\omega}_L)_{n,n+1}     & \dots      & (\hat{\omega}_L)_{n,n+M}
 \end{pmatrix}
\begin{pmatrix} (P_1)_{n+1}(y)\\ \vdots \\ (P_1)_{n+m}(y) \end{pmatrix} \\
&K^{[n+1]}(y,x)=R(y)\hat{K}_R^{[n+1]}(y,x)-\\
 &\begin{pmatrix}(\hat{P}_{R1})_{n+1-M} & \dots & (\hat{P}_{R1})_{n} \end{pmatrix}
 \begin{pmatrix}
  (\hat{h}_R)_{n+1-M}^{-1} &       &                \\
                       &\ddots &                \\
                       &       & (\hat{h}_R)_{n}^{-1}
 \end{pmatrix}
 \begin{pmatrix}
  (\hat{\omega}_R)_{n+1-M,n+1} &       &        0        \\
         \vdots            &\ddots &                \\
  (\hat{\omega}_R)_{n,n+1}     & \dots      & (\hat{\omega}_R)_{n,n+M}
 \end{pmatrix}
\begin{pmatrix} (P_2)_{n+1}(y)\\ \vdots \\ (P_2)_{n+m}(y) \end{pmatrix}
\end{align*}
\end{pro}
\begin{proof}
 The first set of relations follow directly by using the definition of the resolvents, and taking into account their action on the SBPS. The second set of
 relations follow by making explicit the equalities
 \[
\left[(\hat{P}_{L2})^{\top}(x) (\hat{\Omega}_L)^{\top}\right]H^{-1}P_1(y)=(\hat{P}_{L2})^{\top}(x)(\hat{H}_L)^{-1}\left[(\hat{\omega}_L)P_1(y) \right]
\]
 for the first one and
 \[
\left[(\hat{P}_{R1})^{\top}(x) (\hat{\Omega}_R)^{\top}\right]H^{-1}P_2(y)=(\hat{P}_{R1})^{\top}(x)(\hat{H}_R)^{-1}\left[(\hat{\omega}_L)P_2(y) \right]
\]
 for the second one.
\end{proof}

\noi Let us introduce a vector of ``germs" of a function near the points $r_i$, having multiplicities $m_i$.

%define the following object that will be really useful in order to obtain a clean notation.
\begin{definition}
Given a function $f(x)$  and a set $r:=\{(r_i,m_i)\}_{i=1}^{d} $ of points $r_i\in\mathbb{R}$  with associated multiplicities $m_i\in\mathbb{N}$, we define the vector of germs $\Pi_r[f]:\mathcal{F}(x)\longrightarrow \mathbb{R}^{\sum m_i}$ as
\begin{align*}
 \Pi_r[f]&:= \left( \frac{f^{(0)}(r_1)}{0!} , \frac{f^{(1)}(r_1)}{1!} , \dots , \frac{f^{(m_1-1)}(r_1)}{(m_1-1)!} ;
                             \frac{f^{(0)}(r_2)}{0!} , \frac{f^{(1)}(r_2)}{1!} , \dots , \frac{f^{(m_2-1)}(r_2)}{(m_2-1)!}; \dots ;
                             \frac{f^{(0)}(r_d)}{0!} , \dots , \frac{f^{(m_d-1)}(r_d)}{(m_d-1)!}
              \right) \ .
\end{align*}
\end{definition}
Now we can state an useful result.
\begin{pro}
The Christoffel transformed polynomials and their norms are given in terms of the original ones by means of the relations
\begin{align*}
 (\hat{P}_{1L})_n(x)&=\frac{1}{R(x)}\Theta_*\left(\begin{array}{c|c}
                      \Pi_r \begin{bmatrix} (P_1)_n \\ (P_1)_{n+1} \\ \vdots \\ (P_1)_{n+M-1} \end{bmatrix}    &  \begin{matrix} (P_1)_n(x) \\ (P_1)_{n+1}(x) \\ \vdots \\ (P_1)_{n+M-1}(x) \end{matrix}    \\
                      \hline
                      \Pi_r[(P_1)_{n+M}]                                      & (P_1)_{n+M}(x)
                            \end{array}\right) \ , &
   \frac{(\hat{P}_{2L})_n(x)}{(\hat{h}_L)_n}&=
 \Theta_*\left(\begin{array}{c|c}
  \Pi_r\begin{bmatrix}  (P_1)_{n+1} \\ \vdots \\ (P_1)_{n+M}   \end{bmatrix} & \begin{matrix} 0 \\ \vdots \\ 1  \end{matrix} \\
  \hline
 \Pi_r[K^{[n+1]}(x,\cdot)]   &  0
 \end{array}\right) \\
  \frac{(\hat{h}_L)_n}{h_n}=&\Theta_*\left(\begin{array}{c|c}
                      \Pi_r \begin{bmatrix} (P_1)_n \\ (P_1)_{n+1} \\ \vdots \\ (P_1)_{n+M-1} \end{bmatrix}    &  \begin{matrix} 1 \\ 0 \\ \vdots \\ 0 \end{matrix}    \\
                      \hline
                      \Pi_r[(P_1)_{n+M}]                                      & 0
                            \end{array}\right) \ ,
\end{align*}

\begin{align*}
 (\hat{P}_{2R})_n(x)&=\frac{1}{R(x)}\Theta_*\left(\begin{array}{c|c}
                      \Pi_r \begin{bmatrix} (P_2)_n \\ (P_2)_{n+1} \\ \vdots \\ (P_2)_{n+M-1} \end{bmatrix}    &  \begin{matrix} (P_2)_n(x) \\ (P_2)_{n+1}(x) \\ \vdots \\ (P_2)_{n+M-1}(x) \end{matrix}    \\
                      \hline
                      \Pi_r[(P_2)_{n+M}]                                      & (P_2)_{n+M}(x)
                            \end{array}\right) \ , &
   \frac{(\hat{P}_{1R})_n(x)}{(\hat{h}_R)_n}&=
 \Theta_*\left(\begin{array}{c|c}
  \Pi_r\begin{bmatrix}  (P_2)_{n+1} \\ \vdots \\ (P_2)_{n+M}   \end{bmatrix} & \begin{matrix} 0 \\ \vdots \\ 1  \end{matrix} \\
  \hline
 \Pi_r[K^{[n+1]}(\cdot,x)]   &  0
 \end{array}\right) \ , \\
  \frac{(\hat{h}_R)_n}{h_n}=&\Theta_*\left(\begin{array}{c|c}
                      \Pi_r \begin{bmatrix} (P_2)_n \\ (P_2)_{n+1} \\ \vdots \\ (P_2)_{n+M-1} \end{bmatrix}    &  \begin{matrix} 1 \\ 0 \\ \vdots \\ 0 \end{matrix}    \\
                      \hline
                      \Pi_r[(P_2)_{n+M}]                                      & 0
                            \end{array}\right) \ .
\end{align*}

\end{pro}
\begin{proof}
We shall focus on the proof of the left-type deformation; the right-type one follows in a completely analogous way. Selecting the $n$-th component of the connection formula one gets
\begin{align*}
 \begin{pmatrix}(\hat{\omega}_L)_{n,n} & (\hat{\omega}_L)_{n,n+1} & \dots   & (\hat{\omega}_L)_{n,n+M-1} &     1  \end{pmatrix}
 \begin{pmatrix} (P_1)_{n}(x)    \\ (P_1)_{n+1}(x)    \\ \vdots \\ (P_1)_{n+M-1}(x)     \\ (P_1)_{n+M}(x)\end{pmatrix}=R(x)(\hat{P}_{1L})_n(x) \ .
\end{align*}
Evaluating now in the zeroes of $R(x)$ it is easy to see that
\begin{align*}
 \begin{pmatrix}(\hat{\omega}_L)_{n,n} & (\hat{\omega}_L)_{n,n+1} & \dots   & (\hat{\omega}_L)_{n,n+M-1} &     1  \end{pmatrix}
 \Pi\begin{bmatrix} (P_1)_{n}    \\ (P_1)_{n+1}    \\ \vdots \\ (P_1)_{n+M-1}     \\ (P_1)_{n+M}\end{bmatrix}= \begin{pmatrix} 0 & 0 & \dots & 0 \end{pmatrix}
\end{align*}
Therefore,
\begin{align*}
\begin{pmatrix}(\hat{\omega}_L)_{n,n} & (\hat{\omega}_L)_{n,n+1} & \dots   & (\hat{\omega}_L)_{n,n+M-1} \end{pmatrix}
\Pi_r\begin{bmatrix} (P_1)_{n}    \\ (P_1)_{n+1}    \\ \vdots \\ (P_1)_{n+M-1}  \end{bmatrix}&=
-\Pi_r[(P_1)_{n+M}]\ ,\\
\end{align*}
i.e
\begin{align*}
\begin{pmatrix}(\hat{\omega}_L)_{n,n} & (\hat{\omega}_L)_{n,n+1} & \dots   & (\hat{\omega}_L)_{n,n+M-1} \end{pmatrix}&=
-\Pi_r[(P_1)_{n+M}] \left(\Pi_r\begin{bmatrix} (P_1)_{n}    \\ (P_1)_{n+1}    \\ \vdots \\ (P_1)_{n+M-1}  \end{bmatrix}\right)^{-1} \ ,
\end{align*}
from which the result for $\hat{P}_{1L}$ and $\hat{h}_L$ follow. In order to obtain the result for $\hat{P}_{2L}$,
it is sufficient to start from the equation that relates the CD-Kernels, and to use the same procedure of evaluation on the zeroes of $R(x)$.
\end{proof}

\begin{definition}
We introduce the $(2M+1)$ banded matrices
 \begin{align*}
  (\hat{\omega}_L)(\hat{\Omega}_R)&:=\hat{J}_{1LR} \ , &
  (\hat{\omega}_R)(\hat{\Omega}_L)&:=\hat{J}_{2RL} \ .
 \end{align*}

\end{definition}
We point out that a generalization of the notion of recurrence relation can be realized by allowing an intertwining of SBPS associated with different measure matrices instead of the same one. In this case we shall talk of a \textit{quasi-recurrence relation}.
\begin{pro}
The right and left deformed SBPS satisfy the following $(2M+1)$ quasi--recurrence relation
 \begin{align*}
  \hat{J}_{1LR}(\hat{P}_{R1})(x)&=R(x)(\hat{P}_{L1})(x) \ ,\\
  \hat{J}_{2RL}(\hat{P}_{L2})(x)&=R(x)(\hat{P}_{R2})(x) \ ,
 \end{align*}
with
\begin{align*}
 \hat{J}_{1LR}=\hat{H}_L\left[\hat{J}_{2RL}\right]^{\top}\hat{H}_R^{-1}
\end{align*}
\end{pro}
Observe that if $\W \in \W_{x}$, then there would be no distinction between $L$ or $R$ sequences. In addition, if we choose $R(x)$ to be a  polynomial of degree one, then $\hat{\omega} \cdot \hat{\Omega}$ is a $2(1)+1$-diagonal matrix and the standard three term recurrence relation is recovered.

\subsection{Geronimus-Sobolev transformations}
Let us now focus on the Geronimus transformation. To this aim, a polynomial
$Q(x):=\prod_{i=1}^{s}(x-q_i)^{n_i}=Q_0+Q_1x+\dots+Q_{N-1}x^{N-1}+x^N$ of degree $\sum_{i=1}^{s}n_i=N$ is needed in order to define
the left and right transformed measure matrices. We introduce the following auxiliary matrix, related to the polynomial $Q(x)$:
\begin{align*}
 \mathbf{Q}:=\begin{pmatrix}
              Q_1    & Q_2    & Q_3    &\dots   & Q_{N-1} & 1     &  0  &\dots \\
              Q_2    & Q_3    &\dots   &Q_{N-1} & 1       &  0    &\dots&      \\
              Q_3    &\dots   &Q_{N-1} &   1    &  0      &\dots  &     &       \\
              \dots  &Q_{N-1} & 1      &  0     &\dots    &       &     &       \\
              Q_{N-1}&   1    &  0     &\dots   &         &       &     &       \\
                 1   &  0     &\dots   &        &         &       &     &     \\
               0     &\dots   &        &        &         &       &     &
             \end{pmatrix} \ .
\end{align*}

\begin{definition}
 As long as $\{ q_i\}_i\cap \Omega=\varnothing$
 the Geronimus Sobolev deformed measure matrices are defined to be
 \begin{align*}
  \check{\W}_L& :=\left[Q(\mathcal{X})\right]^{-1}\W+\sum_{i=1}^{s}\xi^{(i)}\delta(x-q_i)\d x \ ,
  & \check{\Omega}_L:=\Omega\cup \{q_i\}_i \\
  \check{\W}_R& :=\W\left[Q(\mathcal{X}^{\top})\right]^{-1}+\sum_{i=1}^{s}\xi^{(i)}\delta(x-q_i)\d x \ ,
  & \check{\Omega}_R:=\Omega\cup \{q_i\}_i
 \end{align*}
 where $\xi^{(i)}$ are the $n_i \times n_i$ matrices of free parameters
\begin{align*}
  \xi^{(i)}&:=\begin{pmatrix}
              \frac{\xi^{(i)}_{0,0}}{0!0!} & \frac{\xi^{(i)}_{0,1}}{0!1!} & \dots & \frac{\xi^{(i)}_{0,n_i-1}}{(n_i-1)!(n_i-1)!} \\
              \frac{\xi^{(i)}_{1,0}}{1!0!} &  \ddots                      &       &             \\
              \vdots                       &                              & \ddots&            \\
              \frac{\xi^{(i)}_{n_i-1,0}}{(n_i-1)!0!} &                    &       & \frac{\xi^{(i)}_{n_i-1,n_i-1}}{(n_i-1)!(n_i-1)!}
             \end{pmatrix} \ ,  &
  \mathring{\xi}^{(i)}:=\begin{pmatrix}
              \xi^{(i)}_{0,0} & \xi^{(i)}_{0,1} & \dots & \xi^{(i)}_{0,n_i-1} \\
              \xi^{(i)}_{1,0} &  \ddots                      &       &             \\
              \vdots                       &                              & \ddots&            \\
              \xi^{(i)}_{n_i-1,0} &                    &       & \xi^{(i)}_{n_i-1,n_i-1}
             \end{pmatrix} \ .
 \end{align*}
 \end{definition}
\begin{pro}
The transformed  measure matrices and associated moment matrices are related to the original ones by the formulae
 \begin{align*}
  \W& :=Q(\mathcal{X})\check{\W}_L \ , & \W& :=\check{\W}_RQ(\mathcal{X}^{\top}) \ , \\
  G_{\W}&=\left(Q(\Lambda) \right)G_{\check{\W}_L} \ , &
  G_{\W}&=G_{\check{\W}_R}\left(Q(\Lambda) \right)^{\top} \ .
 \end{align*}
\end{pro}
The latter proposition and the assumption that the transformed moment matrices are $LU$-factorizable
motivate the definition of the resolvents in terms of the following matrices.
\begin{definition}
We introduce the matrices
\begin{align*}
 (\check{\omega}_L)&:=\check{H}_L \left(\check{S}_{1L}^{-1}\right)^{\top}Q(\Lambda^{\top})S_1^{\top}H^{-1}=\check{S}_{2L}S_{2}^{-1} \ ,\\
 (\check{\omega}_R)&:=\check{H}_R \left(\check{S}_{2R}^{-1}\right)^{\top}Q(\Lambda^{\top})S_2^{\top}H^{-1}=\check{S}_{1R}S_{1}^{-1} \ .
\end{align*}
\end{definition}
The r.h.s. follow from the LU factorization of the transformed and non transformed moment matrices. It is
not difficult to see that these equalities also imply that the resolvents are lower uni-triangular matrices with only $N$ non-vanishing diagonals beneath the main one. Precisely:
\begin{align*}
 \hat{\omega}=\begin{pmatrix}
  \check{\omega}_{0,0} &         0            &                   &                         &                          &                      &                          &                      \\
  \check{\omega}_{1,0} &\check{\omega}_{1,1}  &                   &                         &                          &                      &                          &                      \\
      \vdots           &    \vdots            & \ddots            &                         &                          &                      &                          &                      \\
  \check{\omega}_{N,0} & \check{\omega}_{N,1} &                   &\check{\omega}_{N,N}     &                          &                      &                          &                     \\
         0             &\check{\omega}_{N+1,1}&                   &\check{\omega}_{N+1,N}   & \check{\omega}_{N+1,N+1} &                      &                          &                      \\
                       &                      & \ddots            &                         &                          &   \ddots             &                          &                       \\
                       &                      &                   & \check{\omega}_{k,k-N}  &                          &                      &\check{\omega}_{k,k}      &                      \\
                       &                      &                   &                         &     \ddots               &                      &                          &    \ddots                          \\
\end{pmatrix} \ ,
\end{align*}
where $\check{\omega}_{k,k-N}=\frac{\check{h}_k}{h_{k-N}}$ $\forall k>N$ and $\check{\omega}_{k,k}=1$.

\begin{pro}
The Geronimus-Sobolev deformed polynomials and the associated second kind functions are related to the non transformed ones according to the formulae
\begin{align*}
  \check{\omega}_LP_2(x)&=\check{P}_{2L}(x) & &\Longrightarrow&
  \check{\omega}_L C_2(x)&=Q(x)\check{C}_{2L}(x)-\check{H}_L\left(\check{S}_{1L}^{-1}\right)^{\top}\mathbf{Q}\chi(x) \\
  \check{\omega}_RP_1(x)&=\check{P}_{1R}(x) & &\Longrightarrow&
  \check{\omega}_R C_1(x)&=Q(x)\check{C}_{1R}(x)-\check{H}_R\left(\check{S}_{2R}^{-1}\right)^{\top}\mathbf{Q}\chi(x)
 \end{align*}
\end{pro}
\begin{proof}
 On the one hand the connection formulae for the polynomials follows straightforward remembering their definition in terms
 of the factorization matrices and from the definition of $\check{\omega}$. On the other hand the connection formulae for
 the second kind functions is a consequence of the former as we are about to prove.
 Let us prove it for the Right transformation, since the proof for the Left transformation needs of the same ideas.
 Firstly let us make the following definition and give a result that is easily verified
 \begin{align*}
 \Delta Q(x,y)&:=Q(x)-Q(y) & \chi(x)^{\top}\mathbf{Q} \chi(y)=-\frac{\Delta Q(x,y)}{y-x}
 \end{align*}
 Using this result the next chain of equalities can be followed
 \begin{align*}
 &\check{\omega}_R C_1(y)-Q(y)\check{C}_{1R}(y)=
 \check{\omega}_R\left(P_1(x),\frac{1}{y-x};\W \right)-Q(y)\left(\check{P}_{1R}(x),\frac{1}{y-x};\check{\W}_R \right)\\
 &=\left(\check{P}_{1R}(x),\frac{1}{y-x};\W \right)-Q(y)\left(\check{P}_{1R}(x),\frac{1}{y-x};\check{\W}_R \right)
 =\left(\check{P}_{1R}(x),\frac{1}{y-x};\check{\W}_R\left[Q(\mathcal{X}^{\top})-Q(y)\right] \right)\\
 &=\int_{\Omega}
 \begin{pmatrix} \check{P}^{(0)}_{1R}(x) & \check{P}^{(1)}_{1R}(x) & \dots & \check{P}^{(k)}_{1R}(x) & \dots \end{pmatrix}\check{\W}_R
 \begin{pmatrix}
  \Delta Q(x,y)                                 &                                       &              &\\
  \frac{\partial }{\partial x} \Delta Q(x,y)                 &\Delta Q(x,y)                          &              &  \\
  \frac{\partial^2}{\partial x^2} \Delta Q(x,y)             &2\frac{\partial}{\partial x} \Delta Q(x,y)         & \Delta Q(x,y)&    \\
  \vdots                                        &  \vdots                               &            &\ddots
 \end{pmatrix}
 \begin{pmatrix}
 \frac{1}{y-x} \\ \frac{\partial }{\partial x}\frac{1}{y-x} \\ \vdots \\ \frac{\partial^k }{\partial x^k}\frac{1}{y-x} \\ \vdots
 \end{pmatrix}\\
 &=\left(\check{P}_{1R}(x),\frac{\Delta Q(x,y)}{y-x};\check{\W}_R \right)=
 -\big(\check{P}_{1R}(x),\chi(x)^{\top};\check{\W}_R \big)\mathbf{Q} \chi(y)=
 -\check{H}_R\left(\check{S}_{2R}^{-1}\right)^{\top}\mathbf{Q}\chi(y)
 \end{align*}
\end{proof}

\noi Let us now study the deformations of Christoffel--Darboux kernels.
\begin{pro}
The deformed Christoffel--Darboux kernels are related to the original ones by means of the formulae
 \begin{align*}
  &\check{K}_R^{[k]}(x,y)=Q(x)K^{[k]}(x,y)- \begin{pmatrix}(\check{P}_{2R})_{k}(x) & \dots & (\check{P}_{2R})_{k+N-1}(x) \end{pmatrix}\cdot\\
&\cdot\begin{pmatrix}
  (\check{h}_R)_{k}^{-1}   &         &     \\
                      & \ddots  &     \\
                      &         &(\check{h}_R)_{k+N-1}^{-1}
 \end{pmatrix}
  \begin{pmatrix}
  (\check{\omega}_R)_{k,k-N}   &    \dots    &  (\check{\omega}_R)_{k,k-1}  \\
                               & \ddots      &    \vdots \\
                               &             &(\check{\omega}_R)_{k+N-1,k-1}
 \end{pmatrix}
 \begin{pmatrix} (P_1)_{k-N}(y) \\ (P_1)_{k+1-N}(y) \\ \vdots \\ (P_1)_{k-1}(y) \end{pmatrix} \ ,
\end{align*}
\begin{align*}
  &\check{K}_L^{[k]}(x,y)=Q(y)K^{[k]}(x,y)-\begin{pmatrix}(\check{P}_{1L})_{k}(x) & \dots & (\check{P}_{1L})_{k+N-1}(x) \end{pmatrix}\cdot\\
	&\cdot\begin{pmatrix}
  (\check{h}_L)_{k}^{-1}   &         &     \\
                      & \ddots  &     \\
                      &         &(\check{h}_L)_{k+N-1}^{-1}
 \end{pmatrix}
  \begin{pmatrix}
  (\check{\omega}_L)_{k,k-N}   &    \dots    &  (\check{\omega}_L)_{k,k-1}  \\
                               & \ddots      &    \vdots \\
                               &             &(\check{\omega}_L)_{k+N-1,k-1}
 \end{pmatrix}
 \begin{pmatrix} (P_2)_{k-N}(x) \\ (P_2)_{k+1-N}(x) \\ \vdots \\ (P_2)_{k-1}(x) \end{pmatrix} \ .
\end{align*}
\noi Similarly, the mixed kernels $\forall k \geq N$ are related as follows
 \begin{align*}
  &Q(x)\mathcal{K}_{2}^{[k]}(x,y) -\begin{pmatrix}(\check{P}_{2R})_{k}(x) & \dots & (\check{P}_{2R})_{k+N-1}(x) \end{pmatrix} \cdot\\
&\cdot\begin{pmatrix}
  (\check{h}_R)_{k}^{-1}  &         &     \\
                 & \ddots  &     \\
                 &         &(\check{h}_R)_{k+N-1}^{-1}
 \end{pmatrix}
  \begin{pmatrix}
  (\check{\omega}_R)_{k,k-N}   &    \dots    &  (\check{\omega}_R)_{k,k-1}  \\
                           & \ddots      &     \\
                           &             &(\check{\omega}_R)_{k+N-1,k-1}
 \end{pmatrix}
 \begin{pmatrix} (C_1)_{k-N}(y) \\ (C_1)_{k+1-N}(y) \\ \vdots \\ (C_1)_{k-1}(y)  \end{pmatrix} \\
 &=Q(y)\check{\mathcal{K}}_{2R}^{[k]}(x,y)-\left(\chi^{[N]}(x)\right)^{\top}\mathbf{Q}\chi^{[N]}(y) \ ,\\
 %\end{align*}
 %\begin{align*}
  &Q(y)\mathcal{K}_{1}^{[k]}(x,y)-\begin{pmatrix}(\check{P}_{1L})_{k}(y) & \dots & (\check{P}_{1L})_{k+N-1}(y) \end{pmatrix} \cdot\\
	&\cdot\begin{pmatrix}
  (\check{h}_L)_{k}^{-1}  &         &     \\
                 & \ddots  &     \\
                 &         &(\check{h}_L)_{k+N-1}^{-1}
 \end{pmatrix}
  \begin{pmatrix}
  (\check{\omega}_L)_{k,k-N}   &    \dots    &  (\check{\omega}_L)_{k,k-1}  \\
                           & \ddots      &     \\
                           &             &(\check{\omega}_L)_{k+N-1,k-1}
 \end{pmatrix}
 \begin{pmatrix} (C_2)_{k-N}(x) \\ (C_2)_{k+1-N}(x) \\ \vdots \\ (C_2)_{k-1}(x)  \end{pmatrix} \\
 &=Q(x)\check{\mathcal{K}}_{1L}^{[k]}(x,y)-\left(\chi^{[N]}(y)\right)^{\top}\mathbf{Q}\chi^{[N]}(x) \ .
 \end{align*}
\end{pro}
\begin{proof}
\noi These expressions are a direct consequence of the connection formulae.
\end{proof}

\noi  We shall also introduce a couple of useful matrices, which will be relevant in the subsequent discussion.
\begin{definition}
Let
\begin{align*}
Q_i(x)&:=\frac{Q(x)}{(x-q_i)^{n_i}}    &
\eta_{n_i \times n_i}&:=\begin{pmatrix}
          0 &   0        & \dots & 1   \\
          0 &   0        &  1    & 0    \\
      \vdots    & \ddots &       & \vdots     \\
          1     &       &       & 0
        \end{pmatrix}_{n_i \times n_i} &  i&=1,2,\dots,s
\end{align*}
We define the $N \times N$ matrices
 \begin{align*}
  \Xi_L&:=\begin{pmatrix}
         \Xi_{L1} &   0   & \dots & 0   \\
          0    &\Xi_{L2}  &  0    &     \\
               &       & \ddots&      \\
               &       &       &\Xi_{Ls}
        \end{pmatrix} \ , &
  \Xi_R&:=\begin{pmatrix}
         \Xi_{R1} &   0   & \dots & 0   \\
          0    &\Xi_ {R2} &  0    &     \\
               &       & \ddots&      \\
               &       &       &\Xi_{Rs}
        \end{pmatrix} \ ,
 \end{align*}
where
\begin{align*}
  \Xi_{Rj}&:=\left(\mathring{\xi}^{(j)}\right) \left( \eta_{n_j \times n_j} \right)
        \begin{pmatrix}
         \frac{Q_j^{(0)}(q_j)}{0!} &\frac{Q_j^{(1)}(q_j)}{1!}& \dots &\frac{Q_j^{(n_j-2)}(q_j)}{(n_j-2)!}     & \frac{Q_j^{(n_j-1)}(q_j)}{(n_j-1)!}   \\
                       &\frac{Q_j^{(0)}(q_j)}{0!}&       &              &\frac{Q_j^{(n_j-2)}(q_j)}{(n_j-2)!}\\
                       &             &\ddots &              & \vdots    \\
                       &             &       &\ddots        &\frac{Q_j^{(1)}(q_j)}{1!}      \\
                       &             &       &              &\frac{Q_j^{(0)}(q_j)}{0!}
        \end{pmatrix} \\
\end{align*}
\noi and
\begin{align*}
  \Xi_{Lj}&:=\left(\mathring{\xi}^{(j)}\right)^{\top} \left( \eta_{n_j \times n_j} \right)
        \begin{pmatrix}
         \frac{Q_j^{(0)}(q_j)}{0!} &\frac{Q_j^{(1)}(q_j)}{1!}& \dots &\frac{Q_j^{(n_j-2)}(q_j)}{(n_j-2)!}     & \frac{Q_j^{(n_j-1)}(q_j)}{(n_j-1)!}   \\
                       &\frac{Q_j^{(0)}(q_j)}{0!}&       &              &\frac{Q_j^{(n_j-2)}(q_j)}{(n_j-2)!}\\
                       &             &\ddots &              & \vdots    \\
                       &             &       &\ddots        &\frac{Q_j^{(1)}(q_j)}{1!}      \\
                       &             &       &              &\frac{Q_j^{(0)}(q_j)}{0!}
        \end{pmatrix} \ .
 \end{align*}

\end{definition}
\noi We define a couple of matrices useful in the discussion of transformed Genonimus-Sobolev polynomials.
\begin{definition}
We introduce the $N\times N$ matrices
 \begin{align*}
  \mathring{\Pi}_R&:=\left(
 \Pi_q\begin{bmatrix} (C_1)_{0} \\ \vdots \\ (C_1)_{N-1} \end{bmatrix}-
 \Pi_q\begin{bmatrix}(P_1)_{0} \\ \vdots \\ (P_1)_{N-1} \end{bmatrix}\Xi_R
 \right) \left(\mathbf{Q}^{[N]}\Pi_q[\chi^{[N]}] \right)^{-1}  \\
 \mathring{\Pi}_L&:=\left(
 \Pi_q\begin{bmatrix} (C_2)_{0} \\ \vdots \\ (C_2)_{N-1} \end{bmatrix}-
 \Pi_q\begin{bmatrix}(P_2)_{0} \\ \vdots \\ (P_2)_{N-1} \end{bmatrix}\Xi_L
 \right) \left(\mathbf{Q}^{[N]}\Pi_q[\chi^{[N]}] \right)^{-1} \ ,
 \end{align*}
Where $\Pi_q[f]$ is the vector of germs associated to the set $q:=\{q_i,n_i\}$.
\end{definition}
\noi An interesting characterization of the class of Geronimus-type transformed polynomials can be obtained in terms of quasi-determinants,
as clarified by the following \\
\begin{pro}
Geronimus Sobolev transformed polynomials are expressed $\forall k\geq N$  in terms of the original polynomials via the formulae
\begin{align*}
 (\check{P}_{1R})_{k}&= \Theta_*\left(\begin{array}{c|c}
 \Pi_q\begin{bmatrix} (C_1)_{k-N} \\ \vdots \\ (C_1)_{k-1} \end{bmatrix}-
 \Pi_q\begin{bmatrix}(P_1)_{k-N} \\ \vdots \\ (P_1)_{k-1} \end{bmatrix}\Xi_R
 &  \begin{matrix}(P_1)_{k-N} \\ \vdots \\ (P_1)_{k-1} \end{matrix} \\
 \hline
 \Pi_q[(C_1)_k]-\Pi_q[(P_1)_k]\Xi_R & (P_1)_k(x)
 \end{array}\right) \ , \\
 \frac{(\check{P}_{2R})_k(x)}{(\check{h}_R)_k}&=
 \Theta_*\left(\begin{array}{c|c}
  \Pi_q\begin{bmatrix} (C_1)_{k-N} \\ \vdots \\ (C_1)_{k-1} \end{bmatrix}-
 \Pi_q\begin{bmatrix}(P_1)_{k-N} \\ \vdots \\ (P_1)_{k-1} \end{bmatrix}\Xi_R
  & \begin{matrix} 1 \\ \vdots \\ 0  \end{matrix}  \ , \\
  \hline
 Q(x)\left( \Pi_q[\mathcal{K}_2^{[k]}(x,\cdot)]-\Pi_q[K^{[k]}(x,\cdot)]\Xi _R\right)+\left(\chi^{[N]}(x)\right)^{\top}\mathbf{Q}\Pi_q[\chi^{[N]}]  &  0
 \end{array}\right) \ , \\
 (\check{h}_R)_{k}(x)&=h_{k-N}\Theta_*\left(\begin{array}{c|c}
 \Pi_q\begin{bmatrix} (C_1)_{k-N} \\ \vdots \\ (C_1)_{k-1} \end{bmatrix}-\Pi_q\begin{bmatrix}(P_1)_{k-N} \\  (P_1)_{k+1-N}\\ \vdots \\ (P_1)_{k-1} \end{bmatrix}\Xi_R &  \begin{matrix}1 \\ 0\\ \vdots \\ 0 \end{matrix}\\
 \hline
 \Pi_q[(C_1)_k]-\Pi_q[(P_1)_k]\Xi_R & 0
 \end{array}\right) \ ,
\end{align*}

\begin{align*}
 (\check{P}_{2L})_{k}&= \Theta_*\left(\begin{array}{c|c}
 \Pi_q\begin{bmatrix} (C_2)_{k-N} \\ \vdots \\ (C_2)_{k-1} \end{bmatrix}-\Pi_q\begin{bmatrix}(P_2)_{k-N} \\ \vdots \\ (P_2)_{k-1} \end{bmatrix}\Xi_L &  \begin{matrix}(P_2)_{k-N} \\ \vdots \\ (P_2)_{k-1} \end{matrix}\\
 \hline
 \Pi_q[(C_2)_k]-\Pi_q[(P_2)_k]\Xi_L & (P_2)_k(x)
 \end{array}\right) \ , \\
 \frac{(\check{P}_{1L})_k(x)}{(\check{h}_L)_k}&=
 \Theta_*\left(\begin{array}{c|c}
  \Pi_q\begin{bmatrix} (C_2)_{k-N} \\ \vdots \\ (C_2)_{k-1} \end{bmatrix}-
 \Pi_q\begin{bmatrix}(P_2)_{k-N} \\ \vdots \\ (P_2)_{k-1} \end{bmatrix}\Xi_L
  & \begin{matrix} 1 \\ \vdots \\ 0  \end{matrix} \\
  \hline
 Q(x)\left( \Pi_q[\mathcal{K}_1^{[k]}(\cdot,x)]-\Pi_q[K^{[k]}(\cdot,x)]\Xi _L\right)+\left(\chi^{[N]}(x)\right)^{\top}\mathbf{Q}\Pi_q[\chi^{[N]}]  &  0
 \end{array}\right) \ , \\
 \check{h}_{Lk}(x)&=h_{k-N}\Theta_*\left(\begin{array}{c|c}
 \Pi_q\begin{bmatrix} (C_2)_{k-N} \\ \vdots \\ (C_2)_{k-1} \end{bmatrix}-\Pi_q\begin{bmatrix}(P_2)_{k-N} \\  (P_2)_{k+1-N}\\ \vdots \\ (P_2)_{k-1} \end{bmatrix}\Xi_L &  \begin{matrix}1 \\ 0\\ \vdots \\ 0 \end{matrix}\\
 \hline
 \Pi_q[(C_2)_k]-\Pi_q[(P_2)_k]\Xi_L & 0
 \end{array}\right) \ .
\end{align*}
For $k<N$ the following expressions hold
\begin{align*}
 (\check{P}_{1R})_{k}(x)&= \Theta_*\left(\begin{array}{c|c}
 \mathring{\Pi}_R^{[k]}
 &  \begin{matrix}(P_1)_{0}(x) \\ \vdots \\ (P_1)_{k-1}(x) \end{matrix} \\
 \hline
 \begin{matrix}\left(\mathring{\Pi}_R\right)_{k,0} & \dots & \left(\mathring{\Pi}_R\right)_{k,k-1}  \end{matrix}  &  (P_1)_{k}(x)
 \end{array}\right) \ ,\\
(\check{P}_{2R})_{k}(x)&= \Theta_*\left(\begin{array}{c|c}
 \left(\mathring{\Pi}_R^{\top}\right)^{[k]}
 &  \begin{matrix}1 \\ \vdots \\ x^{k-1} \end{matrix} \\
 \hline
 \begin{matrix}\left(\mathring{\Pi}_R^{\top}\right)_{k,0} & \dots & \left(\mathring{\Pi}_R^{\top}\right)_{k,k-1}  \end{matrix}  &  x^k
 \end{array}\right) \\
 (\check{h}_{R})_{k}&= -\Theta_*\left(\begin{array}{c|c}
 \mathring{\Pi}_R^{[k]}
 &  \begin{matrix}\left(\mathring{\Pi}_R\right)_{0,k} \\ \vdots \\ \left(\mathring{\Pi}_R\right)_{k-1,k} \end{matrix} \\
 \hline
 \begin{matrix}\left(\mathring{\Pi}_R\right)_{k,0} & \dots & \left(\mathring{\Pi}_R\right)_{k,k-1}  \end{matrix}  &  \left(\mathring{\Pi}_R\right)_{k,k}
 \end{array}\right) \ ,
\end{align*}

\begin{align*}
 (\check{P}_{2L})_{k}(x)&= \Theta_*\left(\begin{array}{c|c}
 \mathring{\Pi}_L^{[k]}
 &  \begin{matrix}(P_2)_{0}(x) \\ \vdots \\ (P_2)_{k-1}(x) \end{matrix} \\
 \hline
 \begin{matrix}\left(\mathring{\Pi}_L\right)_{k,0} & \dots & \left(\mathring{\Pi}_L\right)_{k,k-1}  \end{matrix}  &  (P_2)_{k}(x)
 \end{array}\right)\ ,\\
 (\check{P}_{1L})_{k}(x)&= \Theta_*\left(\begin{array}{c|c}
 \left(\mathring{\Pi}_L^{\top}\right)^{[k]}
 &  \begin{matrix}1 \\ \vdots \\ x^{k-1} \end{matrix} \\
 \hline
 \begin{matrix}\left(\mathring{\Pi}_L^{\top}\right)_{k,0} & \dots & \left(\mathring{\Pi}_L^{\top}\right)_{k,k-1}  \end{matrix}  &  x^k
 \end{array}\right) \ , \\
 (\check{h}_{L})_{k}&= -\Theta_*\left(\begin{array}{c|c}
 \mathring{\Pi}_L^{[k]}
 &  \begin{matrix}\left(\mathring{\Pi}_L\right)_{0,k} \\ \vdots \\ \left(\mathring{\Pi}_L\right)_{k-1,k} \end{matrix} \\
 \hline
 \begin{matrix}\left(\mathring{\Pi}_L\right)_{k,0} & \dots & \left(\mathring{\Pi}_L\right)_{k,k-1}  \end{matrix}  &  \left(\mathring{\Pi}_L\right)_{k,k}
 \end{array}\right)\ .
\end{align*}
\end{pro}
\begin{proof}
\noi We shall focus on the case of right transformations. We start looking at the Geronimus transformed second kind functions
\begin{align*}
 &(\check{C}_{1R})_{k}(y)=\left((\check{P}_{1R})_k,\frac{1}{y-x} \right)_{\check{\W}_R}=
 \int \begin{pmatrix}
       (\check{P}_{1R})_k & (\check{P}_{1R})_k^{'} & \dots
      \end{pmatrix}\W \left[Q(\mathcal{X}^{\top})\right]^{-1}
      \begin{pmatrix}
       \frac{1}{y-x} \\ \frac{\partial}{\partial x} \left( \frac{1}{y-x}\right) \\ \vdots
      \end{pmatrix} \\
&+
 \sum_{j=1}^{s}\begin{pmatrix}\frac{(\check{P}_{1R})^{(0)}_k(q_j)}{0!} & \frac{(\check{P}_{1R})^{(1)}_k(q_j)}{1!} & \dots & \frac{(\check{P}_{1R})^{(n_j-1)}_k(q_j)}{(n_j-1)!} \end{pmatrix}
 \begin{pmatrix}
         \xi_{0,0}^{(j)}     &\xi_{0,1}^{(j)}    & \dots &\xi_{0,2}^{(j)}     & \xi_{0,n_j-1}^{(j)}   \\
         \xi_{1,0}^{(j)}     &\xi_{1,1}^{(j)}    &       &              &\xi_{1,n_j-1}^{(j)} \\
            \vdots     &             &\ddots &              & \vdots    \\
   \xi_{n_j-1,n_j-1}^{(j)}   &             &       &\ddots        &\xi_{n_j-1,n_j-1}^{(j)}      \\
        \end{pmatrix} \left(\eta \right)_{n_j \times n_j}
 \begin{pmatrix} \left(\frac{1}{y-q_j}\right)^{n_j} \\ \left(\frac{1}{y-q_j}\right)^{n_j-1} \\ \vdots \\ \frac{1}{y-q_j} \end{pmatrix}\ .
\end{align*}
Therefore, multiplying the previous expression by $Q(y)$ and letting $y\rightarrow q_j$, we obtain the Taylor expansion
\begin{align*}
 Q(y)(\check{C}_{1R})_{k}(y)&=
 \begin{pmatrix}\frac{(\check{P}_{1R})^{(0)}_k(q_j)}{0!} & \frac{(\check{P}_{1R})^{(1)}_k(q_j)}{1!} & \dots & \frac{(\check{P}_{1R})^{(n_j-1)}_k(q_j)}{(n_j-1)!} \end{pmatrix}
 \Xi_{Rj}
 \begin{pmatrix} 1 \\ (y-q_j) \\ \vdots \\ (y-q_j)^{n_j-1} \end{pmatrix}+O(y-q_j)^{n_j} \ .
\end{align*}
The previous reasoning can be repeated for each $j$. Consequently, collecting all the information in the same matrix we can write the relation
\begin{align*}
 \Pi_q[Q(\check{C}_{1R})_k]=\Pi_q[(\check{P}_{1R})_k]\Xi_R \ .
\end{align*}
By using the connection formula for the second kind functions, and applying $\Pi_q$  to both sides, if we also take into account the previous relation, we get the equations
\begin{align*}
 \check{\omega}_R\Pi_q[C_1]&=\Pi_q[Q\check{C}_{1R}]-\check{H}_R\left(\check{S}_{2R}^{-1}\right)^{\top}\mathbf{Q}\Pi_q[\chi(x)] \ ,\\
 \check{\omega}_R\Pi_q[C_1]&=\Pi_q[\check{P}_{1R}]\Xi_R-\check{H}_R\left(\check{S}_{2R}^{-1}\right)^{\top}\mathbf{Q}\Pi_q[\chi(x)] \ .
\end{align*}
Rearranging terms and using the connection formula for the polynomials, we arrive at the expression
\begin{align*}
 \check{\omega}_R\left(\Pi_q[C_1]-\Pi_q[P_1]\Xi_R \right)&=-\check{H}_R\left(\check{S}_{2R}^{-1}\right)^{\top}\mathbf{Q}\Pi_q[\chi(x)] \ .
\end{align*}

\noi This result can be made more explicit once written in the form
\begin{align*}
 \begin{pmatrix}(\check{\omega}_R)_{k,k-N} & \dots & (\check{\omega}_R)_{k,k-1} & 1 \end{pmatrix}
 \left(\Pi_q\begin{bmatrix} (C_1)_{k-N} \\ (C_1)_{k-N+1} \\ \vdots \\ (C_1)_{k} \end{bmatrix}-
 \Pi_q\begin{bmatrix} (P_1)_{k-N} \\ (P_1)_{k-N+1} \\ \vdots \\ (P_1)_{k} \end{bmatrix}\Xi_R \right)&=0,     &    &\forall k\geq N \ ,
\end{align*}
whence, the expression for the first right-family and their norms follows straightforwardly. In order to obtain the expression for the second
right-family, a similar approach can be used, based now on the relations between CD kernels and their mixed versions. %\\
%*******************************************************************\\
%*******************************************************************\\
For $k<N$ the expression for both families and norms is a consequence of the following $LU$-factorization of the matrix $\mathring{\Pi}_R$
\begin{align*}
 \left(\check{\omega}_R\right)^{[N]}
 \left(\Pi_q\begin{bmatrix} (C_1)_{0} \\ (C_1)_{1} \\ \vdots \\ (C_1)_{N} \end{bmatrix}-
 \Pi_q\begin{bmatrix} (P_1)_{0} \\ (P_1)_{1} \\ \vdots \\ (P_1)_{N} \end{bmatrix}\Xi_R \right)&=
 -\check{H}_R^{[N]}\left(\left(\check{S}_{2R}^{-1}\right)^{\top}\right)^{[N]}\mathbf{Q}^{[N]}\Pi[\chi^{[N]}(x)]  & \Longrightarrow\\
\mathring{\Pi}_R&=-\left(\check{\omega}_R^{-1}\right)^{[N]}\check{H}_R^{[N]}\left(\left(\check{S}_{2R}^{-1}\right)^{\top}\right)^{[N]} \ .
\end{align*}
The proof for the case of the left deformation is completely analogous and is left to the reader.\\
%*******************************************************************\\
%*******************************************************************\\

\end{proof}
We shall conclude this section with an observation on the recurrence relations for Geronimus-type polynomials arising from our transformation approach.
\begin{definition}
 Let us define the following matrices
\begin{align*}
 \check{J}_{1RL}:=&\check{S}_{1R}Q(\Lambda)\check{S}_{1L}^{-1}   &
 \check{J}_{2LR}:=&\check{S}_{2L}Q(\Lambda)\check{S}_{2R}^{-1} \ .
\end{align*}
 \end{definition}

\begin{pro}
The matrices $\check{J}_{1RL}$ and $\check{J}_{2LR}$ possess a $2N+1$ diagonal structure and are related to each other according to the formulae
\begin{align*}
 \check{J}_{1RL} \check{H}_L=\check{H}_R \check{J}_{2LR}^{\top} \ .
\end{align*}
These induce a left and right $2N+1$ term recurrence relation involving the Geronimus transformed polynomials:
\begin{align*}
 \check{J}_{1RL}\check{P}_{1L}&=Q(x)\check{P}_{1R} &
 \check{J}_{2LR}\check{P}_{2R}&=Q(x)\check{P}_{2L} \ .
\end{align*}
\end{pro}

\begin{proof}
The proposition is a consequence of the relation
\begin{align*}
 Q(\Lambda) \check{G}_L&=\check{G}_R Q(\Lambda^{\top})
\end{align*}
combined with a $LU$-factorization of the moment matrices.
\end{proof}

\subsection{Sobolev--linear spectral transformations}
After the previous discussion concerning both Christoffel and Geronimus Sobolev transformations, the successive composition of the last two follows straightforwardly.
For this reason, proofs will be summarized or omitted in case they provide no new insight.

We start with the selection of two (co-prime) polynomials in order to deform an initial $\W(\Omega)$. Let these be
$R(x):=\prod_{i=1}^{d}(x-r_i)^{m_i}$  of degree $\sum_{i=1}^{d}m_i=M$, and
$Q(x):=\prod_{i=1}^{s}(x-q_i)^{n_i}$ of degree $\sum_{i=1}^{s}n_i=N$, where again we require that
$\{q_i\}\cap \Omega =\varnothing$ in order to define what we understand
for Sobolev linear spectral transformations.
\begin{definition}
 The Sobolev linear spectral deformed measure matrices are defined to be the composition of both a Geronimus and Christoffel
 transformation
 \begin{align*}
 \tilde{\W}_{RL}:=\widehat{(\check{\W}_R)}_L& =R(\mathcal{X})\W\left[Q(\mathcal{X}^{\top})\right]^{-1}
 +\sum_{i=1}^{s}R(\mathcal{X})\xi^{(i)}\delta(x-q_i)  \\
 \tilde{\W}_{LR}:=\widehat{(\check{\W}_L)}_R&=\left[Q(\mathcal{X})\right]^{-1}\W R(\mathcal{X^{\top}}) +\sum_{i=1}^{s}\xi^{(i)}R(\mathcal{X^{\top}})\delta(x-q_i)
 \end{align*}
\end{definition}
\noi Therefore transformed and non transformed moment matrices are related according to the formulae
\begin{align*}
 R(\Lambda)G_{\W}&=G_{\tilde{\W}_{RL}} Q(\Lambda^\top)  & Q(\Lambda)G_{\W}&=G_{\tilde{\W}_{LR}} R(\Lambda^\top) \ .
\end{align*}
After performing a $LU$-factorization of the moment matrices we are led to the following expressions.
\begin{definition}
The resolvents and adjoint resolvents are defined as
\begin{align*}
 (\tilde{\omega}_{RL})&:=(\tilde{S}_{RL1})R(\Lambda)S_1^{-1}  &   (\tilde{\Omega}_{RL})&:=S_2 Q(\Lambda) (\tilde{S}_{RL2})^{-1} \\
 (\tilde{\omega}_{LR})&:=(\tilde{S}_{LR2})R(\Lambda)S_2^{-1}  &   (\tilde{\Omega}_{LR})&:=S_1 Q(\Lambda) (\tilde{S}_{LR1})^{-1}
\end{align*}
and are related as follows
\begin{align*}
 (\tilde{\omega}_{RL})&=(\tilde{H}_{RL})(\tilde{\Omega}_{RL})^{\top} H^{-1 } &
 (\tilde{\omega}_{LR})&=(\tilde{H}_{LR})(\tilde{\Omega}_{LR})^{\top} H^{-1 }
\end{align*}
\end{definition}
The last relation induces a $N+M+1$ diagonal structure for them. For example $\tilde{\omega}$ has only non zero terms along
the main diagonal together with $N$ sub-diagonals and $M$ super-diagonals. It also follows that
$\tilde{\omega}_{k,k-N}=\frac{\tilde{h}_{k}}{h_{k-N}}$ and $\tilde{\omega}_{k,k+M}=1$.

\begin{pro}
The Sobolev linear spectral deformed polynomials and the associated second kind functions are related to the non
transformed ones according to the formulae
\begin{align*}
  \tilde{\omega}_{RL}P_1(x)&=R(x)\tilde{P}_{1RL}(x) &
  \tilde{\omega}_{RL} C_1(x)&=Q(x)\tilde{C}_{1RL}(x)-\tilde{H}_{RL}\left(\tilde{S}_{2RL}^{-1}\right)^{\top}\mathbf{Q}\chi(x) \\
  \tilde{\omega}_{LR}P_2(x)&=R(x)\tilde{P}_{2LR}(x) &
  \tilde{\omega}_{LR} C_2(x)&=Q(x)\tilde{C}_{2LR}(x)-\tilde{H}_{LR}\left(\tilde{S}_{1LR}^{-1}\right)^{\top}\mathbf{Q}\chi(x) \\
 \end{align*}
\end{pro}
\noi Let us use the notation
\begin{align*}
 A=\left(\begin{array}{c|c}
   A^{[k]}  &  A^{[k,\geq k]} \\
   \hline
   A^{[\geq k,k]}& A^{[\geq k]}
  \end{array}\right)
\end{align*}
\noi in order to state the following

\begin{definition}
We define the $(N+M)\times(N+M)$ matrices
\begin{align*}
  \left(\Upsilon_{RL}\right)_k&:=
  \left(\begin{array}{c|c}
   0_{M\times N}  &  -\left(\tilde{h}_{RL}\tilde{\omega}_{RL}\right)^{[k,\geq k]} \\
   \hline
   \left(\tilde{h}_{RL}\tilde{\omega}_{RL}\right)^{[\geq k,k]}& 0_{N\times M}
  \end{array}\right)\ ,  &
  \left(\Upsilon_{LR}\right)_k&:=
  \left(\begin{array}{c|c}
   0_{M\times N}   &  -\left(\tilde{h}_{LR}\tilde{\omega}_{LR}\right)^{[k,\geq k]} \\
   \hline
   \left(\tilde{h}_{LR}\tilde{\omega}_{LR}\right)^{[\geq k,k]}& 0_{N\times M}
  \end{array}\right)\ .
 \end{align*}

\end{definition}

\begin{pro}
The deformed Christoffel--Darboux kernels are related to the original ones by means of the formulae
 \begin{align*}
  &R(y)\tilde{K}_{RL}^{[k]}(x,y)=Q(x)K^{[k]}(x,y)-
  \begin{pmatrix}(\check{P}_{2RL})_{k-M}(x) & \dots & (\check{P}_{2RL})_{k+N-1}(x) \end{pmatrix}
  \left(\Upsilon_{RL}\right)_k
 \begin{pmatrix} (P_1)_{k-N}(y) \\ (P_1)_{k+1-N}(y) \\ \vdots \\ (P_1)_{k+M-1}(y) \end{pmatrix} \\
  &R(y)\tilde{K}_{LR}^{[k]}(y,x)=Q(x)K^{[k]}(y,x)-
  \begin{pmatrix}(\check{P}_{1LR})_{k-M}(x) & \dots & (\check{P}_{1LR})_{k+N-1}(x) \end{pmatrix}
  \left(\Upsilon_{LR}\right)_k
 \begin{pmatrix} (P_2)_{k-N}(y) \\ (P_2)_{k+1-N}(y) \\ \vdots \\ (P_2)_{k+M-1}(y) \end{pmatrix}
\end{align*}
\noi Similarly, the mixed kernels are related by means of the formulae

\begin{align*}
  &Q(y)\tilde{\mathcal{K}}_{2RL}^{[k]}(x,y)=Q(x)\mathcal{K}_2^{[k]}(x,y)-
  \begin{pmatrix}(\check{P}_{2RL})_{k-M}(x) & \dots & (\check{P}_{2RL})_{k+N-1}(x) \end{pmatrix}
  \left(\Upsilon_{RL}\right)_k
 \begin{pmatrix} (C_1)_{k-N}(y) \\ (C_1)_{k+1-N}(y) \\ \vdots \\ (C_1)_{k+M-1}(y) \end{pmatrix}\\
&+\left(\chi^{[N]}(x)\right)^{\top}\mathbf{Q}\chi^{[N]}(y)\\
  &Q(y)\tilde{\mathcal{K}}_{1LR}^{[k]}(y,x)=Q(x)\mathcal{K}_1^{[k]}(y,x)-
  \begin{pmatrix}(\check{P}_{1LR})_{k-M}(x) & \dots & (\check{P}_{1LR})_{k+N-1}(x) \end{pmatrix}
  \left(\Upsilon_{LR}\right)_k
 \begin{pmatrix} (C_2)_{k-N}(y) \\ (C_2)_{k+1-N}(y) \\ \vdots \\ (C_2)_{k+M-1}(y) \end{pmatrix}\\
&+\left(\chi^{[N]}(x)\right)^{\top}\mathbf{Q}\chi^{[N]}(y)
\end{align*}
\end{pro}
\noi Since in the linear spectral type transformations two polynomials are involved, the presence of two vectors of germs is expected.
As was done previously, we denote by
$\Pi_r[f]$ the one related to the set $r:=\{r_i,m_i\}_{i=1}^{d}$ and by $\Pi_q[f]$ the one related to $q:=\{q_i,n_i\}_{i=1}^{s}$.

\begin{pro}
Sobolev linear spectral transformed polynomials are expressed $\forall k\geq N$  in terms of the original polynomials via the formulae
\begin{align*}
 (\tilde{P}_{1RL})_{k}(x)&= \frac{1}{R(x)}
 \Theta_*\left(\begin{array}{c|c}
 \Pi_r\begin{bmatrix} (P_1)_{k-N} \\ \vdots \\ (P_1)_{k+M-1} \end{bmatrix},
 \Pi_q\begin{bmatrix} (C_1)_{k-N} \\ \vdots \\ (C_1)_{k+M-1} \end{bmatrix}-
 \Pi_q\begin{bmatrix}(P_1)_{k-N} \\ \vdots \\ (P_1)_{k+M-1} \end{bmatrix}\Xi_R
 &  \begin{matrix}(P_1)_{k-N} \\ \vdots \\ (P_1)_{k+M-1} \end{matrix} \\
 \hline
 \Pi_r[(P_{1})_{k+M}],\Pi_q[(C_1)_{k+M}]-\Pi_q[(P_1)_{k+M}]\Xi_R & (P_1)_{k+M}(x)
 \end{array}\right) \ , \\
 \frac{(\tilde{P}_{2RL})_k(x)}{(\tilde{h}_{RL})_k}&=
 \Theta_*\left(\begin{array}{c|c}
 \Pi_r\begin{bmatrix} (P_1)_{k-N} \\ \vdots \\ (P_1)_{k+M-1} \end{bmatrix},
 \Pi_q\begin{bmatrix} (C_1)_{k-N} \\ \vdots \\ (C_1)_{k+M-1} \end{bmatrix}-
 \Pi_q\begin{bmatrix}(P_1)_{k-N} \\ \vdots \\ (P_1)_{k+M-1} \end{bmatrix}\Xi_R
 & \begin{matrix} 1 \\ 0 \\ \vdots \\ 0 \end{matrix} \\
  \hline
 Q(x)\Pi_r[K^{[k]}(x,\cdot)],
 Q(x)\left( \Pi_q[\mathcal{K}_2^{[k]}(x,\cdot)]-\Pi_q[K^{[k]}(x,\cdot)]\Xi _R\right)+\left(\chi^{[N]}(x)\right)^{\top}\mathbf{Q}\Pi_q[\chi^{[N]}]  &  0
 \end{array} \right) \ , \\
 (\tilde{h}_{RL})_{k}(x)&=h_{k-N}
 \Theta_*\left(\begin{array}{c|c}
 \Pi_r\begin{bmatrix} (P_1)_{k-N} \\ \vdots \\ (P_1)_{k+M-1} \end{bmatrix},
 \Pi_q\begin{bmatrix} (C_1)_{k-N} \\ \vdots \\ (C_1)_{k+M-1} \end{bmatrix}-
 \Pi_q\begin{bmatrix}(P_1)_{k-N} \\ \vdots \\ (P_1)_{k+M-1} \end{bmatrix}\Xi_R
 &  \begin{matrix}1   \\ 0   \\ \vdots \\ 0 \end{matrix} \\
 \hline
 \Pi_r[(P_{1})_{k+M}],\Pi_q[(C_1)_{k+M}]-\Pi_q[(P_1)_{k+M}]\Xi_R & 0
 \end{array}\right) \ ,
\end{align*}

\begin{align*}
 (\tilde{P}_{2LR})_{k}(x)&= \frac{1}{R(x)}
 \Theta_*\left(\begin{array}{c|c}
 \Pi_r\begin{bmatrix} (P_2)_{k-N} \\ \vdots \\ (P_2)_{k+M-1} \end{bmatrix},
 \Pi_q\begin{bmatrix} (C_2)_{k-N} \\ \vdots \\ (C_2)_{k+M-1} \end{bmatrix}-
 \Pi_q\begin{bmatrix}(P_2)_{k-N} \\ \vdots \\ (P_2)_{k+M-1} \end{bmatrix}\Xi_L
 &  \begin{matrix}(P_2)_{k-N} \\ \vdots \\ (P_2)_{k+M-1} \end{matrix} \\
 \hline
 \Pi_r[(P_{2})_{k+M}],\Pi_q[(C_2)_{k+M}]-\Pi_q[(P_2)_{k+M}]\Xi_L & (P_2)_{k+M}(x)
 \end{array}\right) \ , \\
 \frac{(\tilde{P}_{1LR})_k(x)}{(\tilde{h}_{LR})_k}&=
 \Theta_*\left(\begin{array}{c|c}
 \Pi_r\begin{bmatrix} (P_2)_{k-N} \\ \vdots \\ (P_2)_{k+M-1} \end{bmatrix},
 \Pi_q\begin{bmatrix} (C_2)_{k-N} \\ \vdots \\ (C_2)_{k+M-1} \end{bmatrix}-
 \Pi_q\begin{bmatrix}(P_2)_{k-N} \\ \vdots \\ (P_2)_{k+M-1} \end{bmatrix}\Xi_L
 & \begin{matrix} 1 \\ 0 \\ \vdots \\ 0 \end{matrix} \\
  \hline
 Q(x)\Pi_r[K^{[k]}(\cdot,x)],
 Q(x)\left( \Pi_q[\mathcal{K}_1^{[k]}(\cdot,x)]-\Pi_q[K^{[k]}(\cdot,x)]\Xi _L\right)+\left(\chi^{[N]}(x)\right)^{\top}\mathbf{Q}\Pi_q[\chi^{[N]}]  &  0
 \end{array} \right) \ , \\
 (\tilde{h}_{LR})_{k}(x)&=h_{k-N}
 \Theta_*\left(\begin{array}{c|c}
 \Pi_r\begin{bmatrix} (P_2)_{k-N} \\ \vdots \\ (P_2)_{k+M-1} \end{bmatrix},
 \Pi_q\begin{bmatrix} (C_2)_{k-N} \\ \vdots \\ (C_2)_{k+M-1} \end{bmatrix}-
 \Pi_q\begin{bmatrix}(P_2)_{k-N} \\ \vdots \\ (P_2)_{k+M-1} \end{bmatrix}\Xi_L
 &  \begin{matrix}1   \\ 0   \\ \vdots \\ 0 \end{matrix} \\
 \hline
 \Pi_r[(P_{2})_{k+M}],\Pi_q[(C_2)_{k+M}]-\Pi_q[(P_2)_{k+M}]\Xi_L & 0
 \end{array}\right) \ .
\end{align*}

\noi For $k<N$ the following expressions hold
\begin{align*}
 (\tilde{P}_{1RL})_{k}(x)&= \frac{1}{R(x)}\Theta_*\left(\begin{array}{c|c}
 \mathring{\Pi}_{RL}^{[k]}
 &  \begin{matrix}(P_1)_{0}(x) \\ \vdots \\ (P_1)_{k+N-1}(x) \end{matrix} \\
 \hline
 \begin{matrix}\left(\mathring{\Pi}_{RL}\right)_{k,0} & \dots & \left(\mathring{\Pi}_{RL}\right)_{k,k-1}  \end{matrix}  &  (P_1)_{k+N}(x)
 \end{array}\right)\ ,   \\
 (\tilde{P}_{2RL})_{k}(x)&= \Theta_*\left(\begin{array}{c|c}
 \left(\mathring{\Pi}_{RL}^{\top}\right)^{[k]}
 &  \begin{matrix} 1 \\ x \\ \vdots \\ x^{k-1} \end{matrix} \\
 \hline
 \begin{matrix}\left(\mathring{\Pi}_{RL}^{\top}\right)_{k,0} & \dots & \left(\mathring{\Pi}_{RL}^{\top}\right)_{k,k-1}  \end{matrix}  &  x^k
 \end{array}\right)\ , \\
 (\tilde{h}_{RL})_{k}&= -\Theta_*\left(\begin{array}{c|c}
 \mathring{\Pi}_{RL}^{[k]}
 &  \begin{matrix}\left(\mathring{\Pi}_{RL}\right)_{0,k} \\ \vdots \\ \left(\mathring{\Pi}_{RL}\right)_{k-1,k} \end{matrix} \\
 \hline
 \begin{matrix}\left(\mathring{\Pi}_{RL}\right)_{k,0} & \dots & \left(\mathring{\Pi}_{RL}\right)_{k,k-1}  \end{matrix}  &  \left(\mathring{\Pi}_{RL}\right)_{k,k}
 \end{array}\right)\ ,
\end{align*}
\noi and
\begin{align*}
 (\tilde{P}_{2LR})_{k}(x)&= \frac{1}{R(x)}\Theta_*\left(\begin{array}{c|c}
 \mathring{\Pi}_{LR}^{[k]}
 &  \begin{matrix}(P_2)_{0}(x) \\ \vdots \\ (P_2)_{k+N-1}(x) \end{matrix} \\
 \hline
 \begin{matrix}\left(\mathring{\Pi}_{LR}\right)_{k,0} & \dots & \left(\mathring{\Pi}_{LR}\right)_{k,k-1}  \end{matrix}  &  (P_2)_{k+N}(x)
 \end{array}\right)\ ,   \\
 (\tilde{P}_{1LR})_{k}(x)&= \Theta_*\left(\begin{array}{c|c}
 \left(\mathring{\Pi}_{LR}^{\top}\right)^{[k]}
 &  \begin{matrix} 1 \\ x \\ \vdots \\ x^{k-1} \end{matrix} \\
 \hline
 \begin{matrix}\left(\mathring{\Pi}_{LR}^{\top}\right)_{k,0} & \dots & \left(\mathring{\Pi}_{LR}^{\top}\right)_{k,k-1}  \end{matrix}  &  x^k
 \end{array}\right)\ , \\
 (\tilde{h}_{LR})_{k}&= -\Theta_*\left(\begin{array}{c|c}
 \mathring{\Pi}_{LR}^{[k]}
 &  \begin{matrix}\left(\mathring{\Pi}_{LR}\right)_{0,k} \\ \vdots \\ \left(\mathring{\Pi}_{LR}\right)_{k-1,k} \end{matrix} \\
 \hline
 \begin{matrix}\left(\mathring{\Pi}_{LR}\right)_{k,0} & \dots & \left(\mathring{\Pi}_{LR}\right)_{k,k-1}  \end{matrix}  &  \left(\mathring{\Pi}_{LR}\right)_{k,k}
 \end{array}\right)\ ,
\end{align*}
\noi where the $(N+M)\times (N+M)$ matrices $\mathring{\Pi}$ are defined by
 \begin{align*}
  \mathring{\Pi}_{RL}&:=
  \left[\Pi_r\begin{bmatrix}(P_1)_0 \\ \vdots \\ (P_1)_{N+M-1} \end{bmatrix},
  \left(
 \Pi_q\begin{bmatrix} (C_1)_{0} \\ \vdots \\ (C_1)_{N-1} \end{bmatrix}-
 \Pi_q\begin{bmatrix}(P_1)_{0} \\ \vdots \\ (P_1)_{N-1} \end{bmatrix}\Xi_R
 \right) \left(\mathbf{Q}^{[N]}\Pi_q[\chi^{[N]}] \right)^{-1}\right]\ , \\
 \mathring{\Pi}_{LR}&:=
  \left[\Pi_r\begin{bmatrix}(P_2)_0 \\ \vdots \\ (P_2)_{N+M-1} \end{bmatrix},
  \left(
 \Pi_q\begin{bmatrix} (C_2)_{0} \\ \vdots \\ (C_2)_{N-1} \end{bmatrix}-
 \Pi_q\begin{bmatrix}(P_2)_{0} \\ \vdots \\ (P_2)_{N-1} \end{bmatrix}\Xi_L
 \right) \left(\mathbf{Q}^{[N]}\Pi_q[\chi^{[N]}] \right)^{-1}\right]\ .
 \end{align*}
\end{pro}

\section{Deformations arising from the action of linear differential operators}
In this Sobolev context, where derivatives are ubiquitous, the polynomial deformation theory seems to be missing something. For that reason, in this
section we will now discuss a different, more general class of deformations, obtained when a differential operator acts on one of the entries of
the bilinear form. Although a general theory like the one for Darboux--Sobolev deformations is not available yet, some steps and results
in that direction, together with some easy examples, can be proposed. To address this question, let us start with the derivative operator. We have
\begin{align*}
 D G_{\W}&= D \boldsymbol{D} \left( \int_{\Omega} \chi(x) \W \chi(x)^\top  \right) \boldsymbol{D}^\top= \boldsymbol{D} \begin{pmatrix}
   0 & 0 & 0 & 0 &\dots \\
   \mathbb{I} & 0 & 0 & 0 &\dots \\
   0 & \mathbb{I} & 0 & 0 &\dots\\
   0 & 0 & \mathbb{I} & 0 &\dots\\
   0 & 0 & 0 & \mathbb{I} &\ddots\\
   \vdots & \vdots & \vdots &\vdots
\end{pmatrix}
\left( \int_{\Omega} \chi(x) \W \chi(x)^\top  \right) \boldsymbol{D}^\top=\\
         & \boldsymbol{D} \left( \int_{\Omega} \chi(x) \begin{pmatrix}
   0 & 0 & 0 & 0 &\dots \\
   1 & 0 & 0 & 0 &\dots \\
   0 & 1 & 0 & 0 &\dots\\
   0 & 0 & 1 & 0 &\dots\\
   0 & 0 & 0 & 1 &\ddots\\
   \vdots & \vdots & \vdots &\vdots
\end{pmatrix}
         \W \chi(x)^\top  \right) \boldsymbol{D}^\top \ .
\end{align*}
Consequently, we can obtain immediately the following result.
\begin{theorem}
The relations
 \begin{align*}
  (f',h;\W)&=(f,h;\Lambda^{\top} \W)     &      D G_{\W}&= G_{\Lambda^{\top} \W}\\
  (f,h';\W)&=(f,h; \W \Lambda )     &       G_{\W} D^{\top} &= G_{\W \Lambda}
 \end{align*}
hold.
\end{theorem}

By linearity, we deduce that given any linear differential operator
$\boldsymbol{L}:=\sum_{n,r=0}^{\infty} a_{n,r} x^{n}\frac{\d^r}{\d x^r}$, acting on one of the entries of our inner product,
we can translate its action into a matrix multiplying the initial moment matrix $L:=\sum_{n,r=0}^{\infty} a_{n,r} D^r \Lambda^{n}$ or into a matrix
multiplying the initial measure matrix
$\mathcal{L}=\sum_{n,r=0}^{\infty} a_{n,r}(\Lambda^\top)^{r}\mathcal{X}^n$.
%\footnote{Let us clarify that we are actually using three different names $\boldsymbol{L},L,\mathcal{L}$ for the same linear differential operator depending on the way we want to understand its action.}.

%\noi It is important to observe that, while the Sobolev bilinear function is defined on
%$(*,*;\W):\mathcal{A}^{\mathcal{N}}_{\W}\times \mathcal{A}^{\mathcal{N}}_{\W} \longrightarrow \R$
%the modified Sobolev bilinear function will be defined in a space which is the Cartesian
%product of two quotients of $V$ with the kernels of the operators $\boldsymbol{L}_1$ and $\boldsymbol{L}_2$ respectively, i.e.
%$(\boldsymbol{L}_1[\cdot],\boldsymbol{L}_2[\cdot];\W):\frac{V}{Ker(\boldsymbol{L}_1)}\times \frac{V}{Ker(\boldsymbol{L}_2)}\longrightarrow K$.
%SOMETHING TO DO WITH WEAK ORTHOGONALITY??

The interplay among the three different actions $\boldsymbol{L},L,\mathcal{L}$ is clarified in the next
\begin{pro}\label{pro3} We have
\begin{align*}
 (\boldsymbol{L}_1 [f], \boldsymbol{L}_2 [h];\W)&= (f,h: \mathcal{L}_1\W \mathcal{L}_2^{\top}), &
 L_1 G_{\W} \left(L_2 \right)^{\top}&=G_{\mathcal{L}_1\W\left(\mathcal{L}_2 \right)^{\top}} \ .
\end{align*}
\end{pro}
This is a direct generalization of Proposition \ref{pro2}. Provided both
$G_{\W}$ and $G_{\mathcal{L}_1\W\left(\mathcal{L}_2 \right)^{\top}}$ are $LU$-factorizable, this proposition could allow us, in some
particular cases, to relate the SBPS associated to each of the two moment matrices.
%\footnote{ While the Sobolev inner product
%takes certain vector space of functions $V \subset \A(\Omega)$, $V \times V$ to the field $K$,
%this is, $(\,\cdot\,,\,\cdot\,\,;\W):V\times V\longrightarrow K$ the modified
%Sobolev inner product
%$(\boldsymbol{L}_1[\cdot],\boldsymbol{L}_2[\cdot];\W):\frac{V}{Ker(\boldsymbol{L}_1)}\times \frac{V}{Ker(\boldsymbol{L}_2)}\longrightarrow K$}

A couple of interesting, nontrivial problems arise from the last discussion.
\begin{itemize}
 \item Determine a pair $(\boldsymbol{L}_{1}, \boldsymbol{L}_{2})$ of linear differential operators  with associated $(\mathcal{L}_{1}, \mathcal{L}_{2})$ such
that $\mathcal{L}_{1}\W \sim \W \mathcal{L}_{2}^{\top}$ (and therefore $L_1 G_{\W}=G_{\W}L_2^{\top}$).
\item Determine a pair of operators $(\boldsymbol{L}_1,\boldsymbol{L}_2)$ with associated $(\mathcal{L}_1,\mathcal{L}_2)$
such that $\mathcal{L}_1\W_1 \mathcal{L}_2^{\top} \sim \W_2$ and $\W_1,\W_2$ have some ``suitable'' properties.
\end{itemize}
An answer to the first problem would ensure that the associated SBPS possess many interesting properties. For instance, the special case where the usual three term recurrence relation
holds is just a particular answer to this question for $\boldsymbol{L}_1=\boldsymbol{L}_2=x$. Another example of this kind was given in
proposition \ref{canana} with the operator $\mathbf{F}$.\\
We will devote the next section to a partial answer to the second problem.

\subsection{Orthogonal polynomials with respect to differential operators}
For the second problem some simple cases can be tackled.
The idea behind it is to start with a simple measure matrix $\W_1$
and deform it by means of differential operators
into a new one $\W_2\sim \mathcal{L}_1 \W_1 (\mathcal{L}_2)^{\top}$ so that we can establish explicit relations between
$G_{\W_1}$ and $G_{\W_2}$. If both moment matrices are $LU$-factorizable, they may lead to relations between their associated SBPS.
For example, one can start with the standard (non Sobolev) matrix $\W_1=E_{00}\omega$. This case deserves special attention since
it connects usual moment matrices with certain Sobolev moment matrices in a direct way. This entails the possibility to relate the
associated OPS and SBPS as well. This section is intimately related to the notion of orthogonality with respect to a differential
operator (OPDO) \cite{OPDO}. Here we start from the standard orthogonality, in order to obtain connections between standard
and Sobolev polynomials. A similar approach could be used in the more general case of a diagonal matrix $\W$.
In that case, we would be able to relate Sobolev orthogonal polynomials associated to different measure matrices.

\begin{pro}{\label{UU}}
 Given two linear differential operators $\boldsymbol{L}_{\alpha}:=\sum_k p_{\alpha,k}(x)\frac{\d^k}{\d x^k}$, $\alpha=1,2$, with
$p_{\alpha,k}(x)$ polynomials of any degree for all $k$, the following relation between the standard inner product involving
these differential operators and a Sobolev bilinear function exists
\begin{align*}
 \langle \boldsymbol{L}_1[f],\boldsymbol{L}_2[h]\rangle_{\mu}&=\left(f,h;\W_{\boldsymbol{L}_{1,2}} \right) \ .
\end{align*}
The relation between the associated Sobolev moment matrix and the standard one reads
\begin{align*}
L_1 g_{\mu} (L_2)^{\top}&=G_{\W_{\boldsymbol{L}_{1,2}}},
\end{align*}
and the measure matrix is
\begin{align*}
 \W_{\boldsymbol{L}_{1,2}}=\begin{pmatrix}
  p_{1,0}p_{2,0} & p_{1,0}p_{2,1} & p_{1,0}p_{2,2} & \dots\\
  p_{1,1}p_{2,0} & p_{1,1}p_{2,1} & p_{1,1}p_{2,2} & \dots\\
  p_{1,2}p_{2,0} & p_{1,2}p_{2,1} & p_{1,2}p_{2,2} & \dots\\
  \vdots               & \vdots               & \vdots               &
 \end{pmatrix}\d \mu(x) \ .
\end{align*}

\end{pro}

\begin{proof}
Since $g_{\mu}$ is the usual moment matrix associated to the measure $\d \mu(x)$ we have
\begin{align*}
 \langle \boldsymbol{L}_2[f],\boldsymbol{L}_1[h]\rangle_{\mu}&=
 \left(f,h;[\mathcal{L}_1 E_{0,0}(\mathcal{L}_2 )^{\top}\d \mu]\right)
  &
 L_1 g (L_2)^{\top}&=G_{[\mathcal{L}_1 E_{0,0}(\mathcal{L}_2 E_{0,0})^{\top}\d \mu]} \ .
\end{align*}

\noi Note that the shape of $[\mathcal{L}_1 E_{0,0}(\mathcal{L}_2 E_{0,0})^{\top}\d \mu]$ is particularly simple: it is quite straightforward to see that
\begin{align*}
 [\mathcal{L}_1 E_{0,0}(\mathcal{L}_2 E_{0,0})^{\top}\d \mu]&=
 \begin{pmatrix}
  p_{1,0}(x) \\
  p_{1,1}(x) \\
  p_{1,2}(x) \\
  \vdots
 \end{pmatrix} \cdot \begin{pmatrix} p_{2,0}(x) & p_{2,1}(x) & p_{2,2}(x) & \dots \end{pmatrix} \d \mu(x)=
 \begin{pmatrix}
  p_{1,0}p_{2,0} & p_{1,0}p_{2,1} & p_{1,0}p_{2,2} & \dots\\
  p_{1,1}p_{2,0} & p_{1,1}p_{2,1} & p_{1,1}p_{2,2} & \dots\\
  p_{1,2}p_{2,0} & p_{1,2}p_{2,1} & p_{1,2}p_{2,2} & \dots\\
  \vdots               & \vdots               & \vdots               &
 \end{pmatrix}\d \mu(x) \ .
\end{align*}
\end{proof}
%**************************************************************\\
%*****************************************************************\\
%***********************************************************************\\
\begin{definition}
Given two families of linear differential operators $S=\{\{\boldsymbol{L}_k\},\{\boldsymbol{U}_k\} \}_{k=0}^{\mathcal{N}}$ with
\begin{align*}
 \boldsymbol{L}_k&=\frac{\d^{k}}{\d x^k}+\sum_{j=k+1}l_{jk}(x)\frac{\d^{j}}{\d x^j},  &
 \boldsymbol{U}_k&=\frac{\d^{k}}{\d x^k}+\sum_{j=k+1}u_{kj}(x)\frac{\d^{j}}{\d x^j}
\end{align*}
and a set of measures $\{\d \mu_k(x)\}_{k=0}^{\mathcal{N}}$, we shall call the function
\begin{align*}
 (f,h)_{S}:=\sum_{k=0}^{\mathcal{N}}  \langle \boldsymbol{L}_k[f],\boldsymbol{U}_k[h] \rangle_{\mu_k}
\end{align*}
the generalized diagonal Sobolev bilinear function.
\end{definition}
Shall we had $l_{jk}(x)=0=u_{kj}(x)$ $\forall k,j$ the generalized diagonal Sobolev bilinear function would be
indeed the usual diagonal Sobolev bilinear function.\\

\begin{pro}
Given a $(\mathcal{N}+1) \times (\mathcal{N}+1)$ measure matrix satisfying $\det \W^{[k]}(x)\neq 0$ $\forall x \in \Omega$
and $k=0,1\dots,\mathcal{N}$, then the Sobolev bilinear function $(f,h;\W)$
is equivalent to a generalized diagonal Sobolev bilinear function $(f,h)_S$. The pair
$S=\{\{\boldsymbol{L}_k\},\{\boldsymbol{U}_k\} \}_{k=0}^{\mathcal{N}}$ with
\begin{align*}
 \boldsymbol{L}_k&=\frac{\d^{k}}{\d x^k}+\sum_{j=k+1}l_{jk}(x)\frac{\d^{j}}{\d x^j}  &
 \boldsymbol{U}_k&=\frac{\d^{k}}{\d x^k}+\sum_{j=k+1}u_{kj}(x)\frac{\d^{j}}{\d x^j}
\end{align*}
is determined by the LU factorization of $\W$ by means of the relations
\begin{align*}
 \W(x)&=\begin{pmatrix}
        1         &           &     &        &      &    \\
        l_{10}(x) & 1         &     &        &      &    \\
        l_{20}(x) & l_{21}(x) &  1  &        &      &      \\
        \vdots    &   \vdots  &     & \ddots &      &    \\
                  &           &     &        &      &     \\
        l_{\mathcal{N}0}(x) & l_{\mathcal{N}1}(x) &     &        &      &   1
       \end{pmatrix}
 \begin{pmatrix}
      \d \mu_0(x) &           &      &        &      &    \\
                  &\d \mu_1(x)&      &        &      &    \\
                  &           &\ddots&        &      &      \\
                  &           &      & \ddots &      &    \\
                  &           &      &        &      &     \\
                  &           &      &        &      & \d \mu_{\mathcal{N}}(x)
       \end{pmatrix} \cdot \\
&\cdot  \begin{pmatrix}
      1           & u_{01}(x) &u_{02}(x)& \dots  &      &u_{0\mathcal{N}}(x)    \\
                  &    1      &u_{12}(x)& \dots  &      &u_{1\mathcal{N}}(x)   \\
                  &           &    1    &        &      &      \\
                  &           &         & \ddots &      &    \\
                  &           &         &        &      &u_{\mathcal{N}-1\mathcal{N}}(x)     \\
                  &           &         &        &      & 1
       \end{pmatrix}
\end{align*}
In addition, if each $\d \mu_{k}(x)$ is positive definite and $l_{j,k}(x),u_{k,j}(x)$ are polynomials satisfying the relations
\[
j-\deg[u_{k,j}(x)]>k \qquad \text{and} \qquad j-\deg[l_{j,k}(x)]>k,
\]
then $G_{\W}$ is $LU$-factorizable and therefore has an associated SBPS.
\end{pro}
\begin{proof}
The first part of the proposition is an easy generalization of Proposition \ref{UU}, since the LU factorization of $\W$ can be understood
as follows
\begin{align*}
\W=\left[\mathcal{L}_0 E_{0,0} (\mathcal{U}_0 E_{0,0})^{\top} \omega_0\right]+
                     \left[\mathcal{L}_1 E_{0,0} (\mathcal{U}_1 E_{0,0})^{\top} \omega_1\right]+
                     \dots +
                     \left[\mathcal{L}_{\mathcal{N}} E_{0,0} (\mathcal{U}_{\mathcal{N}} E_{0,0})^{\top} \omega_{\mathcal{N}}\right] \ .
\end{align*}
Therefore, we have that $(f,h;\W)=\sum_{k=0}^{\mathcal{N}}  \langle \boldsymbol{L}_k[f],\boldsymbol{U}_k[h] \rangle_{\mu_k}$ or equivalently
$G_{\W}=\sum_{k=0}^{N} L_k g_{\mu_k} (U_k)^{\top}$. This expression, together with the fact that
the condition on the degrees of $u_{k,j}(x)$ and $l_{k,j}(x)$ is equivalent to requiring that $L_k$ and $U_k$ have the shape of
$D^{k}+diagonals\,\,beneath\,\,this\,\,one$ (also equivalent to $\chi^{[k]}\in ker \boldsymbol{U}_k, \chi^{[k]} \in ker \boldsymbol{L}_k$),
make the reasoning of the positive definiteness of $G_{\W}$ exactly the same as the one
we used for the positive definite diagonal case.
\end{proof}

\subsection{Examples where SBPS and OPS can be related in terms of differential operators}
Let us show in more detail some examples where the relation between OPS and SBPS can
be explicitly constructed.
Assume that $\boldsymbol{L}_{\alpha}$ satisfy the two conditions
\begin{itemize}
 \item $\deg [p_{\alpha,k}\leq k]$, $\forall k$. This implies that $L_{\alpha} \in \mathscr L$.
 \item both $\boldsymbol{L}_\alpha$ are invertible operators. %Therefore the simbol $L_{\alpha}^{-1}$ makes sense.
\end{itemize}
For these cases the LU factorization of $L_1 g (L_2)^{\top}$ is trivial. If $g=S^{-1}h \left(S^{-1}\right)^{\top}$ it is easy to see that
\begin{align*}
L_1 g (L_2)^{\top}= [S (L_1)^{-1}]^{-1} h \left([S (L_2)^{-1}]^{-1}\right)^{\top} \ .
\end{align*}
This means that we can write the SBPS from the OPS. Indeed,
\begin{align*}
 P_1(x)&=S L_1^{-1} \chi(x), &  P_2(x)&= S (L_2)^{-1} \chi(x) \ .
\end{align*}
Let us discuss a couple of examples of this kind.
\begin{enumerate}
\item Consider a $\W$ of the form
\begin{align*}
 \W(x):=\begin{pmatrix}
   1 & -1 & 0 & 0 & \dots\\
   -1 & 1 & 0 & 0 &\dots\\
   0 & 0 & 0 & 0 &\dots\\
   0 & 0 & 0 & 0 &\dots\\
  \vdots&\vdots& \vdots& \vdots & \ddots
\end{pmatrix}\d \mu(x) \ .
\end{align*}
\noi This measure matrix comes from the operator $\boldsymbol{L}_\alpha=1-\frac{\d}{\d x}$, which of course satisfies the two conditions above. The related moment matrix reads
\begin{align*}
 G_{\W}&=
   (\mathbb{I}-D)g(\mathbb{I}-D)^{\top}
& \text{where}
 (\mathbb{I}-D)^{-1}&=\sum_{n=0}^{\infty}D^{n} \ .
\end{align*}
Thus,
\begin{align*}
 G_{\W}&=\left[S(\mathbb{I}-D)^{-1}\right]^{-1}H\left(\left[S(\mathbb{I}-D)^{-1}\right]^{-1}\right)^{\top} \ .
 \end{align*}
We conclude that the SOPS associated with $G_{\W}$ is related to the OPS associated to $\omega$ as follows
 \begin{align*}
 P(x)&=S(\mathbb{I}-D)^{-1}\chi(x)=S (\sum_{n=0}^{\infty}D^{n}) \chi(x)=
 \begin{pmatrix}
   1 & 0 & 0 & 0 & \dots\\
   S_{1,0} & 1 & 0 & 0 &\dots\\
   S_{2,0} & S_{2,1} & 1 & 0&\dots\\
   S_{3,0} & S_{3,1} & S_{3,2} & 1&\dots\\
  \vdots&\vdots& \vdots& \vdots & \ddots
\end{pmatrix}\begin{pmatrix}1 \\ x+1 \\x^2+2x+2 \\x^3+3x^2+6x+6\\ \vdots \end{pmatrix}.
\end{align*}

\item %This second example is, in some way, the opposite to the first one.
We start with a $\W$ of the form
\begin{align*}
 \W(x):=\begin{pmatrix}
   1 & 1 & 1 & 1 & \dots\\
   1 & 1 & 1 & 1 &\dots\\
   1 & 1 & 1 & 1 &\dots\\
   1 & 1 & 1 & 1 &\dots\\
  \vdots&\vdots& \vdots& \vdots & \ddots
\end{pmatrix}\d \mu(x) \ .
\end{align*}
\noi It is not hard to see that the two previous conditions are fulfilled. This allows us to write explicitly
\begin{align*}
 G_{\W}&=\sum_{k=0}^{\infty}D^{k} S^{-1}H (S^{-1})^{\top} (\sum_{k=0}^{\infty}D^{k})^{\top}=
 [S(\mathbb{I}-D)]^{-1}H \left[[S(\mathbb{I}-D)]^{-1}\right]^{\top} \ .
 \end{align*}
Thus, the associated OPS is nothing but
\begin{align*}
 P(x)=S(\mathbb{I}-D)\chi(x)=\begin{pmatrix}
   1 & 0 & 0 & 0 & \dots\\
   S_{1,0} & 1 & 0 & 0 &\dots\\
   S_{2,0} & S_{2,1} & 1 & 0&\dots\\
   S_{3,0} & S_{3,1} & S_{3,2} & 1&\dots\\
  \vdots&\vdots& \vdots& \vdots & \ddots
\end{pmatrix}\begin{pmatrix}1 \\ x-1 \\x^2-2x \\x^3-3x^2\\ \vdots \end{pmatrix} \ .
\end{align*}
\item Now we shall consider a matrix measure of the kind
\begin{align*}
 \W(x):=\begin{pmatrix}
   \frac{a^0}{0!0!} & \frac{a^1}{0!1!} & \frac{a^2}{0!2!} & \frac{a^3}{0!3!} & \dots\\
   \frac{a^1}{1!0!} & \frac{a^2}{1!1!} & \frac{a^3}{1!2!} & \frac{a^4}{1!3!} &\dots\\
   \frac{a^2}{2!0!} & \frac{a^3}{2!1!} & \frac{a^4}{2!2!} & \frac{a^5}{2!3!} &\dots\\
   \frac{a^3}{3!0!} & \frac{a^4}{3!1!} & \frac{a^5}{3!2!} & \frac{a^6}{3!3!} &\dots\\
   \vdots           &            \vdots&            \vdots&           \vdots & \ddots
\end{pmatrix}\d \mu(x) \ .
\end{align*}
Remarkably, $\W(x)\in \W_x$. Its expression corresponds to the one in eq. \eqref{Wx} by choosing $\d \mu_k=\frac{a^k\d \mu}{k!}$.
The previous theory allows us to write
\begin{align*}
 G_{\W}&=\sum_{k=0}^{\infty}\frac{a^kD^{k}}{k!} g \left(\sum_{r=0}^{\infty}\frac{a^rD^{r}}{r!}\right)^{\top}=
 \exp\{aD\}g \exp\{aD^{\top}\}=\left[S \exp\{-aD\} \right]^{-1} H \left[[S \exp\{-aD\}]^{-1} \right]^{\top} \ .
 \end{align*}
 This expression implies that the associated SOPS is nothing but the usual one OPS associated with $\omega$ but with an shift by $a$ in the independent variable, i.e.
 \begin{align*}
  P&=S \exp\{-aD\} \chi(x)= S \chi(y),  &     y&=(x-a) \ .
 \end{align*}
Let us mention here that when $a=1$, this example establishes a connection between ``Hankel transforms'' (as defined in \cite{Layman}) and Sobolev Polynomials to light.
One can show that the matrices that act to the left and right of the initial sequence (the initial moment matrix $g$) are
$\left(\frac{D^k}{k!}\right)_{l,j}={l \choose j}$. In other words, we recover the so called ``Binomial transform'' of the
initial sequence, under which the Hankel transform remains invariant.\\

\end{enumerate}

\appendix

\section{A relation with integrable hierarchies of Toda type}

The purpose of this final section is to clarify the connection of the present theory of Sobolev bi-orthogonal polynomials with the theory of integrable systems.

As usual in this context, one can start from a suitable deformation of the moment matrix with certain appropriate matrices. These matrices involve
the exponential of a linear combination of two set of times and the powers of
the matrices $\Lambda$. Inspired by this approach, we shall generalize to our framework some well-known results. \\
To this aim, let us introduce two different sets of real deformation parameters $t_a=\{t_{a,0}=0,t_{a,1},t_{a,2},\dots\}$ for $a=1,2$, which will allow us
to deform the moment matrix according to the following prescription.
\begin{definition}
We define the time-deformed moment matrix
\begin{align}\label{G(t)}
 G_{\W}(t)&=W_{1,0}(t_1) G_{\W} [W_{2,0}(t_2)]^{-1}
\end{align}
where the deformation matrices $W_{1,0}(t_1)$ and $W_{1,0}(t_2)$ are given by
\begin{align*}
 W_{1,0}(t_1)&=\exp \left( \sum_{j=0}^{\infty}t_{1,j} \Lambda^j \right) &
 W_{2,0}(t_2)&=\exp \left( \sum_{j=0}^{\infty}t_{2,j} \left(\Lambda^{\top}\right)^j \right)
\end{align*}
\end{definition}
As the following result shows, the reason for this deformation of the moment matrix is that it can be directly translated into a deformation of the corresponding measure matrix.

\begin{theorem}
 The deformed moment matrix $ G_{\W}(t)$ can be written as the moment matrix associated to a time dependent measure matrix, this is
 \begin{align*}
  G_{\W}(t)=G_{\W(t)}
 \end{align*}
where the new time dependent measure matrix is given by the following expression
 \begin{align*}
 \W(t):= \left[ \W_{1,0}(t_1,x)\right] \W \left[ \W_{2,0}({t_2,x}) \right]^{-1}=
 \left[ \exp \left( \sum_{j=0}^{\infty}t_{1,j} \mathcal{X}^j \right)\right] \W
 \left[ \exp \left( -\sum_{j=0}^{\infty}t_{2,j} \left(\mathcal{X}^{\top}\right)^j \right)\right]\ .
 \end{align*}

 \end{theorem}
 %Where $\W_{i}(t_i)$ is an upper triangular semi infinte matrix whose entries are
%\begin{align*}
% (\W_{i}(t_i))_{n,n+j}&=\begin{pmatrix}
%                         n+j\\
%                         n
%                        \end{pmatrix} t_i^j
%\end{align*}
%
%
%\end{theorem}
%(Hadamards lemma could help)
\noi It is worth pointing out that $\W_{1,0}(t_1,x)$ is upper triangular while $\W_{2,0}({t_2,x})$ is lower triangular. As an example
\begin{align*}
 \exp(t \mathcal{X})=
 \begin{pmatrix}
  \begin{pmatrix}0\\0\end{pmatrix}t^{0} & \begin{pmatrix}1\\0\end{pmatrix}t^{1}   & \begin{pmatrix}2\\0\end{pmatrix}t^{2}   & \begin{pmatrix}3\\0\end{pmatrix}t^{3}   & \dots \\
                                        & \begin{pmatrix}1\\1\end{pmatrix}t^{1-1} & \begin{pmatrix}2\\1\end{pmatrix}t^{2-1} & \begin{pmatrix}3\\1\end{pmatrix}t^{3-1} &\dots \\
                                        &                                         & \begin{pmatrix}2\\2\end{pmatrix}t^{2-2} & \begin{pmatrix}3\\2\end{pmatrix}t^{3-2} &\dots \\
                                        &                                         &                                         & \begin{pmatrix}3\\3\end{pmatrix}t^{3-3} &\dots \\
                                        &                                         &                                         &                                         &\ddots
 \end{pmatrix}\exp(tx)\ .
\end{align*}

\noi Once the moment matrix is deformed, in case we can still $LU$-factorize it we can write
\begin{align}\label{G(t)LU}
 G_{\W}(t)&=S_1(t)\left(S_{2}(t)\right)^{-1}\ ,
\end{align}
which leads to the time dependent Sobolev orthogonal polynomial sequences. This factorization also is the key for the following
\begin{definition}
The wave semi-infinite matrices are
\begin{align*}
 W_1(t)&:=S_1(t) W_{1,0}(t_1) & W_2(t)&:=S_2(t) W_{2,0}(t_2)\ .
\end{align*}
\end{definition}
These are indeed related to the initial moment matrix.
\begin{pro} \label{GWW} The following relation hold
 \begin{align*}
  G_{\W}=\left(W_1(t)\right)^{-1} W_2(t)
 \end{align*}
\end{pro}
\begin{proof}From eqs. \eqref{G(t)} and \eqref{G(t)LU} we can see that
\begin{align}\label{GW}
 G_{\W}=\left(W_{1,0}(t_1)\right)^{-1}\left(S_1(t)\right)^{-1} S_2(t) W_{2,0}(t_2)=\left(W_1(t)\right)^{-1}W_2(t) \ .
\end{align}
\end{proof}
We shall introduce two operators that will be relevant hereon.
\begin{definition} The Lax operators associated with our moment matrix are
\begin{align*}
 L_1&:=S_1 \Lambda S_1^{-1} &
 L_2&:=S_2 \Lambda^{\top} S_2^{-1} \ .
\end{align*}
\end{definition}
It is important to remark here that in contrast with what happens in the standard theory of deformation of moment matrices, where $L_1=L_2$ (because both coincide
with the tri-diagonal Jacobi matrix responsible for the usual three term recurrence relation), this is no longer the case
in the Sobolev context.  Indeed, $\Lambda G_{\W}\neq G_{\W} \Lambda^{\top}$. Thus $L_1\neq L_2$ and we can only infer that $L_1$ is a
lower triangular matrix with an extra diagonal over the main one, while $L_2$ is an upper triangular matrix with an extra diagonal
beneath the main one.

\begin{pro}
 For $a=1,2$ we have the following differential equations for the wave semi infinite matrices
\begin{align*}
 \frac{\partial W_a}{\partial t_{1,j}} W_{a}^{-1}&=(L_1^j)_{+}  &
 \frac{\partial W_a}{\partial t_{2,j}} W_{a}^{-1}&=(L_2^j)_{-} \ .
\end{align*}
\end{pro}
Here $(A)_{-}$ is the projection of the matrix $A$ onto the space of strictly lower triangular matrices while $(A)_{+}$ is its
projection onto the space of upper triangular matrices.
\begin{proof}
 Deriving eq. \eqref{GW}, on one hand we can obtain that
 \begin{align*}
  \frac{\partial W_1}{\partial t_{a,j}} W_1^{-1}&=\frac{\partial W_2}{\partial t_{a,j}} W_2^{-1} &
  a&=1,2; \,\,\,\,\, j=1,2,3,\dots \ .
 \end{align*}
On the other hand,
\begin{align*}
 \frac{\partial S_1}{\partial t_{1,j}} S_1^{-1}+S_1\Lambda^j S_1^{-1}&=\frac{\partial S_2}{\partial t_{1,j}} S_2^{-1} &
 \frac{\partial S_2}{\partial t_{2,j}} S_2^{-1}+S_2\left(\Lambda^{\top}\right)^{j} S_2^{-1}&=\frac{\partial S_1}{\partial t_{2,j}} S_1^{-1} \ .
\end{align*}
Decomposing them in their upper and strictly lower projections leads to the result of the proposition.
\end{proof}
The results of these proof can also be used to prove the next interesting result.
\begin{pro}
 The following Lax equations hold
 \begin{align*}
  \frac{\partial L_a^j}{\partial t_{b,r}}=\left[(L_b^j)_{(-1)^{b+1}},L_a^j \right]
 \end{align*}
or explicitly
\begin{align*}
 \frac{\partial L_1^j}{\partial t_{1,r}}&=\left[(L_1^j)_{+},L_1^j \right]   &
 \frac{\partial L_1^j}{\partial t_{2,r}}&=\left[(L_2^j)_{-},L_1^j \right]  \\
 \frac{\partial L_2^j}{\partial t_{1,r}}&=\left[(L_1^j)_{+},L_2^j \right]   &
 \frac{\partial L_2^j}{\partial t_{2,r}}&=\left[(L_2^j)_{-},L_2^j \right]
\end{align*}
\end{pro}
\noi The compatibility equations of these give rise to the classical Zakharov--Shabat equations.
%\begin{pro}
% The Zakharov--Shabat equations hold
%\end{pro}

\begin{pro}
  Wave functions evaluated at different times $t$ and $t'$ satisfy the relation
\begin{align*}
  W_1(t)W_1(t')^{-1}=W_2(t)W_2(t')^{-1}.
  \end{align*}
\end{pro}
\begin{proof}
 From Proposition \ref{GWW} we  derive the equality
  \begin{align*}
    (W_1(t))^{-1}W_2(t)=G=(W_1(t'))^{-1}W_2(t')\ ,
  \end{align*}
from which the result follows immediately.
\end{proof}

\textbf{Acknowledgments}.

The research of G.\,A. and M.\,M. has been supported by the research project [MTM2015-65888-C4-3-P]
\emph{``Ortogonalidad, teor\'ia de la aproximaci\'on y aplicaciones en física matem\'atica"'} MINECO, Spain.
G.\,A. also thanks the Program \emph{``Ayudas para Becas y Contratos Complutenses
Predoctorales en España"' 2011}, Universidad Complutense de Madrid, Spain.\\
The research of P.\,T. has been partly supported by the research project FIS2015-63966, MINECO, Spain, and by the ICMAT Severo
Ochoa project SEV-2015-0554 (MINECO).

\end{document}